\documentclass[10pt,reqno]{amsart}
\usepackage{amsmath,amsfonts,amsthm,amscd,amssymb}
\usepackage[utf8]{inputenc}
\usepackage{amsbsy}
\usepackage{graphicx}
\usepackage{hyperref}
\usepackage[margin=1in]{geometry}
\usepackage{subcaption}
\usepackage{enumitem}
\usepackage{float}
\usepackage{mathtools}
\usepackage{epstopdf}
\usepackage{stmaryrd}

\DeclareGraphicsExtensions{.eps}
  
\usepackage{color}
\definecolor{myblue}{rgb}{.8, .8, 1}

\usepackage{amsmath}
\usepackage{empheq}

\newlength\mytemplen
\newsavebox\mytempbox

\makeatletter
\newcommand\mybluebox{%
    \@ifnextchar[%]
       {\@mybluebox}%
       {\@mybluebox[0pt]}}

\def\@mybluebox[#1]{%
    \@ifnextchar[%]
       {\@@mybluebox[#1]}%
       {\@@mybluebox[#1][0pt]}}

\def\@@mybluebox[#1][#2]#3{
    \sbox\mytempbox{#3}%
    \mytemplen\ht\mytempbox
    \advance\mytemplen #1\relax
    \ht\mytempbox\mytemplen
    \mytemplen\dp\mytempbox
    \advance\mytemplen #2\relax
    \dp\mytempbox\mytemplen
    \colorbox{myblue}{\hspace{1em}\usebox{\mytempbox}\hspace{1em}}}

\makeatother

\usepackage{tikz}
\usetikzlibrary{decorations.shapes}
\tikzset{paint/.style={ draw=#1!50!black, fill=#1!50 },
    decorate with/.style=
    {decorate,decoration={shape backgrounds,shape=#1,shape size=2mm}}}
\usepackage{todonotes}

\definecolor{skyblue}{rgb}{0.85,0.85,1}

\newcommand{\Mod}[1]{\ (\mathrm{mod}\ #1)}
\usepackage{accents}

\newtheorem{theorem}{Theorem}
\newtheorem{prop}[theorem]{Proposition}

\newtheorem{define}[theorem]{Definition}
\newtheorem{lemma}[theorem]{Lemma}
\newtheorem{rem}[theorem]{Remark}

\newcommand{\p}{\partial}				% partial derivative

\DeclareMathOperator{\sgn}{sgn}

\newcommand{\eqdef}{\vcentcolon =}

\numberwithin{equation}{section}
\numberwithin{theorem}{section}

\begin{document}
\title{Nodal Sets of Laplacian Eigenfunctions with an Eigenvalue of Multiplicity 2}
%\title{}
\author{Andrew Lyons}\email{ahlyons@email.unc.edu}\address{Dept. of Mathematics, UNC-CH, 418 Phillips Hall, Chapel Hill, NC 27599-3250, USA}

\maketitle
\begin{abstract}
We study the effects of a domain deformation to the nodal set of Laplacian eigenfunctions when the eigenvalue is degenerate. In particular, we study deformations of a rectangle that perturb one side and how they change the nodal sets corresponding to an eigenvalue of multiplicity $2$. We establish geometric properties, such as number of nodal domains, presence of crossings, and boundary intersections, of nodal sets for a large class of boundary deformations and study how these properties change along each eigenvalue branch for small perturbations. We show that internal crossings of the nodal set break under generic deformations and obtain estimates on the location and regularity of the nodal sets on the perturbed rectangle.

\medskip

\noindent\textsc{R\'esum\'e.}
Nous étudions les effets de la déformation d'un domaine sur l'ensemble nodal des fonctions propres du Laplacien lorsque la valeur propre est dégénérée. En particulier, nous étudions les déformations d'un rectangle qui perturbent un côté et la façon dont elles modifient les ensembles nodaux correspondants à une valeur propre de multiplicité deux. Pour une grande classe de déformations de la frontière, nous établissons des propriétés géométriques des ensembles nodaux, telles que le nombre de domaines nodaux, la présence de croisements et d'intersections avec la frontière, et nous étudions comment ces propriétés changent le long de chaque branche de valeur propre pour des petites perturbations. Nous montrons que les croisements internes de l'ensemble nodal se brisent  sous l'effet de déformations génériques et nous obtenons des estimations sur la localisation et la régularité des ensembles nodaux sur le rectangle perturbé.
\end{abstract}

%\tableofcontents
%%%%%%

\section{Introduction} 

\quad We study the behavior of the zero set of Laplacian eigenfunctions corresponding to an eigenvalue of multiplicity $2$. The eigenfunctions are defined on perturbations of a rectangle with Dirichlet boundary conditions. For a compact domain $\Omega\subset\mathbb{R}^d$, we consider an eigenfunction $u$ satisfying
\begin{equation}\label{eq:PDE}
    \begin{cases} (\Delta+\lambda)u=0, & \textrm{in } \Omega \\ \: u=0, & \textrm{on } \p \Omega \end{cases}
\end{equation}
with eigenvalue $\lambda$. The \textit{nodal set} of $u$ is then defined as the closure of the level set $\Gamma(u)=\{\textbf{x}\in \Omega^{\mathrm{o}}: u(\textbf{x})=0\}$ where $\Omega^o$ denotes the interior of $\Omega$. A nodal set partitions the domain into connected regions over which $u$ is nonzero; the connected components of $\Omega^{\mathrm{o}}\backslash\Gamma(u)$ are called \textit{nodal domains}. In a physical medium, the nodal set of an eigenfunction describes points that are stable under vibration. In the late $18^{\textrm{th}}$ century, Ernst Chladni cataloged a large number of nodal configurations by placing sand over a metal plate that vibrated at a fundamental frequency \cite{BH15, C87}. The sand would naturally gravitate to the points that do not vibrate, thus revealing the nodal set for a given frequency. Aside from this demonstration, Laplacian eigenfunctions model several physical phenomena, and understanding the nodal sets they produce is an active area of research \cite{Z17}. In the context of quantum mechanics, the square of a real solution to (\ref{eq:PDE}) satisfying $\left|\left|u\right|\right|_{L^2(\Omega)}=1$ is the probability density of a free quantum particle at energy $\lambda$, and the nodal set describes where such a particle is least likely to be found.  For low-energy eigenfunctions, nodal sets also play a role in spectral partitioning methods \cite{BCCM, HMW, BH10, AS09}, and much work has been done to understand their behavior \cite{CH08, GJ96, GJ98, GJ09}. In particular, this work furthers a recent effort to quantify how oscillations of low-energy eigenfunctions predict geometric properties of the underlying domain, as in \cite{BCM, BCM2, BGM}.

\quad In dimension $1$, solutions to (\ref{eq:PDE}) are known explicitly. In particular, the $j$-th eigenfunction, labeled according to its eigenvalue's position in the spectrum, equipartitions the interval into exactly $j$ nodal domains. This property is independent of the size of the interval but cannot be extended to higher dimensions. For $d\geq 2$, the number of nodal domains produced by an eigenfunction at energy $\lambda$ depends on the shape of the domain, so a natural question to ask is how domain deformations cause changes in a nodal set. A lot of the work to answer variations of this question was pioneered by Grieser and Jerison; one of their first results in this direction considered the first nodal line, corresponding to the lowest-energy eigenfunction with a nonempty nodal set, on bounded, convex domains in $\mathbb{R}^2$ \cite{GJ96}. On most domains, Laplacian eigenfunctions cannot be explicitly computed, so Grieser and Jerison relied on a partial Fourier series representation of the eigenfunction to provide estimates for the location and size of the nodal set with respect to the domain's eccentricity. In \cite{BCM}, Beck, Canzani, and Marzuola modified this technique to study the first nodal line on perturbations of a rectangle. They were able to provide estimates for both the location and regularity of the nodal line with respect to the size of the perturbation. Both sets of authors found that the number of nodal domains was invariant under their respective domain deformations, although this is not necessarily the case for higher-energy eigenvalues.

\quad In \cite{BGM}, Beck, Gupta, and Marzuola considered the lowest-energy eigenfunction on a rectangle to feature a crossing in its nodal set and studied the crossing under perturbations to the domain. Through the work of Uhlenbeck in \cite{U76}, it was known that interior crossings of nodal sets are unstable under small domain deformations, but Beck, Gupta, and Marzuola were able to quantify this instability and precisely describe the nodal set behavior. In particular, they began with an eigenfunction whose nodal set on the rectangle features two perpendicular lines that bisect each boundary component and cross in the center. Building upon methods in \cite{BCM, GJ09}, they detailed how the crossing vanishes, causing the number of nodal domains to decrease. However, their work required a non-resonant assumption that the eigenvalue of interest was simple and sufficiently far from the remainder of the spectrum. This assumption is either present or guaranteed in \cite{BCM, GJ96, GJ98, GJ09}, but to the author's knowledge, the case when an eigenvalue degeneracy is present has not yet been studied.

\quad To impose such a degeneracy, we work on a rectangle that enforces a multiplicity $2$ eigenvalue. Consistent with \cite{BGM}, we show that if the nodal set features a crossing in the rectangle, it must vanish for a large class of boundary perturbations, presented in Definition \ref{def:phi}. Otherwise, small domain deformations do not change geometric properties of the nodal set. In order to state our results more precisely, we describe the setup in detail. Throughout this work, we denote $$R(N)=[0,N]\times[0,1]$$ for $N>1$ and define $\Omega_\phi\left(\eta, N\right)$, pictured in Figure \ref{fig:domainDef}, as a perturbation of $R(N)$ when $\eta>0$. Precisely, 
\begin{equation}
    \Omega_\phi\left(\eta, N\right)=\left\{(x,y)\in\mathbb{R}^2: y\in[0,1], -\eta\phi(y)\leq x\leq N\right\} \nonumber
\end{equation}
where $\phi$ satisfies the requirements outlined in Definition \ref{def:phi} below.

\begin{figure}
    \centering
    \includegraphics[scale=0.5]{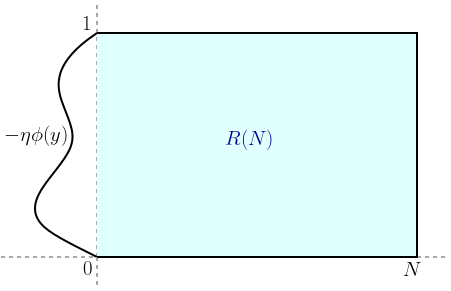}
    \caption{The Domain $\Omega_\phi(\eta, N)$.}
    \label{fig:domainDef}
    \vspace{12pt}
\end{figure}

\begin{define}\label{def:phi}
    A function $\phi$ is said to be admissible if $\phi\in C^5\left([0,1],[0,1]\right)$ has fixed endpoints $\phi(0)=\phi(1)=0$ and satisfies both $\left|\left|\phi^{(\ell)}\right|\right|_{L^\infty}\leq 1$ for $1\leq\ell\leq 5$ and
    \begin{equation}\label{eq:Ass2}
    \Lambda_\phi\eqdef\int_0^1\phi(y)\sin(2\pi y)\sin(\pi y)dy\neq 0.
    \end{equation}
    We denote the set of admissible functions as $\mathcal{A}$.
\end{define}

\begin{rem}
    The results presented in Section \ref{sec:DETAILS!} hold even if we relax the requirement $\left|\left|\phi^{(\ell)}\right|\right|_{L^\infty}\leq 1$ in Definition \ref{def:phi}, although the constants that appear in the estimates would then depend on the size of $\phi$ and its derivatives. For convenience, we bound the norm by $1$ so that the parameter $\eta$ fully encompasses the size of the domain perturbation.
\end{rem}

\quad Note that taking $\phi\in\mathcal{A}$ was also necessary for the analysis in \cite{BGM}, although the requirement $\Lambda_\phi\neq 0$ plays a more crucial role in several key aspects of this problem, as outlined in Remark \ref{rem:1}.  Under the conditions in Definition \ref{def:phi}, $R(N)\subset \Omega_\phi(\eta, N)$ for all $N$, all $\eta>0$. For simplicity, we abbreviate $R=R(N)$ and $\Omega=\Omega_\phi(\eta, N)$ when the context is clear. 
Taking $\eta\ll 1$ ensures that the perturbation $\Omega\backslash R$ is small. When $\eta=0$, the domain is simply the rectangle $R$, over which the eigenpairs of (\ref{eq:PDE}) can be written as $\left(\psi_{m,n},\lambda_{m,n}\right)$ where
\begin{equation}\label{eq:eigenpairs}
    \psi_{m,n}(x,y)\eqdef\frac{2}{\sqrt{N}}\sin\left(\frac{m\pi}{N}x\right)\sin\left(n\pi y\right) \quad \quad \textrm{and} \quad \quad \lambda_{m,n}\eqdef\pi^2\left(\frac{m^2}{N^2}+n^2\right)
\end{equation}
for $(m,n)\in\mathbb{Z}^2$. The eigenfunction prefactor is chosen to normalize $||\psi_{m,n}||_{L^2(R)}=1$. We are interested in the case when $\lambda_{2,2}$ is degenerate, which only occurs when the aspect ratio $N$ satisfies 
\begin{equation}
    k^2=3N^2+4 \nonumber
\end{equation} 
for some integer $k\geq 3$. By (\ref{eq:eigenpairs}), this enforces the multiplicity $\lambda_{2,2}=\lambda_{k,1}$. Under (\ref{eq:Ass2}), the degeneracy breaks for small $\eta>0$, and we find a bifurcation of simple eigenvalue branches stemming from $\lambda_{2,2}$. For the sake of notation, we let $\mu=\mu(\eta)$ denote the eigenvalue along the upper branch and $\gamma=\gamma(\eta)$ the eigenvalue along the lower branch, so that 
\begin{equation}
    \mu(0)=\gamma(0) \quad \textrm{and} \quad \mu(\eta)>\gamma(\eta) \nonumber
\end{equation}
for small positive $\eta$. If $v(x,y;\eta)$ and $w(x,y;\eta)$ are normalized in $L^2(\Omega)$ and satisfy
\begin{equation}\label{eq:v,w}
    \begin{cases} (\Delta+\mu)v=0, & \textrm{in }\Omega \\ \: v=0, & \textrm{on }\p\Omega\end{cases} \quad \quad \quad \begin{cases} (\Delta+\gamma)w=0, & \textrm{in }\Omega \\ \:w=0, & \textrm{on }\p\Omega \end{cases}
\end{equation}
then $v$ and $w$ depend on $\phi, N$ and are pointwise smooth in $\eta$ for $(x,y)\in R$ \cite{TK}. Further, because $\mu,\gamma$ are simple, the nodal sets of $v,w$ are unique. For fixed $\phi, N$, taking $\eta\to 0$ allows us to recover two unique eigenfunctions $v_0, w_0$ on $R$, one corresponding to each eigenvalue branch. Namely,
\begin{equation}\label{eq:onRect}
    v_0=\lim_{\eta\to 0}v \quad \textrm{and} \quad w_0=\lim_{\eta\to 0}w
\end{equation}
where $v,w$ solve (\ref{eq:v,w}) on $\Omega$. Note that the eigenfunctions in (\ref{eq:v,w}) are defined uniquely up to a sign change; however, because we are concerned with their nodal sets, all of our results hold irrespective of sign. While the nodal sets of $v_0,w_0$ can be parametrized explicitly for given $\phi, N$, there are a few common properties as described in the following theorem.

\begin{theorem}\label{prop:n=0}
Let $k\geq 3$ be an integer and $N>0$ be such that $k^2=3N^2+4$. Let $\phi \in\mathcal{A}$ as in Definition \ref{def:phi} and consider $v_0,w_0$ as presented in (\ref{eq:onRect}). There exists a constant $c_0\in\left(0,\frac{1}{2}\right)$, dependent only on $\Lambda_\phi$, such that the following holds:
\begin{enumerate}
    \item If $k\geq 4$ is even, then the nodal set of $v_0$ separates four nodal domains and features a crossing at $\left(\frac{N}{2},\overline{y}\right)$ where the height satisfies $$c_0\leq \overline{y}\leq 1-c_0.$$
    \item If $k\geq 5$ is odd, then the nodal set of $v_0$ separates three nodal domains and lies outside $$\left\{\left|x-\frac{N}{2}\right|< c_0\right\}.$$
    \item If $k\geq 8$, then the nodal set of $w_0$ consists of $(k-1)$ disjoint curves that separate $k$ nodal domains and lie outside $$\bigcup_{j=1, j\textrm{ odd}}^{2k-1}\left\{\left|x-\frac{jN}{2k}\right|<c_0\right\}.$$ Each curve intersects the top and bottom boundary.
    \item If $k=5,6,7$, then the nodal set of $w_0$ separates either $k$ or $(k-2)$ nodal domains.
\end{enumerate}
In each case, every intersection of the nodal set with the boundary $\p R(N)$ is orthogonal.
\end{theorem}

\begin{rem}\label{rem:1}
    The condition (\ref{eq:Ass2}) is necessary to guarantee that the eigenvalue degeneracy breaks ($\mu>\gamma$) for small $\eta$, as described in Remark \ref{rem:BananaSplit}. It also provides a lower bound for the distance between $\overline{y}$ and the boundary $\p R(N)$, as highlighted in Remark \ref{rem:heightBdd}, and plays a role in the nodal set of item $(i)$ in Theorem \ref{prop:n=0} when $\eta>0$, detailed in Remark \ref{rem:BananaDoubleSplit}. For these reasons, the condition $\Lambda_\phi\neq 0$ is necessary for all of the results presented in this work.
\end{rem}

\quad Figure \ref{fig:Examples} below illustrates possible nodal configurations satisfying items $(i)-(iii)$ of Theorem \ref{prop:n=0} respectively. These examples were constructed with the boundary function $\phi(y)=Zy(1-y)^2$ with $Z\in\mathbb{R}_+$ chosen to satisfy Definition \ref{def:phi}.

\begin{figure}[H]
    \centering
    \includegraphics[scale=0.36]{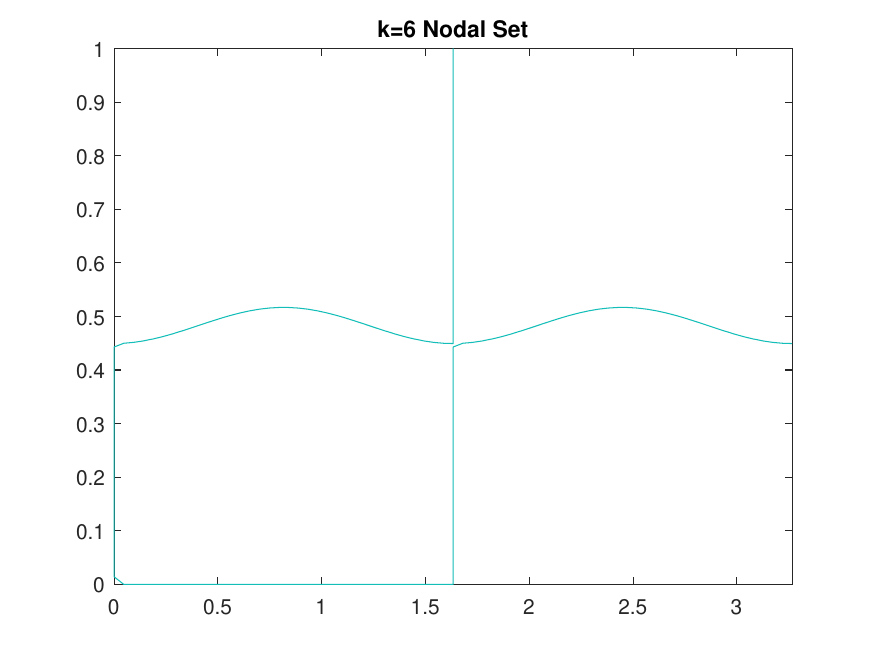}\hfill
    \includegraphics[scale=0.36]{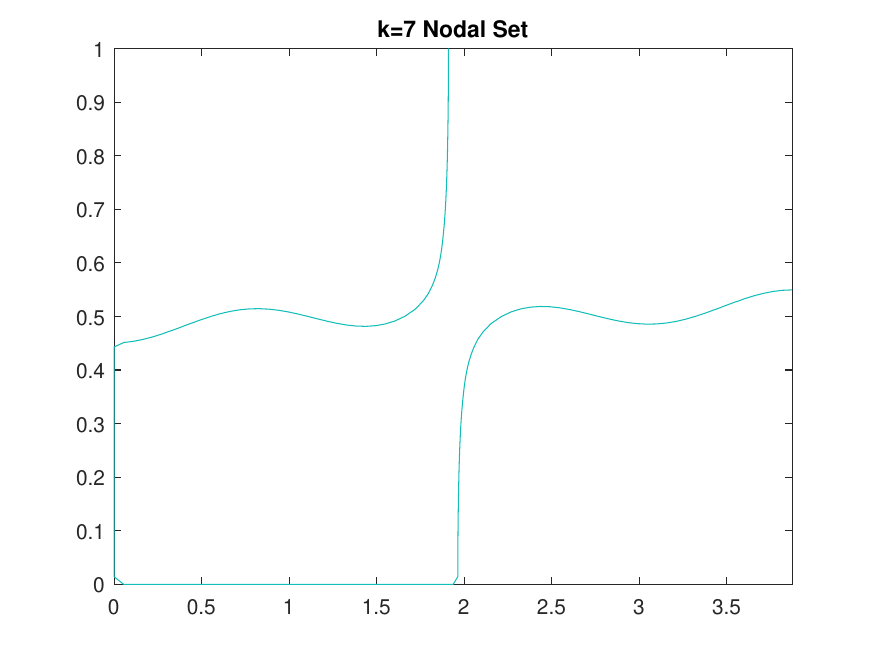}\hfill
    \includegraphics[scale=0.36]{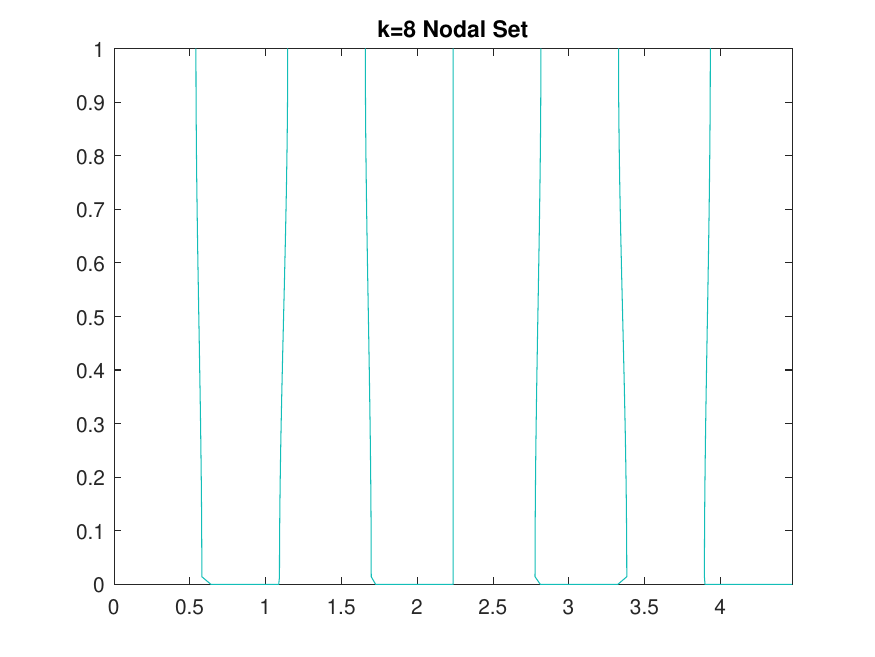}
    \caption{Nodal Sets for $v_0$ when $k=6,7$; Nodal Set for $w_0$ when $k=8$.}
    \label{fig:Examples}
\end{figure}

In the case $k=3$, item $(ii)$ still holds, although the nodal set possibly intersects a corner of $R$. If $k=3,4$, then the nodal set of $w_0$ separates $k$ nodal domains, with the potential of a nodal curve intersecting either a corner or a vertical boundary component. We focus on higher values of $k$ for the sake of uniformity.

\quad For comparison, we establish qualitative properties and a precise description of the nodal set on the perturbed domain $\Omega$. In particular, we are interested in the local nodal behavior near the intersection point in item $(i)$ of Theorem \ref{prop:n=0}. Under the non-resonant assumption in \cite{BGM}, this is where the most substantial change to the nodal set occurred; the potential vanishing of a nodal crossing is quite delicate and demands a careful analysis. For this reason, our initial work focuses on the upper branch with even $k\geq 4$. Before stating our results, we briefly describe our eigenfunction decomposition. Just as in \cite{GJ96}, over $R\subset\Omega$ we write $v$ as a partial Fourier series
\begin{equation}\label{eq:adiabatic}
    v(x,y)=v_1(x)\sin(\pi y)+v_2(x)\sin(2\pi y)+E(x,y), \quad \quad E(x,y)=\sum_{j\geq 3}v_j(x)\sin(j\pi y) 
\end{equation}
where 
\begin{equation}\label{eq:InnerProd}
    v_j(x)=2\int_0^1v(x,y)\sin(j\pi y)dy 
\end{equation}
for $j\geq 1$. Using the fact that the Fourier basis is orthogonal and $v$ satisfies an elliptic equation, we can alternatively write each $v_j(x)$ as the solution to a second-order ODE. This allows us to extend them over $[-\eta, N]$ and determine that the error $E(x,y)$ is in fact small in $\eta$. We find that the first two modes in (\ref{eq:adiabatic}) closely resemble eigenfunctions on $R$. In particular, we show that the restriction $v|_R$ is equal to $v_0$ up to an error bounded by a multiple of $\eta^{1-\epsilon}/N^{\frac{3}{2}}$ for any $\epsilon>0$. A similar analysis holds for the eigenfunction $w$ along the lower branch.

\quad Understanding the nodal behavior on $\Omega$ requires a more detailed characterization of $v_0^{-1}(0), w_0^{-1}(0)$ than provided in Theorem \ref{prop:n=0}. For this reason, we postpone the presentation of precise estimates until Section \ref{sec:DETAILS!}. For now, the following result serves as a companion to Theorem \ref{prop:n=0} and offers a summary of our findings.

\begin{theorem}\label{thm:n>0}
    Let $d$ denote the distance between curves and let $\epsilon>0$ be small. Under the same conditions as in Theorem \ref{prop:n=0}, there exist constants $C_0,k_0>0$, dependent only on $\Lambda_\phi$, such that the eigenfunctions in (\ref{eq:v,w}) have the following properties for positive $\eta$:
    \begin{enumerate}
        \item For each even $k\geq k_0$, there exists $\eta_0(k,\epsilon)>0$ such that for $0<\eta\leq\eta_0(k,\epsilon)$, the nodal set of $v(x,y;\eta)$ features two disjoint curves, $\zeta_1, \zeta_2$, that separate $\Omega_\phi(\eta, N)$ into three nodal domains and satisfy
        \begin{equation}
            C_0^{-1}\eta^{\frac{1}{2}}\leq d\left(\zeta_1,\zeta_2\right)\leq C_0\eta^{\frac{1}{2}}. \nonumber
        \end{equation}
        Each curve intersects one horizontal and one vertical boundary component.
        \item For each odd $k\geq 5$, there exists $\eta_0(k,\epsilon)>0$ such that for $0<\eta\leq\eta_0(k,\epsilon)$, the nodal set of $v(x,y;\eta)$ features two disjoint curves, $\zeta_3, \zeta_4$, that separate $\Omega_\phi(\eta, N)$ into three nodal domains and satisfy
        \begin{equation}
            C_0^{-1}\leq d\left(\zeta_3, \zeta_4\right)\leq C_0. \nonumber
        \end{equation}
        Each curve intersects one horizontal and one vertical boundary component.
        \item For each $k\geq 8$, there exists $\eta_0(k,\epsilon)>0$ such that for $0<\eta\leq \eta_0(k,\epsilon)$, the nodal set of $w(x,y;\eta)$ features $(k-1)$ disjoint curves $\zeta_i$ that separate $\Omega_\phi(\eta, N)$ into exactly $k$ nodal domains and satisfy
        \begin{equation}
            C_0^{-1}\leq d\left(\zeta_i, \zeta_j\right)\leq C_0 \nonumber
        \end{equation}
        for all $i\neq j$. Each curve intersects the top and bottom boundary.
    \end{enumerate}
    In each case, every intersection of the nodal set with the boundary $\p\Omega_\phi(\eta, N)$ is orthogonal.
\end{theorem}

\quad Theorem \ref{thm:n>0} follows from the results in Section \ref{sec:DETAILS!}. In item $(i)$, we require $k$ be large so that the error in \ref{eq:adiabatic} is sufficiently small near $\left(\frac{N}{2},\overline{y}\right)$, the intersection point of $v_0^{-1}(0)$. Otherwise our results hold for all $k$ of interest. By comparing items $(i)$ in Theorems \ref{prop:n=0} and \ref{thm:n>0}, we conclude that the number of nodal domains decreases for small perturbations. In the remaining cases, the characterization of the nodal set does not change.

\begin{rem}
    In the case that $k$ is even, our characterization of the nodal set of $v$ in Theorem \ref{thm:n>0} is consistent with the results presented in \cite{BGM}. However, we require the size of $\eta_0$ to be limited by the integer $k$ satisfying $k^2=3N^2+4$. In Lemma \ref{lem:AD10}, this becomes necessary to counteract the smallness of the spectral gap on the unperturbed rectangle $[0,N]\times[0,1]$. In Proposition \ref{prop:W}, this is necessary to keep other eigenfunction modes from becoming dominant as the perturbation grows.
\end{rem}

\subsubsection{Spectral partitions.} Let $\llbracket u\rrbracket$ denote the number of nodal domains of a function $u$ satisfying (\ref{eq:PDE}). Courant's nodal domain theorem states that $\llbracket u\rrbracket$ is no more than the minimal index of the corresponding eigenvalue $\lambda$ \cite{BH15}. If an eigenfunction saturates this bound, i.e. $\llbracket u\rrbracket=\textrm{index}(\lambda)\eqdef \min\{j:\lambda=\lambda_j\}$, then we label it \textit{Courant sharp}. The collection of nodal domains of $u$ forms a bipartite equipartition of $\Omega$. However, this partition is minimal, as defined in \cite{BCCM22}, if and only if $u$ is Courant sharp, so there is particular interest in understanding when an eigenfunction meets this qualification.

\quad According to Theorems \ref{prop:n=0} and \ref{thm:n>0}, $\llbracket w\rrbracket$ and $\textrm{index}(\gamma)$ remain unchanged as $\eta$ becomes positive. In contrast, $\textrm{index}(\mu)$ increases under boundary perturbations satisfying (\ref{eq:Ass2}), and $\llbracket v\rrbracket$ is non-increasing. In particular, $$\llbracket w_0\rrbracket=k \quad \textrm{and} \quad \llbracket v_0 \rrbracket =\begin{cases} 3, & k \textrm{ odd} \\ 4, & k \textrm{ even} \end{cases}$$ while $\llbracket w\rrbracket=k$ and $\llbracket v\rrbracket=3$ for large enough $k$ and small enough $\eta$. Because $v$ corresponds to the upper-branch eigenvalue $\mu$, we have
$$\llbracket v\rrbracket\leq \llbracket w\rrbracket\leq \textrm{index}(\gamma)<\textrm{index}(\mu)$$ for $\eta>0$. Thus, the eigenfunction $v$ cannot possibly be Courant sharp whereas both $\llbracket w\rrbracket$ and $\textrm{index}(\gamma)$ are invariant under small domain deformations. In particular, $\textrm{index}(\mu)=\textrm{index}(\gamma)>k$ when $\eta=0$ because by (\ref{eq:eigenpairs}), each of the eigenvalues in $$\{\lambda_{1,2}\}\cup\left\{\lambda_{j,1}\right\}_{j=1}^{k-1}$$ is less than $\lambda_{k,1}$. Thus, $v_0,w_0$ are not Courant sharp and neither are $v,w$ for small $\eta>0$.

\subsection{Comprehensive description of the nodal set geometry.}\label{sec:DETAILS!}
In this section, we present results that precisely describe the nodal sets of the eigenfunctions in (\ref{eq:v,w}). We demonstrate tight enough control over both the structure and regularity of the nodal sets to establish the qualitative statements in Theorem \ref{thm:n>0}. In order to properly understand the effects of the domain perturbation, we first introduce some useful notation. We retain the labeling $v_0=\lim_{\eta\to 0}v$ and $w_0=\lim_{\eta\to 0}w$ where $v,w$ solve (\ref{eq:v,w}). Due to the eigenvalue degeneracy $\lambda_{2,2}=\lambda_{k,1}$ on the rectangle, (\ref{eq:eigenpairs}) implies that there is a coefficient pair $(c_1, c_2)$ normalized in $\ell^2$ such that
\begin{equation}
    v_0(x,y)=\left(c_1\psi_{k,1}+c_2\psi_{2,2}\right)(x,y), \quad \quad \textrm{and} \quad \quad w_0(x,y)=\left(-c_2\psi_{k,1}+c_1\psi_{2,2}\right)(x,y). \nonumber
\end{equation}
This motivates the following construction, which presents the nodal parametrizations of $v_0$ and $w_0$ in $R(N)$. Let   
\begin{equation}\label{eq:fv}
    f_v(x)\eqdef \frac{1}{\pi}\arccos\left(-\frac{c_1}{2c_2}\frac{\sin(k\pi x/N)}{\sin(2\pi x/N)}\right)
\end{equation}
so that $(x,f_v(x))$ describes the nodal set $v_0^{-1}(0)$ away from $x=\frac{N}{2}$. Similarly, let
\begin{equation}\label{eq:fw}
    f_w(x)\eqdef \frac{1}{\pi}\arccos\left(\frac{c_2}{2c_1}\frac{\sin(k\pi x/N)}{\sin(2\pi x/N)}\right)
\end{equation}
so that $(x,f_w(x))$ describes the nodal set $w_0^{-1}(0)$ away from $x=\frac{N}{2}$. 

\quad According to Remark \ref{rem:BananaSplit}, the coefficients $c_1,c_2$ are both nonzero under (\ref{eq:Ass2}) and hence the nodal behavior of $v_0$ and $w_0$ at $x=\frac{N}{2}$ depends on the parity of $k$. Theorem \ref{prop:n=0} states that the nodal set for $v_0$ features a crossing along $x=\frac{N}{2}$ when $k$ is even. By (\ref{eq:fv}), the precise location can be determined in terms of $\phi,N$. Retaining the notation from Theorem \ref{prop:n=0}, let $\overline{y}$ denote the height of the nodal set intersection for $v_0$ when $k\geq 4$ is even. Then $\overline{y}$ can be written
\begin{equation}\label{eq:ybar}
    \overline{y}=\frac{1}{\pi}\arccos\left(\cos\left(\frac{k\pi}{2}\right)\frac{kc_1}{4c_2}\right). 
\end{equation}
\quad Theorem \ref{prop:n=0} states that $\overline{y}$ is bounded away from the top and bottom boundaries. The same is true for $f_v(0)=\lim_{x\to 0^+}f_v(x)$. Under this notation, we first present a family of theorems with pointwise estimates on the nodal sets. As in Theorems \ref{prop:n=0} and \ref{thm:n>0}, our results split according to eigenvalue branch and, in the case of the upper branch, the parity of $k$.

\subsubsection{Structure of the nodal sets.} In this section, we present three theorems with estimates on the location of the nodal sets of $v,w$ solving (\ref{eq:v,w}) on $\Omega_\phi(\eta, N)$ for $\eta>0$. Our first result complements item $(i)$ of Theorem \ref{prop:n=0}, when $k$ is an even integer and the eigenfunction corresponds to the upper branch eigenvalue $\mu$.

\begin{theorem}\label{thm:VevenPt}
    Let $\epsilon\in\left(0,\frac{1}{4}\right)$ and let $v(x,y;\eta)$ solve (\ref{eq:v,w}) on $\Omega_\phi(\eta, N)$ with eigenvalue $\mu$. Under the same conditions as in Theorem \ref{prop:n=0}, there exist constants $C_0, k_0>0$, dependent only on $\Lambda_\phi$, such that if $k\geq k_0$ is even, then there exists $\eta_0(k,\epsilon)>0$ such that the nodal set of $v(x,y;\eta)$ has the following properties for $0<\eta\leq \eta_0(k,\epsilon)$:
    \begin{enumerate}
        \item In a disc of radius $2\eta^{\frac{1}{2}-\frac{2}{3}\epsilon}$ centered at $\left(\frac{N}{2},\overline{y}\right)$, the nodal set of $v$ resembles a hyperbola and satisfies 
        \begin{equation}
            C_0^{-1}\eta\leq \left|x-\frac{N}{2}\right|\left|y-\overline{y}\right|\leq C_0 \eta. \nonumber
        \end{equation}
        \item If $x\in\big[\frac{N}{2}-\frac{1}{10},\frac{N}{2}+\frac{1}{10}\big]$, then $v(x,y)\neq 0$ whenever $(x,y)$ satisfies
        \begin{equation}
            \left|x-\frac{N}{2}\right|\left|y-f_v(x)\right|\geq C_0\eta^{1-\epsilon}. \nonumber
        \end{equation}
        \item If $x\notin\big[\frac{N}{2}-\frac{1}{10}, \frac{N}{2}+\frac{1}{10}\big]$ and $x> 0$, then $v(x,y)\neq 0$ whenever $(x,y)$ satisfies
        \begin{equation}
            \left|y-f_v(x)\right|\geq \frac{C_0\eta^{1-\epsilon}}{N\left|\sin\left(\frac{2\pi}{N}x\right)\right|+\eta^{\frac{1}{3}(1-\epsilon)}}. \nonumber
        \end{equation}
        \item If $x\leq 0$, then $v(x,y)\neq 0$ whenever $y$ satisfies $$\left|y-f_v(0)\right|\geq C_0\eta^{\frac{2}{3}(1-\epsilon)}.$$
        \item There exists a tubular neighborhood about one of the curves $y=\overline{y}\pm \left(x-\frac{N}{2}\right)$ which is disjoint from the nodal set of $v$. The width $\omega$ of this neighborhood satisfies 
        \begin{equation}
            C_0^{-1}\eta^{\frac{1}{2}}\leq |\omega|\leq C_0\eta^{\frac{1}{2}}. \nonumber
        \end{equation}
    \end{enumerate}
\end{theorem}

\quad Outside the neighborhood of $\left(\frac{N}{2},\overline{y}\right)$, this result states that the nodal set for $\eta>0$ closely resembles that on the unperturbed rectangle, as described in (\ref{eq:fv}). However, the nodal set on the rectangle $R$ features a crossing and separates four nodal domains, whereas this description features three nodal domains and no internal crossings. This distinction is highlighted by item $(i)$ of Theorem \ref{thm:VevenPt}, where the nodal set of $v$ is hyperbolic in a disc about $\left(\frac{N}{2},\overline{y}\right)$; the size of this disc is justified in Remark \ref{rem:Radius}.

\quad The proof of Theorem \ref{thm:VevenPt} is spread over three sections. Item $(i)$ follows from Propositions \ref{prop:propRemainder} and \ref{prop:PolyStuff} and is the focus of Section \ref{sec:centerEven}. Item $(ii)$ follows from Propositions \ref{prop:horizontal} and \ref{prop:vertical}, while item $(iii)$ follows from results in Sections \ref{sec:awayEven} and \ref{sec:bdry}. Finally, items $(iv)$ and $(v)$ are proved in Section \ref{sec:bdry}. Fortunately, the analysis established in these sections sets the groundwork for the following two results.  In Section \ref{sec:UpperOdd}, we prove the following result, which serves as a companion to item $(ii)$ in Theorem \ref{prop:n=0}, when $k$ is odd.

\begin{theorem}\label{thm:VoddPt}
    Let $\epsilon\in\left(0,\frac{1}{4}\right)$ and let $v(x,y;\eta)$ solve (\ref{eq:v,w}) on $\Omega_\phi(\eta, N)$ with eigenvalue $\mu$. Under the same conditions as in Theorem \ref{prop:n=0}, there exists a constant $C_0>0$, dependent only on $\Lambda_\phi$, such that if $k\geq 5$ is odd, then there exists $\eta_0(k,\epsilon)>0$ such that the nodal set of $v(x,y;\eta)$ has the following properties for $0<\eta\leq \eta_0(k,\epsilon)$:
    \begin{enumerate}
        \item If $\left|y-\frac{1}{2}\right|<\frac{1}{2}\max\{f_v(0),1-f_v(0)\}$ and $x>0$, then $v(x,y)\neq 0$ whenever $(x,y)$ satisfies
        \begin{equation}
            \left|y-f_v(x)\right|\geq \frac{C_0\eta^{1-\epsilon}}{N\left|\sin\left(\frac{2\pi}{N}x\right)\right|+\eta^{\frac{1}{3}(1-\epsilon)}}. \nonumber
        \end{equation}
        \item If $\left|y-\frac{1}{2}\right|\geq\frac{1}{2}\max\{f_v(0),1-f_v(0)\}$ and $x>0$, then $v(x,y)\neq 0$ whenever $(x,y)$ satisfies
        \begin{equation}
            \left|x-f_v^{-1}(y)\right|\geq C_0\eta^{1-\epsilon}. \nonumber
        \end{equation}
        \item If $x\leq 0$, then $v(x,y)\neq 0$ whenever $y$ satisfies $$\left|y-f_v(0)\right|\geq C_0\eta^{\frac{2}{3}(1-\epsilon)}.$$ 
        \item There exists a tubular neighborhood about one of the curves $y=\frac{1}{2}\pm\left(x-\frac{N}{2}\right)$ which is disjoint from the nodal set of $v$. The width $\omega$ of this neighborhood satisfies
        \begin{equation}
            C_0^{-1}\leq |\omega|\leq C_0. \nonumber
        \end{equation}
    \end{enumerate}
\end{theorem}

\quad Qualitatively, Theorem \ref{thm:VoddPt} states that the nodal set of $v$ does not change when $k$ is odd. Just as when $\eta=0$, this result describes a nodal set that features no internal crossings and separates three nodal domains. Finally, we prove the following theorem in Section \ref{sec:Lower}, which complements item $(iii)$ of Theorem \ref{thm:n>0}, when the eigenfunction corresponds to the lower branch eigenvalue $\gamma$.

\begin{theorem}\label{thm:WPt}
    Let $\epsilon>0$ and let $w(x,y;\eta)$ solve (\ref{eq:v,w}) on $\Omega_\phi(\eta, N)$ with eigenvalue $\gamma$. Under the same conditions as in Theorem \ref{prop:n=0}, there exists a constant $C_0>0$, dependent only on $\Lambda_\phi$, such that if $k\geq 8$, then there exists $\eta_0(k,\epsilon)>0$ such that the nodal set of $w(x,y;\eta)$ has the following properties for $0<\eta\leq \eta_0(k,\epsilon)$:
    \begin{enumerate}
        \item The nodal set of $w$ lies outside 
        \begin{equation}
            \bigcup_{j=1, j\textrm{ odd}}^{2k-1} \left\{\left|x-\frac{jN}{2k}\right|<C_0\right\}. \nonumber
        \end{equation}
        Further, $w(x,y)\neq 0$ when $x\leq \frac{N}{2k}$ or $x\geq N-\frac{N}{2k}$.
        \item If $y\in \big[\frac{1}{10}, \frac{9}{10}\big]$, then $w(x,y)\neq 0$ whenever $(x,y)$ satisfies
        \begin{equation}
            \left|y-f_w(x)\right|\geq C_0\eta^{1-\epsilon}. \nonumber
        \end{equation}
        If $k$ is even, then $w(x,y)\neq 0$ also requires that $x$ satisfy $\left|x-\frac{N}{2}\right|\geq C_0\eta^{1-\epsilon}$.
        \item Let $j\in\{1,\dots, 2k-3\}$ be odd. If $y\notin\left[\frac{1}{10},\frac{9}{10}\right]$ and $x\in\left(\frac{jN}{2k}, \frac{N}{k}+\frac{jN}{2k}\right)$, then $w(x,y)\neq 0$ whenever $(x,y)$ satisfies
        \begin{equation}
            \left|x-f_w^{-1}(y)\right|\geq C_0\eta^{1-\epsilon}. \nonumber
        \end{equation}
        If $k$ is even, then $w(x,y)\neq 0$ also requires that $x$ satisfy $\left|x-\frac{N}{2}\right|\geq C_0\eta^{1-\epsilon}$.
    \end{enumerate}
\end{theorem}

\quad Theorem \ref{thm:WPt} states that when $k$ is odd, the nodal set of $w$ behaves much like the graph $\left(x, f_w(x)\right)$ with $f_w(x)$ as in (\ref{eq:fw}). When $k$ is even, the nodal set of $w$ also includes a nearly-vertical curve near $x=\frac{N}{2}$. In either case, much like when $\eta=0$, the nodal set features disjoint curves that separate $k$ nodal domains and intersect only the top and bottom boundary components of $\Omega$. Theorems \ref{thm:VoddPt} and \ref{thm:WPt} demonstrate two cases when the nodal set characterization remains unchanged for small domain perturbations. In both results, the number of nodal domains, lack of any nodal crossings, and the number of curves are invariant for small $\eta$. 

\begin{rem}\label{rem:Expl}
    In Theorems \ref{thm:VevenPt} and \ref{thm:VoddPt}, we let $\epsilon\in\left(0,\frac{1}{4}\right)$ to fully capture the nodal set behavior of $v$. In Section \ref{sec:centerEven}, we study the nodal set of $v$ in a disc of radius $r$ centered at $\left(\frac{N}{2},\overline{y}\right)$ when $k\geq 4$ is even. In order to identify the main contributions from the eigenfunction expansion (\ref{eq:adiabatic}), it is necessary that the radius $r$ be comparable to $\eta^p$ for a power $p$ strictly between a third and a half. To describe the nodal set as a graph just outside of this disc, we then need the bounds in items $(i)$ and $(ii)$ of Proposition \ref{prop:V} to be $o(\eta^{2p})$ as $\eta\to 0$. For these reasons, we choose the radius $r=2\eta^{\frac{1}{2}-\frac{2}{3}\epsilon}$ and restrict $\epsilon\in\left(0,\frac{1}{4}\right)$. In Theorem \ref{thm:WPt}, we just require $\epsilon>0$ because the nodal curves do not intersect the left or right boundary components.
\end{rem}

\quad While Theorems \ref{thm:VevenPt} - \ref{thm:WPt} demonstrate tight control over the structure of the nodal sets in each case of Theorem \ref{thm:n>0}, we are able to say much more regarding local parametrizations of the nodal set. Our second family of theorems focuses on estimating the regularity of the nodal set.

\subsubsection{Regularity of the nodal sets.} In this section, we present three theorems with estimates on the first-order regularity of the nodal sets of $v,w$ solving (\ref{eq:v,w}) on $\Omega_\phi(\eta, N)$ for $\eta>0$. In particular, we determine that every intersection between the nodal set and the boundary is orthogonal. The first result builds upon Theorem \ref{thm:VevenPt}, when $k$ is even and the eigenfunction corresponds to the upper eigenvalue branch $\mu$.

\begin{theorem}\label{thm:VevenDer}
    Under the same conditions as in Theorem \ref{thm:VevenPt}, there exist constants $C_0, k_0>0$, dependent only on $\Lambda_\phi$, such that if $k\geq k_0$ is even, then there exists $\eta_0(k,\epsilon)>0$ such that the nodal set of $v(x,y;\eta)$ has the following properties for $0<\eta\leq\eta_0(k,\epsilon)$:
    \begin{enumerate}
        \item In a disc of radius $2\eta^{\frac{1}{2}-\frac{2}{3}\epsilon}$ centered at $\big(\frac{N}{2},\overline{y}\big)$, the nodal set can be parametrized as the graphs of two smooth functions, each with a bounded derivative.
        \item Outside a disc of radius $2\eta^{\frac{1}{2}-\frac{2}{3}\epsilon}$ centered at $\big(\frac{N}{2},\overline{y}\big)$, the nodal set can be parametrized by graphs $x=g(y)$ and $y=h(x)$, with
        \begin{equation}
            \left|g'(y)\right|\leq \frac{C_0\eta^{1-\epsilon}}{\left|y-\overline{y}\right|^2} \nonumber
        \end{equation}
        and
        \begin{equation}
            \left|h'(x)-f_v'(x)\right|\leq \begin{cases} \displaystyle \frac{C_0\eta^{1-\epsilon}}{N\left|\sin\left(\frac{2\pi}{N}x\right)\right|}, & 0.1<\left|x-\frac{N}{2}\right|<\frac{N}{2}-0.1 \\ \\ \displaystyle \frac{C_0\eta^{1-\epsilon}}{\left|x-\frac{N}{2}\right|^2}, & \left|x-\frac{N}{2}\right|\leq 0.1 \\ \\ \displaystyle \frac{C_0\eta^{1-\epsilon}}{\left|x-tN\right|^2+\eta^{\frac{2}{3}(1-\epsilon)}}, & \left|x-tN\right|\leq 0.1\end{cases} \nonumber
        \end{equation}
        for each $t\in\left\{0,1\right\}$.
        \item The nodal set intersects the boundary $\p\Omega_\phi(\eta, N)$ orthogonally at exactly $4$ points.
    \end{enumerate}
\end{theorem}

\quad In combination with Theorem \ref{thm:VevenPt}, this result states that the vertical component of the nodal set near $x=\frac{N}{2}$ is nearly flat and the horizontal component away from $x=\frac{N}{2}$ behaves like $f_v(x)$ in (\ref{eq:fv})  up to first order. The estimates in item $(ii)$ of Theorem \ref{thm:VevenDer} worsen for points close to $x=\frac{N}{2}$ or $x=0,1$. This is expected because $v_0=0$ along each of these vertical lines and our methods rely on estimates between the restriction $v|_R$ and the known expression for $v_0$.

\quad As in the case of Theorem \ref{thm:VevenPt}, the proof of this theorem is given over three sections. Item $(i)$ is proved in Section \ref{sec:centerEven}. Item $(ii)$ follows from results in Sections \ref{sec:awayEven} and \ref{sec:bdry}, and item $(iii)$ is proved in Section \ref{sec:bdry}. The following theorem, which accompanies Theorem \ref{thm:VoddPt}, is proved in Section \ref{sec:UpperOdd}.

\begin{theorem}\label{them:VoddDer}
    Under the same conditions as in Theorem \ref{thm:VoddPt}, there exists a constant $C_0>0$, dependent only on $\Lambda_\phi$, such that if $k\geq 5$ is odd, then there exists $\eta_0(k,\epsilon)>0$ such that the nodal set of $v(x,y;\eta)$ has the following properties for $0<\eta\leq\eta_0(k,\epsilon)$:
    \begin{enumerate}
        \item For $\left|y-\frac{1}{2}\right|<\frac{1}{2}\max\{f_v(0),1-f_v(0)\}$, the nodal set can be parametrized by graph $y=h(x)$, with
        \begin{equation}
            \left|h'(x)-f_v'(x)\right|\leq \begin{cases} \displaystyle \frac{C_0\eta^{1-\epsilon}}{N\left|\sin\left(\frac{2\pi}{N}x\right)\right|}, & \left|x-\frac{N}{2}\right|<\frac{N}{2}-0.1 \\ \\ \displaystyle \frac{C_0\eta^{1-\epsilon}}{\left|x-tN\right|^2+\eta^{\frac{2}{3}(1-\epsilon)}}, & \left|x-tN\right|\leq 0.1\end{cases} \nonumber
        \end{equation}
        for each $t\in\{0,1\}$.
        \item For $\left|y-\frac{1}{2}\right|\geq \frac{1}{2}\max\{f_v(0),1-f_v(0)\}$, the nodal set can be locally parametrized by graph $x=g(y)$, with
        \begin{equation}
            \left|g'(y)-\big(f_v^{-1}\big)'(y)\right|\leq C_0\eta^{1-\epsilon}. \nonumber
        \end{equation}
        \item The nodal set intersects the boundary $\p\Omega_\phi(\eta, N)$ orthogonally at exactly $4$ points.
    \end{enumerate}
\end{theorem}

\quad Alongside Theorem \ref{thm:VoddPt}, this result states that the local parametrization of $v^{-1}(0)$ is closely approximated by $f_v(x)$ in (\ref{eq:fv}) up to first order. As in Theorem \ref{thm:VevenDer}, the estimates worsen for points close to the left and right boundary, although we retain enough control to establish orthogonality at each of the points where the nodal set and boundary intersect. Finally, we prove the following result, which accompanies Theorem \ref{thm:WPt}, in Section \ref{sec:Lower}.

\begin{theorem}\label{thm:WDer}
    Under the same conditions as in Theorem \ref{thm:WPt}, there exists a constant $C_0>0$, dependent only on $\Lambda_\phi$, such that if $k\geq 8$, then there exists $\eta_0(k,\epsilon)>0$ such that the nodal set of $w(x,y;\eta)$ has the following properties for $0<\eta\leq\eta_0(k,\epsilon)$:
    \begin{enumerate}
        \item In each neighborhood $\left(\frac{jN}{2k}, \frac{N}{k}+\frac{jN}{2k}\right)$ for odd $j\in\{1,\dots, 2k-3\}$ satisfying $\frac{N}{2}\notin\left(\frac{jN}{2k},\frac{N}{k}+\frac{jN}{2k}\right)$, the nodal set can be locally parametrized as the graph of functions $x=g(y)$ or $y=h(x)$, with
        \begin{equation}
            \left|g'(y)-\left(f_w^{-1}\right)'(y)\right|+\left|h'(x)-f_w'(x)\right|\leq C_0\eta^{1-\epsilon}. \nonumber
        \end{equation}
        \item If $k$ is even, the nodal set in the neighborhood $\left(\frac{N}{2}-\frac{N}{2k}, \frac{N}{2}+\frac{N}{2k}\right)$ can be parametrized as a graph of the function $x=m(y)$ with $$\left|m'(y)\right|\leq C_0\eta^{1-\epsilon}.$$ 
    \item The nodal set intersects the boundary $\p\Omega_\phi(\eta, N)$ orthogonally at exactly $(2k-2)$ points.
    \end{enumerate}
\end{theorem}

\quad  In combination with Theorem \ref{thm:WPt}, this result states that each nodal curve closely resembles one of the $(k-1)$ curves in item $(iii)$ of Theorem \ref{prop:n=0}. In particular, it states that when $k$ is even, the nodal set near $x=\frac{N}{2}$ is nearly flat. The estimates in items $(i)$ and $(ii)$ of Theorem \ref{thm:WDer} are not impacted by how close points are to the boundary because the nodal set of $w$ is bounded away from the left and right boundaries of $\Omega$. Theorems \ref{thm:VevenPt} - \ref{thm:WDer} collectively establish the nodal descriptions featured in Theorem \ref{thm:n>0}. 

\subsection{Outline.} The structure of the paper is as follows. In Section \ref{sec:Hadamard}, we determine the linear combinations of $\psi_{k,1},\psi_{2,2}$ that define $v_0$ and $w_0$, following methods presented in \cite{PG, DH}. In Section \ref{sec:n=0}, we describe the nodal sets on the rectangle and prove Theorem \ref{prop:n=0}. In particular, we highlight how the number of nodal domains, interior crossings, and boundary intersections depend on the integer $k$. Then in Section \ref{sec:Prop1.1}, we study how the eigenfunctions (\ref{eq:v,w}) on $\Omega$ compare to $v_0,w_0$ on $R$.

\quad In the case that $k$ is even, we split our analysis of $v^{-1}(0)$ over three sections by region. In Section \ref{sec:centerEven}, we study this nodal set in a neighborhood of the intersection point from item $(i)$ of Theorem \ref{prop:n=0}. It is here we determine that the size of the opening is comparable to $\eta^{\frac{1}{2}}$ for $k$ large enough. In Section \ref{sec:awayEven}, we study the nodal set of $v$ outside of this neighborhood but away from the boundary $\p\Omega$. In Section \ref{sec:bdry}, we study the nodal set near the left and right boundaries, making use of methods presented in \cite{BCM,BGM}. Finally, in Sections \ref{sec:UpperOdd} and \ref{sec:Lower}, we establish properties of $v^{-1}(0)$ for odd $k$ and of $w^{-1}(0)$ respectively, largely following the analysis of previous sections.

\textbf{Acknowledgements.} The author is grateful to Jeremy Marzuola, Yaiza Canzani, and Thomas Beck for helpful conversations regarding the featured problem. They are also thankful for the anonymous reviewers who provided valuable comments that helped strengthen this paper. The author received support from NSF grants DMS-2045494
and DMS-1900519 as well as from NSF RTG DMS-2135998.

%%%%%%%%%%%%%%%%%%%%%%%%%%%%%%%%%%%%%%%%%%%%%%%%%%%%%%%%%
\section{Hadamard Variation Methods}\label{sec:Hadamard}

\quad In this section, we detail a method for computing expressions for $v_0, w_0$ as presented in (\ref{eq:onRect}). By construction, both $v_0$ and $w_0$ are eigenfunctions on the rectangle $R(N)=[0,N]\times[0,1]$ with the same eigenvalue. By (\ref{eq:eigenpairs}), the degeneracy $\lambda_{2,2}=\lambda_{k,1}$ means that any linear combination of
\begin{equation}\label{eq:basis}
    \psi_{k,1}(x,y)=\frac{2}{\sqrt{N}}\sin\left(\frac{k\pi}{N}x\right)\sin\big(\pi y\big) \quad \textrm{and} \quad \psi_{2,2}(x,y)=\frac{2}{\sqrt{N}}\sin\left(\frac{2\pi}{N}x\right)\sin\big(2\pi y\big)
\end{equation}
belongs to the kernel of $(\Delta+\lambda_{2,2})$ and vanishes on the boundary $\p R(N)$. In this way, we can write
\begin{equation}
    v_0(x,y)=\left(c_1\psi_{k,1}+c_2\psi_{2,2}\right)(x,y) \quad \: \textrm{and} \quad \: w_0(x,y)=\left(-c_2\psi_{k,1}+c_1\psi_{2,2}\right)(x,y) \nonumber
\end{equation}
for some coefficient pair $\textbf{c}=(c_1,c_2)$ with normalization $||\textbf{c}||_{\ell^2}=1$. The coefficients depend on the boundary perturbation $\phi$ and integer $k$ satisfying $k^2=3N^2+4$. Likewise, the eigenvalue branches $\mu,\gamma$ depend on $\phi,k$. By \cite{TK}, these eigenvalues and their corresponding eigenfunctions are smooth in $\eta$. Letting $\dot{u}=\left(\p_\eta u\right)\big|_{\eta=0}$ be the derivative in $\eta$ evaluated at $\eta=0$,  we make use of a Hadamard variation argument, presented in \cite{PG,DH} on smooth domains, to characterize the coefficient dependence.

\begin{prop}\label{prop:Hadamard}
Under the same conditions as in Theorem \ref{prop:n=0}, there exists a symmetric $2\times 2$ matrix $\textbf{D}$ with spectrum $\sigma(\textbf{D})=\{-\dot{\mu}, -\dot{\gamma}\}$ and corresponding normalized eigenvectors $\textbf{c}=(c_1, c_2)^T$, $\textbf{c}^\perp=(-c_2, c_1)^T$. More precisely, 
\begin{equation}
    \textbf{D}=\frac{4\pi^2}{N^3}\begin{pmatrix} k^2\int_0^1 \phi(y)\sin^2(\pi y)dy & 2k\int_0^1 \phi(y)\sin(2\pi y)\sin(\pi y)dy \\ \\ 2k\int_0^1 \phi(y)\sin(2\pi y)\sin(\pi y)dy & 4\int_0^1 \phi(y)\sin^2(2\pi y)dy\end{pmatrix}. \nonumber
\end{equation}
\end{prop}

\begin{rem}\label{rem:BananaSplit}
    Notice that the off-diagonal terms in $\textbf{D}$ are multiples of $\Lambda_\phi$. Thus, (\ref{eq:Ass2}) implies that $\dot{\mu}\neq \dot{\gamma}$, causing the degeneracy to break for small $\eta$. Further, $\sgn(\dot{\mu})=\sgn(\dot{\gamma})=-1$ as the trace of $\textbf{D}$ is positive by Definition \ref{def:phi} and the determinant of $\textbf{D}$ is positive by Cauchy-Schwarz.
\end{rem}

\begin{proof}[Proof of Proposition \ref{prop:Hadamard}]
    Given a boundary function $\phi$, we can deform $R$ into $\Omega$ by the diffeomorphism
\begin{equation}\label{eq:diffeo}
    h(x,y;\eta)=\left(\frac{N+\eta\phi(y)}{N}x-\eta\phi(y),y\right) 
\end{equation}
satisfying $h(x,y;0)=I_R$, the identity on the rectangle. In the case where the domain is $C^2$, the eigenfunction variation satisfies $\dot{v}+\dot{h}\cdot\nabla v_0=0$ on the boundary, as in \cite{PG, DH}. However, by tracking each boundary component individually, we show that this expression remains valid on $\p R$. Along the left boundary, $v(-\eta\phi(y),y;\eta)=0$ for all $y\in(0,1)$; this also holds throughout the deformation. By evaluating the $\eta$-derivative at $\eta=0$,
\begin{equation}\label{eq:eigenvar}
    \dot{v}(0,y)=\phi(y)\p_xv_0(0,y).
\end{equation}
In a similar manner, $\dot{v}=0$ along each of the remaining boundary components. By Definition \ref{def:phi}, this means that $\dot{v}$ is continuous along $\p R$. Since $v$ satisfies (\ref{eq:v,w}), the variation $\dot{v}(x,y)$ solves 
\begin{equation}\label{eq:Vdot}
    \begin{cases} 
        (\Delta+\mu)\dot{v}+\dot{\mu}v_0=0, & \textrm{in }R \\
        \: \dot{v}+\dot{h}\cdot \nabla v_0=0, & \textrm{on }\p R. 
    \end{cases}
\end{equation}
Here $\mu(0)=\lambda_{2,2}=\lambda_{k,1}$, but we retain this labeling to emphasize that $\dot{\mu}$ is not necessarily equal to $\dot{\gamma}$. By integrating (\ref{eq:Vdot}) against $\psi_{k,1}$, $\psi_{2,2}$ and using Green's identity,
\begin{align}
    \dot{\mu}c_1+c_1\int_{\p R}\left(\dot{h}\cdot\nabla \psi_{k,1}\right)\frac{\p}{\p\nu}\psi_{k,1}+c_2\int_{\p R}\left(\dot{h}\cdot\nabla\psi_{2,2}\right)\frac{\p}{\p\nu}\psi_{k,1}=0 \nonumber \\
    \dot{\mu}c_2+c_1\int_{\p R}\left(\dot{h}\cdot\nabla \psi_{k,1}\right)\frac{\p}{\p\nu}\psi_{2,2}+c_2\int_{\p R}\left(\dot{h}\cdot\nabla\psi_{2,2}\right)\frac{\p}{\p\nu}\psi_{2,2}=0 \nonumber
\end{align}
where $\frac{\p}{\p\nu}$ is the normal derivative along the boundary $\p R$. By (\ref{eq:diffeo}), this can be written as the system $\left(\dot{\mu}I+\textbf{D}\right)\textbf{c}=0$ where $\textbf{c}=(c_1, c_2)$ and $\textbf{D}$ is the desired matrix. This argument also holds for the eigenvalue branch $\gamma$. 
\end{proof}

\quad Proposition \ref{prop:Hadamard} implies that the eigenvalue variation scales nicely with respect to the boundary, whereas the coefficient pair $\textbf{c}$ is independent of scaling. Namely, $\dot{\mu}(s\phi)=s\dot{\mu}(\phi)$ and $\textbf{c}(s\phi)=\textbf{c}(\phi)$ for $s\in \mathbb{R}_+$, $\phi\in\mathcal{A}$. It also suggests that these coefficients are closely related to variational quantities, such as $\dot{\mu},\dot{\gamma}$. This relationship is highlighted further in the following statement.

\begin{lemma}\label{prop:2nd}
    Let the conditions in Theorem \ref{prop:n=0} hold. The coefficient pair $\textbf{c}=(c_1,c_2)$ can be constructed as
    \begin{equation}
        c_1=-\frac{k\pi}{N^{\frac{3}{2}}}\left(\frac{\dot{v}_1(0)}{\dot{\mu}}\right) \quad \textrm{and} \quad c_2=-\frac{2\pi}{N^{\frac{3}{2}}}\left(\frac{\dot{v}_2(0)}{\dot{\mu}}\right) \nonumber
    \end{equation}
    where $\dot{v}_n(0)$, $n=1,2$, is the first variation of (\ref{eq:InnerProd}) evaluated at $x=0$. 
\end{lemma}

\begin{proof}[Proof of Lemma \ref{prop:2nd}] To establish the variational construction of $\textbf{c}$, we begin with the system in (\ref{eq:Vdot}). If we write the solution $\dot{v}$ as the partial Fourier series
\begin{equation}
    \dot{v}(x,y)=\sum_{j\geq 1} \dot{v}_j(x)\sin(j\pi y), \quad \quad \dot{v}_j(x)=2\int_0^1 \dot{v}(x,y)\sin(j\pi y) dy \nonumber
\end{equation}
then each Fourier coefficient $\dot{v}_j(x)$ solves an inhomogeneous differential equation on $[0,N]$. For our purposes, we focus on the first two modes.
\begin{align}
    \dot{v}_1''(x)+\frac{k^2\pi^2}{N^2}\dot{v}_1(x)=-2\dot{\mu}\int_0^1 (c_1\psi_{k,1}+c_2\psi_{2,2})\sin(\pi y)dy, \nonumber \\ \dot{v}_2''(x)+\frac{4\pi^2}{N^2}\dot{v}_2(x)=-2\dot{\mu}\int_0^1 (c_1\psi_{k,1}+c_2\psi_{2,2})\sin(2\pi y)dy. \nonumber
\end{align}
From this, we can explicitly solve for $\dot{v}_n(x)$ up to a multiple of the fundamental solution. However, the fundamental solution has Dirichlet boundary conditions (at $x=0,N$) and therefore does not contribute in determining $\dot{v}_n(0)$. From this solution,
\begin{equation}
    \dot{v}_1(0)=-\frac{\dot{\mu}c_1N^{\frac{3}{2}}}{k\pi} \quad \textrm{and} \quad \dot{v}_2(0)=-\frac{\dot{\mu}c_2N^{\frac{3}{2}}}{2\pi} \nonumber
\end{equation}
which can both be rearranged to solve for $c_1,c_2$. 
\end{proof}

\quad From the proof of Proposition \ref{prop:Hadamard}, we recover an expression for the variation of each Fourier coefficient $v_j(0)$ in terms of known quantities. By (\ref{eq:InnerProd}) and (\ref{eq:eigenvar}),
\begin{equation}\label{eq:VarCoeff}
    \dot{v}_j(0)=\frac{4\pi}{N^{\frac{3}{2}}}\int_0^1 \phi(y)\left(2c_2\sin(2\pi y)+kc_1\sin(\pi y)\right)\sin(j\pi y)dy.
\end{equation}
This construction is consistent with both Proposition \ref{prop:Hadamard} and Lemma \ref{prop:2nd} and provides some intuition as to why Proposition \ref{prop:VJ} holds.

%%%%%%%%%%%%%%%%%%%%%%%%%%%%%%%%%%%%%%%%%%%%%%%%%%%
\section{Proof of Theorem \ref{prop:n=0}}\label{sec:n=0}

\quad This section is dedicated to proving Theorem \ref{prop:n=0}, which describes the nodal sets of two eigenfunctions $v_0, w_0$ on the rectangle $R(N)$. As presented in Theorem \ref{prop:n=0}, we can write
\begin{equation}
    v_0(x,y)=\left(c_1\psi_{k,1}+c_2\psi_{2,2}\right)(x,y) \quad \: \textrm{and} \quad \: w_0(x,y)=\left(-c_2\psi_{k,1}+c_1\psi_{2,2}\right)(x,y) \nonumber
\end{equation}
where the coefficient pair $\textbf{c}=(c_1, c_2)$ is constructed in Proposition \ref{prop:Hadamard} and the basis eigenfunctions $\psi_{k,1},\psi_{2,2}$ are defined in (\ref{eq:basis}). When convenient, we hereafter use the notation $A\sim B$ when the quantity $|A|/|B|$ is bounded both above and below by constants independent of $\eta$, $N$ and only dependent on $\phi$ through the value $\Lambda_\phi$. By construction, the integer $k$ satisfying $k^2=3N^2+4$ determines the aspect ratio, $k\sim N$. However, for the sake of uniformity with prior work, all of our estimates will be in terms of $N$. Before we proceed to the proof of Theorem \ref{prop:n=0}, the following estimate is needed.

\begin{lemma}\label{prop:coeff}
Under the same conditions as in Theorem {\ref{prop:n=0}}, there exists a constant $c>0$, dependent only on $\Lambda_\phi$, such that the coefficient pair $(c_1, c_2)$ satisfies
\begin{equation}
    \frac{c}{N}<|c_1|<\sqrt{\frac{16}{k^2+16}}-\frac{c}{N} \quad \:\: \textrm{and} \quad \:\: \sqrt{\frac{k^2}{k^2+16}}+\frac{c}{N}<|c_2|<1-\frac{c}{N}. \nonumber
\end{equation}
for all $k\geq 4$. In particular, because $k\sim N$, $c_1\sim 1/N$ and $c_2\sim 1$.
\end{lemma}

\begin{rem}\label{rem:3}
    Lemma \ref{prop:coeff} implies that $|c_1|<\frac{4}{k}$ for all $k\geq 4$. By the normalization $||\textbf{c}||_{\ell^2}=1$, this inequality also holds for $k=3$. Because $k^2=3N^2+4$, this extends the estimate $c_1\sim 1/N$ to all $k\geq 3$.
\end{rem} 

\begin{proof}[Proof of Lemma \ref{prop:coeff}.] We prove the bounds for $c_1$ as the normalization $||\textbf{c}||_{\ell^2}=1$ then provides the corresponding bounds for $c_2$. Let $\textbf{D}$ be the symmetric matrix defined in Proposition \ref{prop:Hadamard}.  For simplicity of notation, label the elements of this matrix as $a,b,c$, i.e.
\begin{equation}\label{eq:Dc12}
    \textbf{D}=\begin{pmatrix} a & b \\ b & c\end{pmatrix} \quad \quad \textrm{so that} \quad \quad c_1^2=\frac{b^2}{b^2+(a+\dot{\mu})^2}=\frac{1}{1+\big(\frac{a+\dot{\mu}}{b}\big)^2}.
\end{equation}
We quickly remark that by Definition \ref{def:phi}, the entries in $\textbf{D}$ satisfy $a\sim 1/N$, $b\sim 1/N^2$, and $c\sim 1/N^3$; this immediately provides the lower bound for $|c_1|$. To determine the upper bound, we focus first on $\dot{\mu}$. Because we denote $\mu$ as the upper branch eigenvalue, Remark \ref{rem:BananaSplit} implies that $|\dot{\mu}|<|\dot{\gamma}|$. Therefore, because $\{-\dot{\mu}, -\dot{\gamma}\}$ are eigenvalues of $\textbf{D}$, we can write
\begin{equation}
    \dot{\mu}=-\frac{1}{2}\left((a+c)-\sqrt{(a-c)^2+4b^2}\right). \nonumber
\end{equation}
If we show that $\big|\frac{a+\dot{\mu}}{b}\big|-\frac{k}{4}$, is bounded below by a constant, then (\ref{eq:Dc12}) provides the desired result. To determine the lower bound of $\left|\frac{a+\dot{\mu}}{b}\right|$, we write
\begin{align}\label{eq:c2Step3}
    \left|\frac{a+\dot{\mu}}{b}\right|=\frac{1}{2}\left|\frac{a-c}{|b|}+ \sqrt{\left(\frac{a-c}{b}\right)^2+4}\right|=\frac{1}{2}\left(\sqrt{\left(\frac{a-c}{b}\right)^2+4}+\frac{a-c}{|b|}\right)
\end{align}
which is monotonically increasing with respect to the argument $\left(\frac{a-c}{|b|}\right)$. Fortunately, we can bound the difference
\begin{align}\label{eq:c2Step2}
    \frac{a-c}{|b|}-\frac{k^2-16}{4k}
\end{align}
below by a multiple of $\frac{1}{k}$ because $4k(a-c)-(k^2-16)|b|$ is equal to
\begin{align}\label{eq:integral}
    \frac{4\pi^2}{N^3}\int_0^1 \phi(y)\left(4k(k^2\sin^2(\pi y)-4\sin^2(2\pi y))-2k\sgn(b)(k^2-16)\sin(2\pi y)\sin(\pi y)\right)dy.
\end{align}
Because $\phi\in\mathcal{A}$ and $k\geq 4$ is even, the integrand in (\ref{eq:integral}) is bounded below by a multiple of $k^5b\sim N^3$ for all $y\in(0,1)$ and hence (\ref{eq:integral}) is bounded below by a constant independent of $k$. This implies that (\ref{eq:c2Step3}) is bounded below by $\frac{k}{4}$ plus a constant and establishes the upper bound for $c_1^2$.
\end{proof}

\quad We now show that these conditions are enough to establish the geometric properties presented in Theorem \ref{prop:n=0} of the nodal sets of $v_0$ and $w_0$ over the rectangle $R$. We break the proof into two parts.

\begin{proof}[Proof of Theorem \ref{prop:n=0} Items (i) and (ii).] 
The nodal set for $v_0(x,y)$ can be written as
\begin{equation}\label{eq:NodalDescription}
    \left\{(x,y)\in R: c_1\sin\left(\frac{k\pi}{N}x\right)\sin(\pi y)+c_2\sin\left(\frac{2\pi}{N}x\right)\sin(2\pi y)=0\right\}.
\end{equation}
For all $k\geq 3$, we can describe this set by the graph $(x, f_v(x))$ from (\ref{eq:fv}) away from $x=\frac{N}{2}$.

\quad If $k$ is even, then the nodal set of $v$ includes the line $x=\frac{N}{2}$, so the intersection point occurs at $\left(\frac{N}{2}, \overline{y}\right)$ where $\overline{y}=\lim_{x\to \frac{N}{2}}f_v(x)$ is presented in (\ref{eq:ybar}). The bounds in Lemma \ref{prop:coeff} imply that $\overline{y}$ is bounded away from the boundary $\p R$ by a distance $\sim 1$. By comparing (\ref{eq:fv}) and (\ref{eq:ybar}), $$\left|f_v(x)-\frac{1}{2}\right|\leq \left|\overline{y}-\frac{1}{2}\right|$$ for all even $k\geq4$. This implies that the nodal set extending from the midline only intersects the left and right boundaries. Thus, the nodal set for $v_0$  has exactly $4$ nodal domains with an interior crossing along $x=\frac{N}{2}$.

\quad Suppose that $k\geq 3$ is instead odd. By (\ref{eq:NodalDescription}) it is clear that the nodal set cannot intersect the line $x=\frac{N}{2}$. To extend this property to a neighborhood about $x=\frac{N}{2}$, we use
the function in (\ref{eq:fv}) which still describes the nodal set for odd $k$. Let $x_0$ minimize $\left|x-\frac{N}{2}\right|$ over all $x$ satisfying $v_0(x,f_v(x))=0$. By horizontal symmetry, there are two such points and $(x_0,f_v(x_0))$ must lie on either the top or bottom boundary of $R$. This means that $x_0$ must satisfy
\begin{equation}
    \left|\frac{\sin(k\pi x_0/N)}{\sin(2\pi x_0/N)}\right|=2\left|\frac{c_2}{c_1}\right|. \nonumber
\end{equation}
By Lemma \ref{prop:coeff}, the right side is bounded by a multiple of $k$. Because $k$ is odd, this implies $x_0$ is bounded away from $\frac{N}{2}$ by a constant independent of $k$. In the remaining region, we now show that there are exactly two nodal curves. When $k=3$, (\ref{eq:Ass2}) immediately implies that this is true, so we hereafter consider odd $k\geq 5$. 

\quad For $|c_2|=1$, the nodal set of $v_0$ matches that of $\psi_{2,2}$, separating exactly four nodal domains, and for $|c_2|$ slightly less than $1$, there are three nodal domains. As $|c_2|$ continues to decrease, the number of nodal domains increases by steps of two (due to symmetry) up to $k$. Our analysis focuses on $v_0^{-1}(0)\cap \{(x,y)\in R: x\in I_k\}$ for $I_k=\left(\frac{N}{2}+\frac{N}{2k}, \frac{N}{2}+\frac{3N}{2k}\right)$. It is this region in which a local extremum of $f_v(x)$ forms as $|c_2|$ decreases. In order to retain three nodal domains, it is necessary that the height of this extremum remains in $(0,1)$. Observe that
\begin{equation}
    f_v'(x)=0 \quad \textrm{if and only if} \quad \frac{d}{dx}\left(\frac{\sin(k\pi x/N)}{\sin(2\pi x/N)}\right)=0. \nonumber
\end{equation}
Let $x_1$ be the point in $I_k$ such that
\begin{equation}\label{eq:maxK}
    M_k\eqdef\max_{x\in I_k}\left(\frac{\sin^2(k\pi x/N)}{\sin^2(2\pi x/N)}\right)=\frac{\sin^2(k\pi x_1/N)}{\sin^2(2\pi x_1/N)}
\end{equation}
meaning $f_v(x_1)$ is the height of interest. The value $f_v(x_1)$ belongs to $(0,1)$ if and only if
\begin{align}
    \frac{4c_2^2}{c_1^2}>\frac{\sin^2(k\pi x_1/N)}{\sin^2(2\pi x_1/N)} \quad \quad \textrm{or equivalently} \quad \quad 1>|c_2|>\sqrt{\frac{M_k}{4+M_k}}.\nonumber
\end{align}
By the bounds on $|c_2|$ in Lemma \ref{prop:coeff}, it suffices to show $M_k\leq \frac{k^2}{4}$. From (\ref{eq:maxK}), this is equivalent to
\begin{equation}
    \left|\frac{\sin(k\pi x_1/N)}{\sin(2\pi x_1/N)}\right|\leq \frac{k}{2}  \nonumber
\end{equation}
which holds in fact for all $x\in I_k$, $k\geq 5$. Thus, the nodal set of $v_0$ for odd $k\geq 3$ features two curves that separate three nodal domains. For all $k\geq 4$, orthogonality at the boundary follows from (\ref{eq:fv}); namely, we find $f_v'(x)=0$ when $x=0,N$ and $\left(f_v^{-1}\right)'(y)=0$ when $y=0,1$. However, it is possible that the nodal set of $v_0$ intersects a corner when $k=3$. If this is the case, then the equation for $f_v(x)$ implies that it does so at an angle of $\pm\frac{\pi}{4}$.
\end{proof}

\begin{rem} \label{rem:heightBdd}
    By (\ref{eq:ybar}), we find that the height of the intersection $\overline{y}$ is bounded away from the top and bottom boundaries of $R$ by a constant dependent on $\Lambda_\phi$. This is crucial when we study the nodal set of $v$ near $\left(\frac{N}{2},\overline{y}\right)$ in Section \ref{sec:centerEven} and near the corners of $\Omega$ at the end of Section \ref{sec:bdry}.
\end{rem}

\begin{proof}[Proof of Theorem \ref{prop:n=0} Items (iii) and (iv).] First consider the case $k\geq 8$. If $|c_2|= 1$, $w_0^{-1}(0)$ separates $k$ nodal domains. As $|c_2|$ decreases from $1$, interior extremum form in the parametrization of the nodal set that have the potential to decrease the number of nodal domains. In a similar argument to the previous proof, denote the interval of interest $I_k$ below.
\begin{equation}
    I_k=\begin{cases} 
    \left(\frac{3N}{4}-\frac{N}{k}, \frac{3N}{4}\right), & 0\equiv k\Mod{4}
    \\
    \left(\frac{3N}{4}-\frac{3N}{4k}, \frac{3N}{4}+\frac{N}{4k}\right), & 1\equiv k\Mod{4}
    \\
    \left(\frac{3N}{4}-\frac{N}{2k}, \frac{3N}{4}+\frac{N}{2k}\right), & 2\equiv k\Mod{4}
    \\ 
    \left(\frac{3N}{4}-\frac{N}{4k}, \frac{3N}{4}+\frac{3N}{4k}\right), & 3\equiv k\Mod{4}
    \end{cases} \nonumber
\end{equation}
This interval features the first local extremum formed in $R$ as $|c_2|$ decreases from $1$. If the nodal set is to retain exactly $k$ nodal domains, it is necessary that this extremum remains outside of $(0,1)$. By the substitution $(c_1, c_2)\to(-c_2, c_1)$ in (\ref{eq:fv}), we have a local description $f_w(x)$, as in (\ref{eq:fw}), satisfying $w_0\left(x, f_w(x)\right)=0$. Keeping the extremum outside of $(0,1)$ requires the coefficients satisfy
\begin{equation}\label{eq:Wcond}
    \frac{4c_1^2}{c_2^2}\leq\frac{\sin^2(k\pi x_1/N)}{\sin^2(2\pi x_1/N)} \quad \textrm{or equivalently} \quad 0\leq |c_1|\leq \sqrt{\frac{M_k}{4+M_k}}
\end{equation}
where $x_1\in I_k$ is the location of the extremum, i.e. $f_w'(x_1)=0$, and $M_k$ is defined as in (\ref{eq:maxK}). By the bounds on $|c_1|$ in Lemma \ref{prop:coeff}, it suffices to show that $\frac{64}{k^2}\leq M_k$. However, because $k\geq 8$, it is simpler to show that $M_k\geq 1$. Consider $x=\textrm{mid}(I_k)$ as the midpoint of $I_k$. This point belongs to the interval $I_k$ and offers an explicit lower bound for $M_k$.
\begin{equation}
    M_k\geq\left(\frac{\sin(k\pi x/N)}{\sin(2\pi x/N)}\right)^2\bigg|_{x=\textrm{mid}(I_k)}=\begin{cases}
    \sec^2\left(\frac{\pi}{k}\right), & 0\equiv k\Mod{4}
    \\
    \sec^2\left(\frac{\pi}{2k}\right), & 1\equiv k\Mod{4}
    \\
    1, & 2\equiv k\Mod{4}
    \\
    \sec^2\left(\frac{\pi}{2k}\right), & 3\equiv k\Mod{4}
    \end{cases} \nonumber
\end{equation}
all of which are no less than $1$ for $k\geq 8$. This establishes the number of nodal domains for $k\geq 8$. Notice that the cases $k=3,4$ are true because they maintain exactly $3$, respectively $4$, nodal domains by (\ref{eq:Ass2}). To complete item $(iii)$, let $x_j=\frac{jN}{2k}$ for $j\in\{1,\dots, 2k-1\}$ odd. Observe that if $w_0(x_j, y)=0$, then
\begin{equation}
    c_2\sin\left(\frac{j\pi}{2}\right)=2c_1\sin\left(\frac{j\pi}{k}\right)\cos(\pi y) \quad \quad \textrm{or equivalently} \quad \quad \frac{1}{2}\bigg|\frac{c_2}{c_1}\bigg|\leq \big|\cos(\pi y)\big|. \nonumber
\end{equation}
By Lemma \ref{prop:coeff}, this means that $y$ must satisfy $\frac{k}{8}\leq |\cos(\pi y)|$, which is not possible when $k\geq 8$. Further, if we fix $j$ and let $x_0$ minimize $\big|x-x_j\big|$ over all $x$ satisfying $w_0(x,f_w(x))=0$, then in a similar argument as in the proof of item $(ii)$,
\begin{equation}
    \bigg|\frac{\sin(k\pi x_0/N)}{\sin(2\pi x_0/N)}\bigg|=2\bigg|\frac{c_1}{c_2}\bigg|. \nonumber
\end{equation}
The right side is bounded above by $\frac{8}{k}$ by Lemma \ref{prop:coeff}, allowing us to bound $x_0$ away from $x_j$ for all $k\geq 8$.

\quad The nodal characterization in item $(iii)$ of Theorem \ref{prop:n=0} does not necessarily hold for $k\in \{5,6,7\}$. To show that the number of nodal domains cannot be extended to these renegade values of $k$, consider the boundary functions $\phi_1(y)=Z_1y^{24}(1-y)$ and $\phi_2(y)=Z_2y^8(1-y)$ with $Z_1, Z_2\in\mathbb{R}_+$ chosen to satisfy Definition \ref{def:phi}. Recall that scaling does not impact the coefficients from Section \ref{sec:Hadamard}. The nodal sets of $w_0$ prior to a domain deformation with the left boundary parametrized by $\phi_1,\phi_2$ are plotted below in Figures \ref{fig:1568} and \ref{fig:1567} for $k=5,6,7$.
\begin{figure}[H]
    \centering
    \includegraphics[scale=0.36]{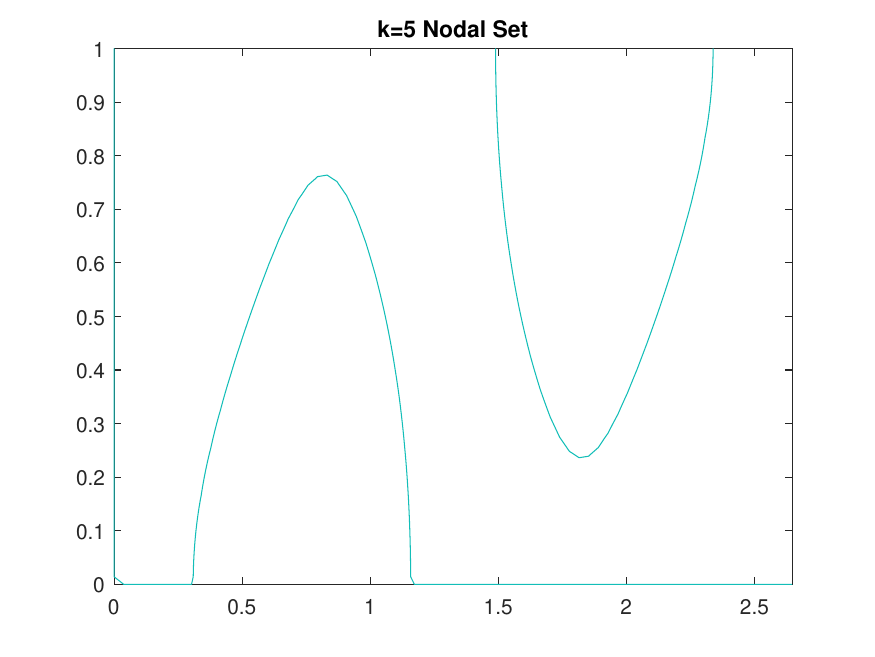}\hfill
    \includegraphics[scale=0.36]{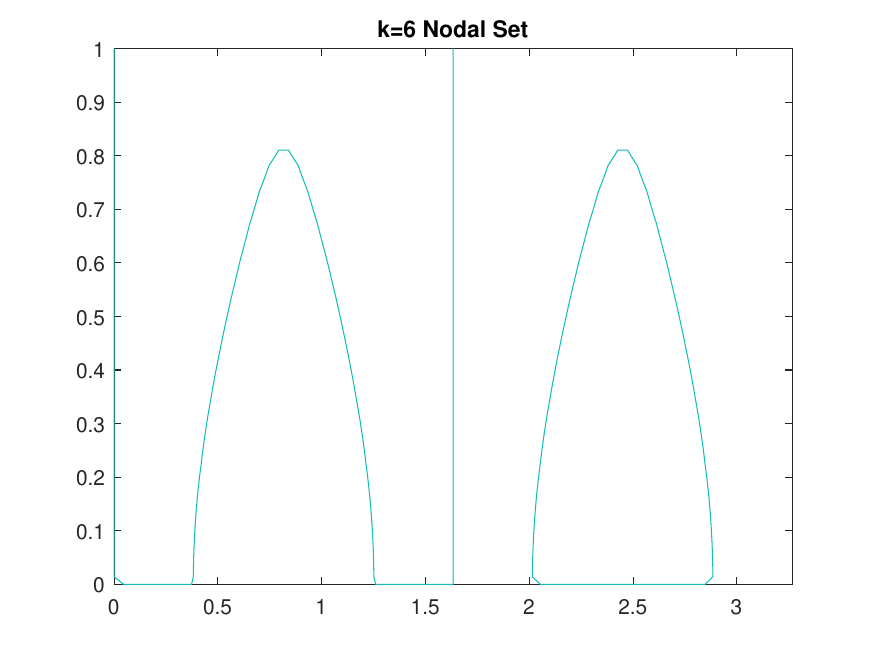}\hfill
    \includegraphics[scale=0.36]{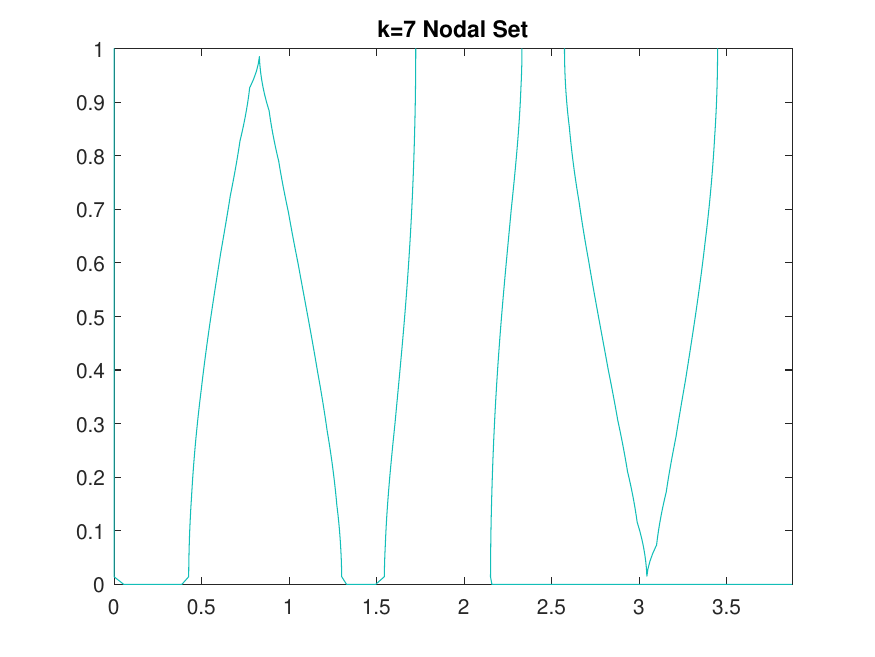}
    \caption{Nodal Sets for $w_0$ with Boundary Function $\phi_1(y)$, $k=5,6,7$.}
    \label{fig:1568}
\end{figure}
\begin{figure}[H]
    \centering
    \includegraphics[scale=0.36]{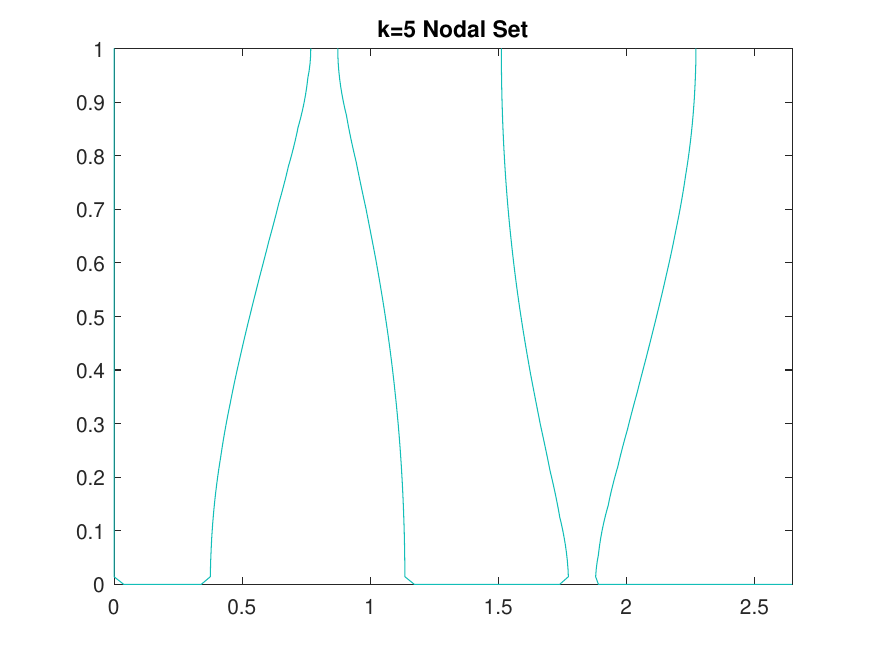} \hfill
    \includegraphics[scale=0.36]{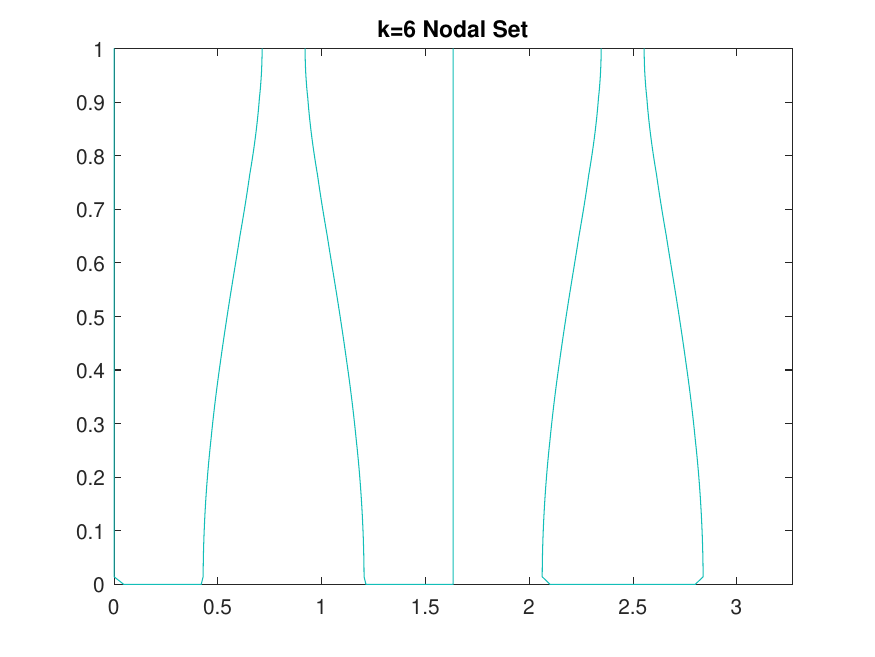} \hfill
    \includegraphics[scale=0.36]{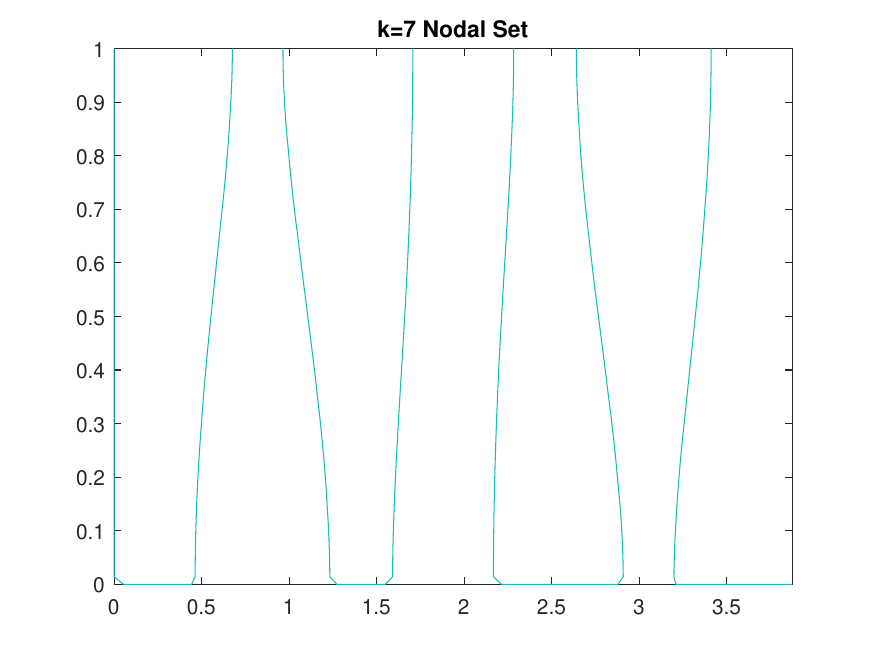}
    \caption{Nodal Sets for $w_0$ with Boundary Function $\phi_2(y)$, $k=5,6,7$.}
    \label{fig:1567}
\end{figure}
\quad As demonstrated, it is possible for the eigenfunction $w_0$ to have either $k$ or $(k-2)$ nodal domains depending on the boundary function $\phi$. Due to symmetry, the nodal sets when $k=5$ separate either three or five nodal domains. Likewise the nodal sets when $k=6$ separate either four or six nodal domains. However, it is not possible for the nodal set of $w_0$ to have three nodal domains when $k=7$. If this were to occur, the extremum formed in $I_7=\left(\frac{4\sqrt{15}}{7}, \frac{5\sqrt{15}}{7}\right)$ would be contained in $(0,1)$. This occurs if and only if
\begin{equation}
    |c_1|>\sqrt{\frac{M_7}{4+M_7}} \quad \textrm{where} \quad M_7=\max_{x\in I_7}\left(\frac{\sin(7\pi x/N)}{\sin(2\pi x/N)}\right)^2\approx 1.736. \nonumber
\end{equation}
But from Lemma \ref{prop:coeff},
\begin{equation}
    |c_1|<\sqrt{\frac{16}{16+7^2}}\approx 0.496 \quad \textrm{and} \quad \sqrt{\frac{M_7}{4+M_7}}\approx 0.55. \nonumber
\end{equation}
We conclude that for $k=5,6,7$, there may be either $k$ or $(k-2)$ nodal domains. In a similar argument to the previous proof, orthogonality at the boundary follows from the expression $f_w(x)$ in (\ref{eq:fw}). 
\end{proof}

%%%%%%%%%%%%%%%%%%%%%%%%%%%%%%%%%%%%%%%%%%%%%%%%%%%
 \section{Estimates on Eigenfunction Behavior}\label{sec:Prop1.1}

\quad In this section, we obtain estimates for the behavior of the eigenfunctions (\ref{eq:v,w}) on $\Omega_\phi(\eta, N)$, making use of the eigenfunction decomposition (\ref{eq:adiabatic}). For brevity, we focus on $v$, the eigenfunction corresponding to the upper eigenvalue $\mu$ after a perturbation to the left boundary of size $\eta\ll 1$, solving
\begin{equation}\label{eq:P1v}
    \begin{cases} (\Delta+\mu)v=0, & \textrm{in }\Omega \\ \: v=0, & \textrm{on }\p\Omega. \end{cases}
\end{equation}
The methods presented in this section hold similarly for the eigenfunction $w$ along the lower eigenvalue branch. For convenience, we hereafter let $c, C$ denote positive constants that may depend on $\Lambda_\phi$ but are otherwise independent of $\eta, N$ and $\phi\in\mathcal{A}$. Such constants are left arbitrary and may change line by line.

\quad We decompose $v$ into the partial Fourier series
\begin{equation}\label{eq:decomposition}
    v(x,y)=v_1(x)\sin(\pi y)+v_2(x)\sin(2\pi y)+E(x,y), \quad \quad E(x,y)=\sum_{j\geq 3} v_j(x)\sin(j\pi y)
\end{equation}
and show that the first two modes closely resemble the eigenfunctions $\psi_{k,1},\psi_{2,2}$ defined in (\ref{eq:basis}) while the terms in the error $E(x,y)$ are small, as made precise in Proposition \ref{prop:V}. In order to better understand the expansion in (\ref{eq:decomposition}), we use (\ref{eq:P1v}) to rewrite each Fourier coefficient 
\begin{equation}\label{eq:P1FCeq}
    v_j(x)=2\int_0^1 v(x,y)\sin(j\pi y)dy
\end{equation}
as a solution to the differential equation
\begin{equation}\label{eq:ODE}
    v_j''+\left(\mu-\pi^2j^2\right)v_j=0 
\end{equation}
with boundary condition $v_j(N)=0$. To simplify notation, we define
\begin{equation}\label{eq:modifEig}
    \mu_j^2=\begin{cases} \mu-\pi^2j^2, & j=1,2 \\ \pi^2j^2-\mu, & j\geq 3.\end{cases}
\end{equation}
Because $[0,N]\times[0,1]\subset \Omega\subset[-\eta,N]\times[0,1]$, domain monotonicity for Dirichlet eigenvalues \cite{H06} guarantees that each of the values in (\ref{eq:modifEig}) is positive. Further, it provides control over the first two modified eigenvalues,
\begin{equation}\label{eq:close}
    \left|\mu_1N-k\pi\right|+\left|\mu_2N-2\pi\right|\leq C\eta /N.
\end{equation}
By construction, $k\sim N$ and therefore $\mu_1\sim 1$ and $\mu_2\sim 1/N$. From (\ref{eq:ODE}), we obtain expressions for the Fourier coefficients in terms of their left boundary condition and the modified eigenvalues $\mu_j$. By (\ref{eq:modifEig}), the first two modes are oscillatory:
\begin{align}
    v_1(x)=v_1(0)\cos(\mu_1x)+A_1\sin(\mu_1x), \quad \textrm{with} \quad A_1\eqdef\frac{-v_1(0)\cos(\mu_1N)}{\sin(\mu_1N)} \label{eq:v1}
    \\
     v_2(x)= v_2(0)\cos(\mu_2x)+A_2\sin(\mu_2x), \quad \textrm{with} \quad A_2\eqdef\frac{-v_2(0)\cos(\mu_2N)}{\sin(\mu_2N)} \label{eq:v2}
\end{align}
whereas the higher modes are exponential:
\begin{equation}
      v_j(x)=\frac{v_j(0)}{\sinh(\mu_jN)}\sinh(\mu_j(N-x)), \quad \textrm{for} \quad j\geq 3. \label{eq:v_k}
\end{equation}
This difference in behavior motivates our definition of the error $E(x,y)$. For convenience, we will often use the index $n=1,2$ to exclusively describe the lower modes. Away from the left and right boundaries, the higher Fourier coefficients experience exponential decay in $N$, allowing us to bound
\begin{equation}\label{eq:coeffbound}
    |v_j(x)|\leq C|v_j(0)|e^{-cj\min\{x, N-x\}} \quad \textrm{for} \quad j\geq 3
\end{equation}
where the constant $c$ is positive and independent of $j$. To improve this estimate in $\eta$, we reference (\ref{eq:P1v}). Elliptic estimates imply that the gradient is bounded, $|\nabla v(x,y)|\leq C$, and a line integral from the left boundary of $\Omega$ implies $|v(0,y)|\leq C\eta$. By (\ref{eq:P1FCeq}), this bound extends to $|v_j(0)|$ for all $j\geq 1$. 

\quad As $\eta\to 0$, both $v_n(0)$ and $\sin(\mu_n N)$ vanish for $n=1,2$. Because both $\mu=\mu(\eta)$ and $v(x,y)=v(x,y;\eta)$ vary smoothly in $\eta$, (\ref{eq:P1FCeq}) and (\ref{eq:modifEig}) imply that $v_n(x)=v_n(x;\eta)$ and $\mu_n=\mu_n(\eta)$ also vary smoothly with respect to the perturbation size. By the definitions in (\ref{eq:v1}), (\ref{eq:v2}) alongside Lemma \ref{prop:2nd}, we can determine the limits
\begin{equation}\label{eq:Alimit}
    \lim_{\eta\to 0} A_1=-\frac{\dot{v}_1(0)}{\dot{\mu}_2N}=\frac{2}{\sqrt{N}}c_1, \quad \quad \quad \lim_{\eta\to 0}A_2=-\frac{\dot{v}_2(0)}{\dot{\mu}_1N}=\frac{2}{\sqrt{N}}c_2
\end{equation}
where $c_1$ and $c_2$ are the coefficients such that $v_0(x,y)=\left(c_1\psi_{k,1}+c_2\psi_{2,2}\right)(x,y)$, featured in Section \ref{sec:Hadamard}. With respect to (\ref{eq:decomposition}), it is apparent that the main contribution towards $v$ comes from the terms with $A_1,A_2$, since these are the only coefficients that do not vanish for small $\eta$. 

\quad The bounds in (\ref{eq:close}) imply that for small $\eta$, $|\cos(\mu_nN)|$ is bounded below by some positive constant and $|\sin(\mu_nN)|$ is bounded above by $C\eta/N$. Therefore,
\begin{equation}\label{eq:hurray!}
    |v_n(0)|=\left|\frac{A_n\sin(\mu_nN)}{\cos(\mu_nN)}\right|\leq C\eta|A_n|/N
\end{equation}
for both $n=1,2$. Because $v$ is bounded, the coefficients $A_1, A_2$ are bounded independently of relevant parameters. However, we need more precise estimates to improve (\ref{eq:hurray!}) and prove Proposition \ref{prop:V}. To accomplish this, we first use a Lyapunov-Schmidt argument to quantitatively compare each $A_n$ to its limit in (\ref{eq:Alimit}). We then show that all of the higher coefficients $v_j(0)$ are bounded by a multiple of $\eta/N^{\frac{3}{2}}$, giving us control over the error. 

\begin{lemma}\label{lem:AD10}
    Let $(c_1, c_2)$ be as constructed in Proposition \ref{prop:Hadamard}. For any $\epsilon>0$, there exists $\eta_0(k,\epsilon)>0$ such that 
    \begin{equation}
        \left|A_1-\frac{2}{\sqrt{N}}c_1\right|+\left|A_2-\frac{2}{\sqrt{N}}c_2\right|\leq C\eta^{1-\epsilon}/N^{\frac{3}{2}} \nonumber
    \end{equation}
    for all $0<\eta\leq \eta_0(k,\epsilon)$. 
\end{lemma}

\begin{proof}[Proof of Lemma \ref{lem:AD10}.] Following a method presented in \cite{DH}, define $$L= h^*(\Delta_\Omega)h^{*-1}:H^2(R)\cap H_0^1(R)\to L^2(R)$$ where $h:R\to\Omega$ is described in (\ref{eq:diffeo}) and $h^*$ is the pullback by $h$. Throughout this proof, we use $||\cdot||$ to denote the operator norm unless otherwise stated. Explicitly, $L$ is the second-order differential operator 
\begin{equation}\label{eq:L}
    L=\p_y^2+\frac{N^2+\eta^2\phi'^2(N-x)^2}{(N+\eta\phi)^2}\p_x^2+\frac{2\eta\phi'(N-x)}{N+\eta\phi}\p_{xy}^2+\frac{\eta(N-x)}{(N+\eta\phi)^2}\left((N+\eta\phi)\phi''-2\eta\phi'^2\right)\p_x.
\end{equation}
From this expression, the operator $(L-\Delta)$ has norm $||(L-\Delta)||\leq C\eta$ and is almost self-adjoint in the sense that
\begin{equation}\label{eq:adjointalmost}
    \left|\int_R g_1(L-\Delta)g_2-\int_R g_2(L-\Delta)g_1\right|\leq C\eta ||g_1||_{L^2(R)}||g_2||_{H^1(R)}
\end{equation}
for any $g_j$ in the domain of $L$. Define the vector-valued function
\begin{equation}\label{eq:Fbig}
    \textbf{F}\left(u,\mu,\eta,\phi\right)=\left(F_1\left(u,\mu,\eta,\phi\right), F_2\left(u,\mu,\eta,\phi\right)\right)=\left( (L+\mu)u, \int_R u^2\det J_h-1\right)
\end{equation}
where $J_h$ is the Jacobian of $h$. If $\textbf{F}\left(u,\mu,\eta,\phi\right)=(0,0)$, then $v=h^{*-1}u$ satisfies (\ref{eq:P1v}). If $\phi_0\in\mathcal{A}$, we seek to understand the zero set of $\textbf{F}$ in a neighborhood of $\textbf{a}=\left(v_0, \lambda_{2,2}, 0, \phi_0\right)$. However, the Implicit Function theorem fails because $\ker (\Delta+\lambda_{2,2})=\textrm{span}\{\psi_{k,1},\psi_{2,2}\}$ and $D_{(u,\mu)}\textbf{F}(\textbf{a})$ is not invertible. To instead reduce this problem to a finite dimension, we employ Lyapunov-Schmidt. 

\quad Let $P$ denote the projection onto the kernel of $\left(\Delta+\lambda_{2,2}\right)$ and decompose $u=z+\zeta$ where $z=Pu$ and $\zeta=(I-P)u$. Applying this projection to the first component of (\ref{eq:Fbig}), we have
\begin{equation}\label{eq:pieces}
   \begin{cases} PF_1\big(z+\zeta, \mu, \eta, \phi\big)=0 \\ (I-P)F_1\big(z+\zeta, \mu, \eta, \phi\big)=0 \end{cases}
\end{equation}
The second equation in (\ref{eq:pieces}) can alternatively be written $(\Delta+\mu)\zeta=-(I-P)(L-\Delta)\big(z+\zeta\big)$. Because we are restricted away from the kernel of $(\Delta+\lambda_{2,2})$, the operator $(\Delta+\mu)$ is invertible for small $\eta$ and we can solve for $\zeta$ in terms of the projection $z$. For the sake of notation, denote
\begin{equation}
    T(\mu,h)=-(\Delta+\mu)^{-1}(I-P)(L-\Delta):H^2(R)\cap H_0^1(R)\to H^2(R)\cap H_0^1(R)  \nonumber
\end{equation}
so that $\zeta=(I-T)^{-1}Tz$. In particular, notice that $T$ is linear and depends on the map $h$ rather than its components $\eta$ and $\phi$ separately. For a given function $g$,
\begin{equation}
    T(\mu,h)g=\sum_{\mu_j\neq\lambda_{k,1},\lambda_{2,2}}\frac{1}{\mu_j-\mu}\psi_j\int_R\psi_j(L-\Delta)g \nonumber
\end{equation}
where $\psi_j$ satisfies $(\Delta+\mu_j)\psi_j=0$ with Dirichlet boundary conditions and $L^2$-normalization on the rectangle. For small enough $\eta$, the spectral gap is bounded below by a multiple of $1/N^2$, so we have the estimate $||T||\leq C\eta N^2$. Because $z=Pu$, we write $z=b_1\psi_{k,1}+b_2\psi_{2,2}$ so that
\begin{equation}
     u=\left(I+(I-T)^{-1}T\right)z=(I-T)^{-1}\left(b_1\psi_{k,1}+b_2\psi_{2,2}\right) \nonumber
\end{equation}
or $u=(I+S)z$ for $S=\sum_{j\geq 1} T^j$. The first equation in (\ref{eq:pieces}) then implies that $\textbf{M}\textbf{b}=\textbf{0}$ where
\begin{equation}
    \textbf{M}=\begin{pmatrix} \displaystyle\int_R \psi_{k,1}(L+\mu)(I+S)\psi_{k,1} & \displaystyle\int_R \psi_{k,1}(L+\mu)(I+S)\psi_{2,2} \\ \\ \displaystyle\int_R \psi_{2,2}(L+\mu)(I+S)\psi_{k,1} & \displaystyle\int_R \psi_{2,2}(L+\mu)(I+S)\psi_{2,2} \end{pmatrix} \nonumber
\end{equation}
and $\textbf{b}=(b_1, b_2)$. For $u$ to be nonzero, it is necessary that $\det \textbf{M}=0$. While $\lim_{\eta\to 0}\textbf{M}=\textbf{0}$, note that $\lim_{\eta\to 0}\eta^{-1}\textbf{M}=(\textbf{D}+\dot{\mu}I)$ where $\textbf{D}$ is the matrix featured in Proposition \ref{prop:Hadamard}. This is expected as $\lim_{\eta\to 0}\textbf{b}$ describes the coefficients in Section \ref{sec:Hadamard}. To get estimates on how $\textbf{M}$ compares to its limit, decompose
\begin{equation}
    \textbf{M}=\big(\mu-\lambda_{2,2}\big)I+\textbf{M}_1+\textbf{M}_2 \nonumber
\end{equation}
where
\begin{equation}
    \textbf{M}_1=\begin{pmatrix} \displaystyle\int_R\psi_{k,1}\big(L-\Delta)\psi_{k,1} & \displaystyle\int_R\psi_{k,1}\big(L-\Delta)\psi_{2,2} \\ \\ \displaystyle\int_R\psi_{2,2}\big(L-\Delta)\psi_{k,1} & \displaystyle\int_R\psi_{2,2}\big(L-\Delta)\psi_{2,2}\end{pmatrix} \nonumber
\end{equation}
and
\begin{equation}
    \textbf{M}_2=\begin{pmatrix} \displaystyle\int_R\psi_{k,1}\big(L-\Delta)S\psi_{k,1} & \displaystyle\int_R\psi_{k,1}\big(L-\Delta)S\psi_{2,2} \\ \\ \displaystyle\int_R\psi_{2,2}\big(L-\Delta)S\psi_{k,1} & \displaystyle\int_R\psi_{2,2}\big(L-\Delta)S\psi_{2,2}\end{pmatrix}. \nonumber
\end{equation}
By (\ref{eq:L}) and (\ref{eq:adjointalmost}), we can bound $$||\textbf{M}_1-\eta\textbf{D}||\leq C\eta^2 \quad \textrm{and} \quad ||\textbf{M}_2||\leq C\eta^2N^2.$$ 
\quad By Proposition \ref{prop:Hadamard} and (\ref{eq:Ass2}), the difference in eigenvalue variations satisfies $\left|\dot{\mu}-\dot{\gamma}\right|\geq c/N^2$ for some positive constant $c$. Setting the determinant of $\eta^{-1}\textbf{M}$ equal to zero then implies that $\mu=\lambda_{2,2}+\dot{\mu}\eta$ up to an error of size $C\eta^2N^2$. This allows us to bound
\begin{equation} \label{eq:MandD}
    \left|\left|\eta^{-1}\textbf{M}-\big(\textbf{D}+\dot{\mu}I\big)\right|\right|\leq C\eta N^2.
\end{equation}
To estimate $\textbf{b}$, we rewrite it in a new basis $\textbf{b}=B_1\textbf{c}+B_2\textbf{c}^\perp$ where $\textbf{c}=(c_1,c_2)$. By Proposition \ref{prop:Hadamard}, $$\big(\textbf{D}+\dot{\mu}I\big)\textbf{b}=B_2(\dot{\mu}-\dot{\gamma})\textbf{c}^\perp \quad \textrm{and therefore} \quad |B_2||\dot{\mu}-\dot{\gamma}|= \left|\left|(\textbf{D}+\dot{\mu}I)\textbf{b}\right|\right|_{\ell^2}.$$ Alongside (\ref{eq:MandD}), this implies $\left|B_2\right|\leq C\eta N^4$. The normalization $\int_R u^2\det J_h=1$ implies that $||\textbf{b}||_{\ell^2}^2=1$ up to an error of size $C\eta/N$. In combination with the bound on $|B_2|$, this means that $|B_1|$ is close enough to $1$ so that 
\begin{equation}
    ||\textbf{b}-\textbf{c}||_{\ell^2}\leq C\eta N^4. \nonumber
\end{equation}
\quad Finally, by the construction of $h$, if $v=h^{*-1}u$, $||v-u||_{C^0(R)}\leq C\eta$. Meanwhile, by (\ref{eq:coeffbound}) and (\ref{eq:hurray!}), $A_1=\frac{2}{\sqrt{N}}\int_R v\psi_{k,1}$ and $A_2=\frac{2}{\sqrt{N}}\int_R v\psi_{2,2}$ up to error of size $C\eta$, meaning
\begin{equation}
    \left|A_n-\frac{2}{\sqrt{N}}c_n\right|\leq \left|\left|v-u\right|\right|_{C^0(R)}+\frac{2}{\sqrt{N}}\left|b_n-c_n\right|\leq C\eta N^{\frac{7}{2}} \nonumber
\end{equation}
for $n=1,2$. The power of $N$ in this estimate is positive because both the spectral gap on the rectangle and the difference in eigenvalue variations are bounded below by a multiple of $1/N^2$. To compensate, we trade a slight loss $\eta\to \eta^{1-\epsilon}$ for the desired denominator in Lemma \ref{lem:AD10}. The estimate thus holds for $\eta>0$ up to a threshold $\eta_0$ that depends both on $k\sim N$ and the choice of $\epsilon>0$. 
\end{proof}

\begin{rem}\label{rem:eps}
    Lemma \ref{lem:AD10} serves as a limiting factor to the estimates in Proposition \ref{prop:V}. Namely, it requires we sacrifice slightly in $\eta$ in order to gain the denominator $N^{\frac{3}{2}}$, as any worse power in $N$ would limit the error estimate. This requires we limit the size of $\eta_0$ by both the integer $k\sim N$ and the loss $\epsilon$. The limitations on $\epsilon$ are discussed in Remark \ref{rem:Expl}. 
\end{rem} 

\quad When $k\geq 3$ and $\eta$ is sufficiently small, Lemma \ref{prop:coeff}, Lemma \ref{lem:AD10}, and Remark \ref{rem:3} imply that $$A_1\sim 1/N^{\frac{3}{2}} \quad \textrm{and} \quad A_2\sim 1/N^{\frac{1}{2}}.$$ This improves the bound in (\ref{eq:hurray!}) so that $|v_n(0)|\leq C\eta/N^{\frac{3}{2}}$ for both $n=1,2$. It provides slightly better decay in $N$ for $n=1$, but as we shall see through the following lemma, this does not gain us anything.

\begin{lemma}\label{lem:AllJ}
For any $\epsilon>0$, there exists $\eta_0(k,\epsilon)>0$ such that the following holds for $0<\eta\leq \eta_0(k,\epsilon)$:
\begin{equation}
    \sum_{j\geq 3} v_j(0)^2\leq C\eta^2/N^3. \nonumber
\end{equation}
\end{lemma}

\begin{proof}[Proof of Lemma \ref{lem:AllJ}.] We mimic the argument in \cite{BGM} to estimate the coefficients $v_j(0)$ to first order via a bootstrapping argument. For simplicity of notation, denote
\begin{equation}
    B^2=\sum_{j\geq 3}v_j(0)^2. \nonumber
\end{equation}
We first show that $|E(x,y)|\leq C(\eta/N^{\frac{3}{2}}+B)$ for all $(x,y)\in\Omega$, small enough $\eta$. From (\ref{eq:v_k}), we find $(\Delta+\mu)E=0$. Along the top, bottom, and right boundary components of $\Omega$, $E(x,0)=E(x,1)=E(N,y)=0$, and along the left boundary,
\begin{align}
    \left|E(-\eta\phi(y), y)\right|=\left|v_1(-\eta\phi(y))\sin(\pi y)+v_2(-\eta\phi(y))\sin(2\pi y)\right|\leq \left|v_1(-\eta\phi(y))\right|+\left|v_2(-\eta\phi(y))\right| \nonumber
    \\
    \leq\left|v_1(0)\right|+\left|A_1\right|\left|\sin(\mu_1\eta\phi(y))\right|+\left|v_2(0)\right|+\left|A_2\right|\left|\sin(\mu_2\eta\phi(y))\right|. \nonumber
\end{align}
Because $\phi\in\mathcal{A}$ and $\mu_1\sim 1$, $\mu_2\sim 1/N$, we have $\left|E\right|\leq C\eta/N^{\frac{3}{2}}$ along the boundary $\p\Omega$. Meanwhile, in the region $\Omega\cap\left\{0.1\leq x\leq N-0.1\right\}$, we can use (\ref{eq:coeffbound}) to write $|E|$ as a geometric series bounded by a multiple of $B$. In particular, this means that there exists a constant $\tilde{C}$ such that $|E(x,y)|\leq \tilde{C}(\eta/N^{\frac{3}{2}}+B)$ along the boundaries of $$\Omega_{L}\eqdef\Omega\cap\left\{x\leq 0.1\right\} \quad \textrm{and} \quad\Omega_{R}\eqdef\Omega\cap \left\{x\geq N-0.1\right\}.$$ To determine a bound over the interiors of $\Omega_{L}$ and $\Omega_{R}$, we follow \cite{GJ98}. With $\tilde{C}$ as above, define a comparison function
\begin{equation}
    S(x,y)= 8\tilde{C}\left(\eta/N^{\frac{3}{2}}+B\right)\cos\left(\frac{\pi}{4}\left(y-\frac{1}{2}\right)\right)\cos\left(4\pi  x\right) \nonumber
\end{equation}
so that $S>0$ on $\Omega_{L}$ for sufficiently small $\eta$ and $\Delta S=-\pi^2\left(16+\frac{1}{16}\right) S$. This means that $(\Delta+\mu)S<0$ in $\Omega_{L}$ and $S>\tilde{C}(\eta/N^{\frac{3}{2}}+B)\geq |E|$ on $\p\Omega_{L}$. By the generalized maximum principle,
\begin{equation}
    |E(x,y)|\leq S(x,y)\leq 8\tilde{C}\left(\eta/N^{\frac{3}{2}}+B\right), \quad \: \: (x,y)\in\Omega_{L}. \nonumber
\end{equation}
Repeating this argument for $S(x-N+0.1,y)$ over $\Omega_{R}$ allows us to conclude that $|E(x,y)|\leq C(\eta/N^{\frac{3}{2}}+B)$ over all of $\Omega$. With this established, we turn to bounding the coefficients $v_j(0)$. By definition,
\begin{equation}\label{eq:vJ}
    v_j(0)=2\int_0^1 v(0,y)\sin(j\pi y)dy= 2\int_{\p \Omega_0}v(x,y)\frac{\p}{\p \nu}\left(x\sin(j\pi y)\right)d\sigma
\end{equation}
where $\Omega_0\eqdef\Omega\cap\{x\leq 0\}$, $d\sigma$ is the measure on $\p\Omega_0$, and $\nu$ is the outward-pointing unit normal along the boundary. The last equality in (\ref{eq:vJ}) holds because $v$ vanishes along $\p\Omega$ and $\frac{\p}{\p\nu}=\p_x$ along the line $x=0$. By Green's Formula,
\begin{equation}
    2\int_{\p\Omega_0} v(x,y)\frac{\p}{\p \nu}\left(x\sin(j\pi y)\right)d\sigma=2\int_{\p\Omega_0} \frac{\p v}{\p\nu} x\sin(j\pi y)d\sigma-2(\mu-j^2\pi^2)\int_{\Omega_0} v(x,y)x\sin(j\pi y)dA. \nonumber
\end{equation}
Because the area of $\Omega_0$ is bounded by $\eta$, we can bound the second integral,
\begin{equation}\label{eq:GreenBound}
    \left|(\mu-j^2\pi^2)\int_{\Omega_0}v(x,y)x\sin(j\pi y)dA\right|\leq C\eta^2 j^2\left(\eta/N^{\frac{3}{2}}+B\right).
\end{equation}
Meanwhile, we break the first integral into pieces
\begin{equation}
    2\int_{\p\Omega_0} \frac{\p v}{\p\nu} x\sin(j\pi y)d\sigma=J_1+J_2+K \nonumber
\end{equation}
where
\begin{align}
    J_n\eqdef2\int_{\p\Omega_0} \frac{\p}{\p \nu}\left(v_n(x)\sin(n\pi y)\right)x\sin(j\pi y)d\sigma \nonumber
    \quad \textrm{and} \quad 
    K\eqdef2\int_{\p\Omega_0} \frac{\p E}{\p\nu}x\sin(j\pi y)d\sigma. \nonumber
\end{align}
Elliptic estimates, along with the bound on the error, allow us to bound $|K|\leq C\eta(\eta/N^{\frac{3}{2}}+B)$. In order to bound $J_n$, observe that we can express the normal vector along $x=-\eta\phi(y)$ as
\begin{equation}
    \nu=\left(1+\eta^2\phi'(y)^2\right)^{-\frac{1}{2}}\langle -1, -\eta\phi'(y)\rangle. \nonumber
\end{equation}
The integral $J_n$ vanishes along the line $x=0$, so it suffices to consider the left boundary curve $(-\eta\phi(y),y)$. For both $n=1,2$, the integral $J_n$ is equal to
\begin{align}\label{eq:415}
    2\int_0^1 v_n'(-\eta\phi(y))\sin(n\pi y)\eta\phi(y)\sin(j\pi y) dy+2\int_0^1 \eta\phi'(y)v_n(-\eta\phi(y))n\pi\cos(n\pi y)\eta\phi(y)\sin(j\pi y)dy.
\end{align}
To control the second term in (\ref{eq:415}), we consider the Fourier coefficient bound
\begin{equation}
    |v_n'(x)|\leq|\mu_n|\left(|A_n|+|v_n(0)|\right) \quad \textrm{which implies} \quad |v_n'(x)|\leq C/N^{\frac{3}{2}}. \nonumber
\end{equation}
The expressions in (\ref{eq:v1}) and (\ref{eq:v2}) also imply $\left|v_n(-\eta\phi(y))\right|\leq C\eta/N^{\frac{3}{2}}$ and therefore
\begin{equation}\label{eq:JJJJ}
    J_n=2\eta \int_0^1 v_n'(-\eta\phi(y))\sin(n\pi y)\phi(y)\sin(j\pi y)dy
\end{equation}
up to an error of size $C\eta^3/N^{\frac{3}{2}}$. Using integration by parts on (\ref{eq:JJJJ}), we can pull out a $j^{-1}$ factor and write 
\begin{equation}
    |J_1+J_2|\leq  C\eta(\eta^2+j^{-1})/N^{\frac{3}{2}}. \nonumber
\end{equation}
In combination with the bound on $|K|$ and (\ref{eq:GreenBound}), we have
\begin{equation}\label{eq:small j bound}
    |v_j(0)|\leq C\eta \left(\eta/N^{\frac{3}{2}}+B\right)+C\eta^2j^2\left(\eta/N^{\frac{3}{2}}+B\right)+C\eta j^{-1}/N^{\frac{3}{2}}. 
\end{equation}
This bound is helpful for small values of $j$, but we need a second bound to control the higher modes. If we use integration by parts on (\ref{eq:P1FCeq}), then alternatively
\begin{align}
    v_j(0)=-\frac{2}{j}\int_0^1 \left(v_1(0)\cos(\pi y)+2v_2(0)\cos(2\pi y)\right)\cos(j\pi y)dy-\frac{2}{j\pi}\int_0^1 \p_yE(0,y)\cos(j\pi y)dy \nonumber
\end{align}
and the gradient bound on $E$ allows us to write
\begin{equation}\label{eq:large j bound}
    |v_j(0)|\leq Cj^{-1}\left(\eta/N^{\frac{3}{2}}+B\right).
\end{equation}
Finally, to bound $B^2$, we apply (\ref{eq:small j bound}) to its truncated sum and (\ref{eq:large j bound}) to the remainder. Namely,
\begin{align}
    B^2= \sum_{ j\leq \eta^{-\frac{1}{2}}}|v_j(0)|^2+\sum_{j>\eta^{-\frac{1}{2}}}|v_j(0)|^2 \nonumber
    \\
    \leq C\eta^{-\frac{1}{2}}\left(\eta^{\frac{3}{2}}/N^{\frac{3}{2}}+\eta B\right)^2+C\left(\eta/N^{\frac{3}{2}}+B\right)^2\sum_{j>\eta^{-\frac{1}{2}}}j^{-2}. \nonumber
\end{align}
The final sum in $j$ is bounded by a multiple of $\eta^{\frac{1}{2}}$. Without loss of generality, suppose that $B\geq c\eta/N^{\frac{3}{2}}$ for some positive constant $c$ (otherwise the conclusion is trivial). Then,
\begin{equation}
    B^2\leq C\eta^{-1/2}\left(\eta^{\frac{3}{2}}/N^{\frac{3}{2}}+\eta B\right)^2+C\eta^{\frac{1}{2}} B^2\leq C\eta^{\frac{1}{2}}(\eta^2/N^{3}+B^2) \implies B^2\leq C\eta^2/N^3 \nonumber
\end{equation}
after bootstrapping the final term in the inequality and choosing $\eta$ sufficiently small. This implies that $|v_j(0)|\leq C\eta/N^{\frac{3}{2}}$ for all $j\geq 3$. 
\end{proof}

\quad We can use the results in Lemmas \ref{lem:AD10} and \ref{lem:AllJ} to estimate the behavior of the Fourier modes in (\ref{eq:P1v}). We find that the first two modes behave much like $\psi_{k,1}$ and $\psi_{2,2}$ respectively and the error is small.

\begin{prop}\label{prop:V}
Let $\epsilon>0$ and $\textbf{c}=(c_1,c_2)$ be as constructed in Proposition \ref{prop:Hadamard}. Under the same conditions as in Theorem \ref{prop:n=0}, there exists a constant $C_0>0$, dependent only on $\Lambda_\phi$, such that for each $k\geq 3$, there exists $\eta_0(k,\epsilon)>0$ such that the following holds for all $0<\eta\leq \eta_0(k,\epsilon)$:
\begin{enumerate}
    \item The first Fourier mode $v_1(x)$ satisfies the estimates
    \begin{equation}
        \left|v_1^{(\ell)}(x)-\frac{2c_1}{\sqrt{N}}\frac{d^\ell}{dx^\ell}\left(\sin\left(\frac{k\pi}{N}x\right)\right)\right|\leq C_0\eta^{1-\epsilon}/N^{\frac{3}{2}} \quad \textrm{for} \quad 0\leq \ell\leq 3. \nonumber
    \end{equation}
    \item The second Fourier mode $v_2(x)$ satisfies the estimates
    \begin{equation}
        \left|v_2^{(\ell)}(x)-\frac{2c_2}{\sqrt{N}}\frac{d^\ell}{dx^\ell}\left(\sin\left(\frac{2\pi}{N}x\right)\right)\right|\leq C_0\eta^{1-\epsilon}/N^{\ell+\frac{3}{2}} \quad \textrm{for} \quad 0\leq \ell\leq 3. \nonumber
    \end{equation}
    \item The error $E(x,y)$ satisfies the estimates
    \begin{equation}
        \left|E(x,y)\right|+\left|\nabla E(x,y)\right|\leq C_0\eta/N^{\frac{3}{2}} \quad \textrm{for} \quad (x,y)\in\Omega. \nonumber
    \end{equation}
    Further, for $\omega\in(0,1)$, there exists a constant $C(\omega)$ such that
    \begin{equation}
        \sup_{x\in\left[\frac{N}{4}, \frac{3N}{4}\right]}\left|\nabla^\ell E(x,y)\right|\leq C_0\eta e^{-cN}, \quad \sup_{y\in[\omega, 1-\omega]}\left|\nabla^\ell E(x,y)\right|\leq C(\omega)\eta/N^{\frac{3}{2}} \quad \textrm{for} \quad 0\leq \ell \leq 3. \nonumber
    \end{equation}
\end{enumerate}
\end{prop}

\quad This proposition allows us to compare the restriction $v|_R$ to the expression $v_0$ determined in Section \ref{sec:Hadamard}. The goal of Sections \ref{sec:centerEven} - \ref{sec:bdry} is to extend these function estimates to estimates on the nodal set. In item $(iii)$, we momentarily leave $\omega$ arbitrary; in Section \ref{sec:bdry} we choose a value $\omega\sim 1$ for a particular application.  

\begin{proof}[Proof of Proposition \ref{prop:V}.]
To establish items $(i)$ and $(ii)$, consider the expressions in (\ref{eq:v1}), (\ref{eq:v2}) for $v_1(x)$, $v_2(x)$. The estimates in (\ref{eq:close}), (\ref{eq:hurray!}) and Lemma \ref{lem:AD10} provide the pointwise bound. The same argument holds for the derivatives, although for $v_2(x)$, each iteration of chain rule pulls out another $N$ in the denominator. 

\quad Alongside Lemma \ref{lem:AllJ}, equation (\ref{eq:coeffbound})  implies the bound $||E||_{L^\infty(\Omega)}+||E||_{L^2(\Omega)}\leq C\eta/N^{\frac{3}{2}}$. Elliptic estimates, as in \cite{GV}, allow us to extend this bound to $||E||_{W^{1,\infty}(\Omega)}$. Further, for $x\in\left[\frac{N}{4}, \frac{3N}{4}\right]$, we can use (\ref{eq:v_k}) and Lemma \ref{lem:AllJ} to get exponential decay in $N$, which will be important for the analysis near the intersection point of $v_0^{-1}(0)$. For $y\in[\omega, 1-\omega]$ with $\omega>0$, we are away from the corners of $\Omega$ and can use the boundary regularity as in \cite{BGM} to establish item $(iii)$ in Proposition \ref{prop:V}.
\end{proof}

\quad With Proposition \ref{prop:V} established, we can use details in the proof for Lemma \ref{lem:AllJ} to determine a more precise characterization of the Fourier coefficients $v_j(0)$.

\begin{prop} \label{prop:VJ} Under the same conditions as in Proposition \ref{prop:V}, the following holds for all $j\geq 1$:
\begin{equation}
    \left|v_j(0)-\frac{4\pi\eta}{N^{\frac{3}{2}}}\int_0^1\phi(y)\left(kc_1\sin(\pi y)+2c_2\sin(2\pi y)\right)\sin(j\pi y)dy\right|\leq C_0\left(\eta^3j^2/N^{\frac{3}{2}}+\eta^{2-\epsilon}/N^{\frac{3}{2}}\right). \nonumber
\end{equation}
\end{prop}

\quad This proposition gives us good control over the Fourier modes when $j$ is small, which will play a role in determining the nodal behavior of $v$ near the intersection point featured in item (i) of Theorem \ref{prop:n=0}. 

\begin{proof}[Proof of Proposition \ref{prop:VJ}.] Items $(i)-(ii)$ of Proposition \ref{prop:V} imply that for all $y\in[0,1]$,
\begin{equation}\label{eq:derivativesApprox}
    \left|v_1'(-\eta\phi(y))-\frac{2k\pi}{N^{\frac{3}{2}}}c_1\right|+\left|v_2'(-\eta \phi(y))-\frac{4\pi}{N^{\frac{3}{2}}}c_2\right|\leq C\eta^{1-\epsilon}/N^{\frac{3}{2}}.
\end{equation}
Further, if we combine (\ref{eq:JJJJ}) with (\ref{eq:derivativesApprox}), then 
\begin{equation}
    J_1=\frac{4k\pi\eta }{N^{\frac{3}{2}}}c_1\int_0^1 \phi(y)\sin(\pi y)\sin(j\pi y) dy \quad \textrm{and} \quad J_2=\frac{8\pi\eta}{N^{\frac{3}{2}}}c_2\int_0^1\phi(y)\sin(2\pi y)\sin(j\pi y)dy \nonumber
\end{equation}
up to an error of size $C\eta^{2-\epsilon}/N^{\frac{3}{2}}$. Combining all of this together, we get $v_j(0)=J_1+J_2$ up to an error of size $C(\eta^3j^2/N^{\frac{3}{2}}+\eta^{2-\epsilon}/N^{\frac{3}{2}})$. This establishes Proposition \ref{prop:VJ}; observe that by (\ref{eq:VarCoeff}), the expression $J_1+J_2$ is the linear term in the Taylor series for $v_j(0)$ as a function of $\eta$. 
\end{proof}

%%%%%%%%%%%%%%%%%%%%%%%%%%%%%%%%%%%%%%%%%%%%%%%%%%%
\section{Description of the Nodal Set Near the Center (k even)}\label{sec:centerEven}

\quad In this section, we focus on the nodal set of $v$, the eigenfunction corresponding to the upper branch $\mu$, when $k\geq 4$ is even and prove the first item of both Theorem \ref{thm:VevenPt} and Theorem \ref{thm:VevenDer}. As stated in Theorem \ref{prop:n=0}, the nodal set of $v_0$ on the rectangle $R(N)$ features a crossing at $\textbf{z}\eqdef\left(\frac{N}{2}, \overline{y}\right)$ with $\overline{y}$ as in (\ref{eq:ybar}). Our analysis takes place in the open disk $D_{r}(\textbf{z})$ with $r=2\eta^{\frac{1}{2}-\frac{2}{3}\epsilon}$ for $\epsilon\in\left(0,\frac{1}{4}\right)$; Remark \ref{rem:Radius} discusses the purpose of this construction. 

\quad In \cite{C76}, Cheng showed that, under a coordinate transformation, Laplacian eigenfunctions can be described locally by harmonic, homogeneous polynomials. However, we use a different approach, relying on Proposition \ref{prop:V} to show that it suffices to study the polynomial
\begin{equation}\label{eq:Poly}
    P(x,y)\eqdef\p_{xy}^2\left(c_1\psi_{k,1}+c_2\psi_{2,2}\right)_\textbf{z}\left(x-\frac{N}{2}\right)\left( y-\overline{y}\right)+\p_x\left(v-E\right)_\textbf{z}\left(x-\frac{N}{2}\right)+\p_y\left(v-E\right)_\textbf{z}\left(y-\overline{y}\right)+v(\textbf{z})
\end{equation}
which closely resembles the eigenfunction $v$ in $D_r(\textbf{z})$. We then establish properties of the zero set of $P(x,y)$.

\begin{rem}\label{rem:Radius}
    The radius $r=2\eta^{\frac{1}{2}-\frac{2}{3}\epsilon}$ with $\epsilon\in\left(0,\frac{1}{4}\right)$ is chosen so that $\eta^{\frac{1}{2}}\ll r$ and $r \ll \eta^{\frac{1}{3}}$ as $\eta\to 0$. The lower bound is necessary so that $D_r(\textbf{z})$ is large enough to capture the hyperbolic nature of $v^{-1}(0)$. The upper bound is necessary to ensure that the remainder in Proposition \ref{prop:propRemainder} is small enough. As long as $r\sim \eta^{\frac{1}{2}-\delta}$ for $\delta\in\left(\frac{\epsilon}{2},\frac{1}{6}\right)$, our analysis holds; for ease of notation, we set $\delta=\frac{2}{3}\epsilon$.
\end{rem}

\quad The polynomial in (\ref{eq:Poly}) features the largest terms from the Taylor expansion of $v$ near $\textbf{z}$, but to ensure that it approximates the local behavior of $v$, we first need the following bound on the remainder. 

\begin{prop} \label{prop:propRemainder}
Let $\epsilon\in\left(0,\frac{1}{4}\right)$. Under the same conditions as in Theorem \ref{prop:n=0}, there exists a constant $C_0>0$, dependent only on $\Lambda_\phi$, such that for each $k\geq 4$ even, there exists $\eta_0(k,\epsilon)>0$ such that the remainder $R(x,y)=v(x,y)-P(x,y)$ satisfies
\begin{equation}
    \sup_{(x,y)\in D_r(\textbf{z})}\left(|R(x,y)|+\eta^{\frac{1}{2}-\frac{2}{3}\epsilon}|\nabla R(x,y)|\right)\leq C_0\eta^{\frac{3}{2}-2\epsilon}/N^{\frac{1}{2}}. \nonumber
\end{equation}
for $0<\eta\leq \eta_0(k,\epsilon)$.
\end{prop}

\quad Proposition \ref{prop:propRemainder} tells us that the polynomial in (\ref{eq:Poly}) is indeed a good local approximation for $v(x,y)$ near $\textbf{z}$ and worth studying further.  In order to better understand $v^{-1}(0)$ in $D_r(\textbf{z})$, we first study properties of the zero set of (\ref{eq:Poly}). Then we show that for small enough perturbations, the same properties hold for the nodal set of $v$.

\begin{proof}[Proof of Propositions \ref{prop:propRemainder}.] If we write 
\begin{equation}
    v(x,y)=\left(c_1\psi_{k,1}+c_2\psi_{2,2}\right)(x,y)+G(x,y) \nonumber
\end{equation}
for a difference function $G$, then we can use Proposition \ref{prop:V} to bound $$\sup_{(x,y)\in D_r(\textbf{z})}\left|\nabla^\ell G(x,y)\right|\leq C\eta^{1-\epsilon}/N^{\frac{3}{2}}$$ for all $\ell\leq 3$. By expanding the unperturbed eigenfunction $v_0$ near $\textbf{z}$, we have 
\begin{equation}
    \left(c_1\psi_{k,1}+c_2\psi_{2,2}\right)(x,y)= \p_{xy}^2\left(c_1\psi_{k,1}+c_2\psi_{2,2}\right)\big|_\textbf{z}\left(x-\frac{N}{2}\right)\big(y-\overline{y}\big) \nonumber
\end{equation}
up to an error of size $C\eta^{\frac{3}{2}-2\epsilon}/N^{\frac{1}{2}}$. This provides the first term in the model polynomial (\ref{eq:Poly}), but in order to understand how the nodal crossing turns into a hyperbola, we need to include some lower-order terms from $G(x,y)$. Because $v_0(\textbf{z})=\nabla v_0(\textbf{z})=0$,
\begin{equation}
    G(x,y)=v(\textbf{z})+\nabla v(\textbf{z})\cdot \left((x,y)-\textbf{z}\right) \nonumber
\end{equation}
up to an error of size $C\eta^{2-\frac{7}{3}\epsilon}/N^{\frac{3}{2}}$.
If we define $P(x,y)$ as in (\ref{eq:Poly}), rolling the error terms of $v$ into the remainder and taking $\epsilon$ small enough, then we have the pointwise bound in Proposition \ref{prop:propRemainder}. The gradient bound follows similarly from Proposition \ref{prop:V} and the choice of the radius $r\sim\eta^{\frac{1}{2}-\frac{2}{3}\epsilon}$. 
\end{proof}

\quad Because $v$ locally resembles (\ref{eq:Poly}), it suffices to establish properties of the model polynomial in $D_r(\textbf{z})$. We find that the set $P(x,y)=0$ is a hyperbola centered close to $\textbf{z}$ with the two branches separated by a distance $\sim\eta^{\frac{1}{2}}$.

\begin{prop} \label{prop:PolyStuff} Under the same conditions as in Proposition \ref{prop:propRemainder}, there exist constants $C_0,k_0>0$, dependent only on $\Lambda_\phi$, such that for each even $k\geq k_0$, there exists $\eta_0(k,\epsilon)$ so that for $0<\eta\leq \eta_0(k,\epsilon)$, the zero set of the model polynomial $P(x,y)=v(x,y)-R(x,y)$ is a hyperbola satisfying the following properties:
\begin{enumerate}
    \item The center is located at $(x_c,y_c)$ where $\left|x_c-\frac{N}{2}\right|\leq C_0\eta$ and $\left|y_c-\overline{y}\right|\leq C_0\eta^{1-\epsilon}$. The hyperbola can alternatively be written
    \begin{equation}
        (x-x_c)(y-y_c)=D \nonumber
    \end{equation}
    where $D$ satisfies $C_0^{-1}\eta\leq |D|\leq C_0\eta$.
    
    \item The vertices of the hyperbola are separated by a distance $\omega$ satisfying $C_0^{-1}\eta^{\frac{1}{2}}\leq|\omega|\leq C_0\eta^{\frac{1}{2}}$. 
    
    \item The principal axis makes an angle of $\theta=\pm \frac{\pi}{4}$ with the positive $x$-axis. The precise orientation satisfies $\sgn\theta=\sgn v(\textbf{z})$.
\end{enumerate}
\end{prop}

\begin{proof}[Proof of Propositions \ref{prop:PolyStuff}.] For simplicity of notation, denote
\begin{align}
    \alpha=\p_{xy}^2 \left(c_1\psi_{k,1}+c_2\psi_{2,2}\right)_\textbf{z}, \quad \quad \beta=\p_x\left(v-E\right)_\textbf{z}, \quad \quad \upsilon=\p_y\left(v-E\right)_\textbf{z}, \quad \quad \rho=v(\textbf{z}) \nonumber
\end{align}
so that
\begin{equation}
    P(x,y)=\alpha\left(x-\frac{N}{2}\right)\left(y-\overline{y}\right)+\beta \left(x-\frac{N}{2}\right)+\upsilon\big(y-\overline{y}\big)+\rho \nonumber
\end{equation}
by (\ref{eq:Poly}). This expression describes a hyperbola so long as both $|\alpha|$ and $|\beta\upsilon-\rho\alpha|$ are positive. By evaluating at the point of intersection, $\nabla(c_2\psi_{k,1}+c_2\psi_{2,2})|_\textbf{z}=\textbf{0}$. From Proposition \ref{prop:V}, the following bounds are then immediate:
\begin{equation}
    \big|\beta\big|\leq C\eta^{1-\epsilon}/N^{\frac{3}{2}}, \quad |\upsilon|\leq C\eta/N^{\frac{3}{2}}, \quad |\rho|\leq C\eta/N^{\frac{3}{2}}. \nonumber
\end{equation}
Meanwhile, the quantity $|\alpha|$ is comparable to $1/N^{\frac{3}{2}}$ by (\ref{eq:ybar}) and Lemma \ref{prop:coeff}. The final estimate needed is a lower bound on $\rho$. To establish this, consider the related quantity $|\rho-E(\textbf{z})|$, which we can simplify using (\ref{eq:v1}), (\ref{eq:v2}), and (\ref{eq:ybar}). 
\begin{equation}
    \left|\rho-E(\textbf{z})\right|=\frac{1}{2}\sin(\pi\overline{y})\left| \cos\left(\frac{k\pi}{2}\right)v_1(0)\sec\left(\frac{\mu_1N}{2}\right)+ \frac{kc_1}{2c_2}v_2(0)\sec\left(\frac{\mu_2N}{2}\right)\right|. \nonumber
\end{equation}
By (\ref{eq:close}), we can bound $|\cos(\mu_nN/2)|$ below by a constant for $n=1,2$, small enough $\eta$. In combination with (\ref{eq:hurray!}), this implies
\begin{equation}\label{eq:needsAremark}
    |\rho-E(\textbf{z})|=\frac{1}{2}\sin(\pi\overline{y})\left| v_1(0)-\frac{kc_1}{2c_2}v_2(0)\right|
\end{equation}
up to an error of size $C\eta^2/N^{\frac{5}{2}}$. Proposition \ref{prop:VJ} gives us an estimate for each $v_n(0)$ and, in combination with (\ref{eq:Ass2}), we can conclude that $$\big|2c_2v_1(0)-kc_1v_2(0)\big|\geq c\eta/N^{\frac{3}{2}}.$$ Therefore, because $c_2\sim 1$, the difference $|\rho-E(\textbf{z})|$ is bounded below by a multiple of $\eta/N^{\frac{3}{2}}$ for small enough $\eta$. Further, because $D_r(\textbf{z})\subset \Omega\cap\{x\in\left[\frac{N}{4}, \frac{3N}{4}\right]\}$, we can use the exponential error bound from item $(iii)$ of Proposition \ref{prop:V} to write
\begin{equation}
    |\rho|\geq c\eta/N^{\frac{3}{2}}-C\eta e^{-cN} \quad \textrm{which implies} \quad |\rho|\geq c\eta/N^{\frac{3}{2}} \nonumber
\end{equation}
and therefore $\rho\sim \eta/N^{\frac{3}{2}}$ for sufficiently large $N$ (i.e. sufficiently large $k$). With the necessary bounds in place, it is now apparent that $P(x,y)=0$ indeed describes a hyperbola. The center occurs at $(x_c, y_c)=\textbf{z}-\frac{1}{\alpha}\big(\upsilon, \beta\big)$ or
\begin{equation}
    x_c-\frac{N}{2}=-\frac{\upsilon}{\alpha}, \quad \textrm{and} \quad y_c-\overline{y}=-\frac{\beta}{\alpha} \nonumber
\end{equation}
giving us the bounds in item ($i$) of Proposition \ref{prop:PolyStuff}. We can rewrite the level set $P(x,y)=0$ as 
\begin{equation}
    (x-x_c)(y-y_c)=D, \quad \textrm{with} \quad D\eqdef \frac{\beta\upsilon-\rho\alpha}{\alpha^2} \nonumber
\end{equation}
which is a hyperbola where the principal axis makes an angle of $\theta=\pm \frac{\pi}{4}$ with the positive $x$-axis. Namely, the sign of this angle matches the sign of $D$. With this knowledge, we can rotate by $\theta$ to obtain the isometry
\begin{equation}
    \{\textbf{x}:P(\textbf{x})=0\}\cong\{(x, y):x^2-y^2=2D\}
\end{equation}
which allows us to write the distance separating the vertices of our hyperbola in terms of the coefficients, $\sqrt{8|D|}$. To bound this quantity, we compare it to the value $\rho= v(\textbf{z})$. 
\begin{equation}
    \left| |D|-\left|\frac{\rho}{\alpha}\right|\right|\leq\left|D+\frac{\rho}{\alpha}\right|=\left|\frac{\beta\upsilon}{\alpha^2}\right|\leq C\eta^{2-\epsilon}. \nonumber
\end{equation}
Fortunately, $\rho/\alpha\sim \eta$, which implies that $D\sim \eta$. Lastly, note that by construction of $\alpha$, $\sgn(D)=\sgn(\rho)$ and therefore the orientation of the hyperbola is determined by the sign of $v(\textbf{z})$. 
\end{proof}

\begin{rem}\label{rem:BananaDoubleSplit}
    The key to bounding (\ref{eq:needsAremark}) lies in Proposition \ref{prop:VJ} and the condition $\Lambda_\phi\neq 0$. The size of the nodal splitting described in item $(ii)$ of Proposition \ref{prop:PolyStuff} thus depends explicitly on the value of $\Lambda_\phi$.
\end{rem}

\quad Now that Propositions \ref{prop:propRemainder} and \ref{prop:PolyStuff} are proved, we aim to translate the properties of the zero set of $P(x,y)$ to the nodal set of $v(x,y)$. Fix $k$ in the range $k\geq k_0$. Via an isometry, there exist coordinates $(x', y')$ such that
\begin{equation}
    v(x', y')= 2D-(x')^2+(y')^2+R'(x', y') \nonumber
\end{equation}
with $R'$ satisfying the estimate in Proposition \ref{prop:propRemainder} in $(x',y')$ coordinates. Without loss of generality, suppose that $D>0$. We hereafter work in these coordinates, so we drop the primes. To better understand the nodal set of $v$, we write the zero set of the model polynomial locally as a graph $x(y)$ and study the perturbation $w(y)$ to the full nodal set. For $R'=0$, the right branch of the resulting level set can be written as the graph
\begin{equation} \label{eq:x(y)}
    x(y)=\sqrt{2D+y^2} \quad \textrm{with} \quad x(y)\geq C_1\eta^{\frac{1}{2}}
\end{equation}
where $C_1$ is some positive constant. If we denote $\tilde{x}(y)=x(y)+w(y)$ as the corresponding branch of the nodal set of $v$, then by (\ref{eq:x(y)}), the set $v(\tilde{x}, y)=0$ can be alternatively written
\begin{equation}\label{eq:W}
    (2D-x^2+y^2)-w^2-2xw+R'(x+w, y)=0 \quad \textrm{or equivalently} \quad w^2+2xw=R'(x+w, y).
\end{equation}
Our goal is to obtain bounds on $|w(y)|$. First we show that for each $y$ satisfying $|(x(y), y)|\leq \eta^{\frac{1}{2}-\frac{2}{3}\epsilon}$, it holds that $|(\tilde{x}(y), y)|\leq 2\eta^{\frac{1}{2}-\frac{2}{3}\epsilon}$. To accomplish this, fix $y_0$ such that $|(x(y_0), y_0)|\leq \eta^{\frac{1}{2}-\frac{2}{3}\epsilon}$ and denote
\begin{equation}\label{eq:f(y,w)}
    g_{y_0}(w)=w^2+2xw-R'(x+w, y_0).
\end{equation}
Then, because $x(y)>0$,
\begin{equation}
    g_{y_0}\big(\eta^{\frac{1}{2}-\frac{2}{3}\epsilon}\big)=\eta^{1-\frac{4}{3}\epsilon}+2\eta^{\frac{1}{2}-\frac{2}{3}\epsilon}x-R'\big(x+\eta^{\frac{1}{2}-\frac{2}{3}\epsilon}, y_0\big) \geq \eta^{1-\frac{4}{3}\epsilon}-\sup_{(x,y)\in D_r(\textbf{0})} \left|R'(x,y)\right|\geq \eta^{1-\frac{4}{3}\epsilon}-C\eta^{\frac{3}{2}-2\epsilon} \nonumber
\end{equation}
which is positive for sufficiently small $\eta$. Here we have used the fact that $|x(y_0)|<\eta^{\frac{1}{2}-\frac{2}{3}\epsilon}$ so $$|x(y_0)+\eta^{\frac{1}{2}-\frac{2}{3}\epsilon}|<2\eta^{\frac{1}{2}-\frac{2}{3}\epsilon}=r$$ and Proposition \ref{prop:propRemainder} controls the remainder term. Next, 
\begin{equation}
    g_{y_0}\left(-\frac{1}{2}C_1\eta^{\frac{1}{2}}\right)=\frac{1}{4}C_1^2\eta-C_1x\eta^{\frac{1}{2}}-R'\left(x-\frac{1}{2}C_1\eta^{\frac{1}{2}}, y_0\right)\leq -\frac{3}{4}C_1^2\eta +\sup_{(x,y)\in D_r(\textbf{0})}\left|R'(x,y)\right| \nonumber
\end{equation}
because $x(y_0)>C_1\eta^{\frac{1}{2}}$. By the bounds on $|R|$, this quantity is negative for small enough $\eta$. By the intermediate value theorem, there exists some $w\in \big(-\frac{1}{2}C_1\eta^{\frac{1}{2}}, \eta^{\frac{1}{2}-\frac{2}{3}\epsilon}\big)$ such that $g_{y_0}(w)=0$, i.e. for fixed $y_0$, there exists a value $w(y_0)$ satisfying (\ref{eq:W}). Therefore,
\begin{equation}
    \tilde{x}(y)=x(y)+w(y)> C_1\eta^{\frac{1}{2}}-\frac{1}{2}C_1\eta^{\frac{1}{2}}=\frac{1}{2}C_1\eta^{\frac{1}{2}} \nonumber
\end{equation}
and $|(\tilde{x}(y),y)|< 2\eta^{\frac{1}{2}-\frac{2}{3}\epsilon}$. Plugging the expression (\ref{eq:x(y)}) for $x(y)$ into the equation $v=0$ yields
\begin{equation}
    0=(x^2-y^2)-\tilde{x}^2+y^2+R'(\tilde{x},y) \quad \textrm{or equivalently} \quad \tilde{x}^2-x^2=R'(\tilde{x}, y). \nonumber
\end{equation}
Then,
\begin{equation}
    w(\tilde{x}+x)=(\tilde{x}-x)(\tilde{x}+x)=R'(\tilde{x},y)\quad \textrm{which implies} \quad |w|=\bigg|\frac{R'(\tilde{x}, y)}{\tilde{x}+x}\bigg|\leq C\eta^{1-2\epsilon}. \nonumber
\end{equation}
Because $\epsilon<\frac{1}{4}$, this bound is sufficient in showing that the perturbation $w(y)$ is too small to close the space between $x(y)$ and the other hyperbola branch. Alongside item $(i)$ of Proposition \ref{prop:PolyStuff}, this proves item $(i)$ of Theorem \ref{thm:VevenPt}.

\quad Next we study the regularity of $\tilde{x}$. As constructed, $\tilde{x}'(y)=x'(y)+w'(y)$, so we first bound the derivative of the model polynomial's level set and then the perturbation. From (\ref{eq:x(y)}), we can compute explicitly
\begin{equation}
    x'(y)=\frac{y}{\sqrt{2D+y^2}} \implies |x'(y)|\leq 1.
\end{equation}
To bound $|w'(y)|$, we employ the Implicit Function Theorem. If we redefine the function from (\ref{eq:f(y,w)}) in terms of $(y,w)$ as follows:
\begin{equation}
    g(y,w)=w^2+2w\sqrt{2D+y^2}-R'\left(w+\sqrt{2D+y^2}, y\right) \nonumber
\end{equation}
then so long as $\p_wg\neq 0$, the local graph $w(y)$ has a derivative given by
\begin{equation}\label{eq:coolname}
    |w'(y)|=\left|\frac{\p_yg(y, w(y))}{\p_wg(y, w(y))}\right|. 
\end{equation}
Let $(y_0, w_0)$ be a point in the set $g^{-1}(0)$. Then,
\begin{align}
    \p_wg(y_0,w_0)=2\left(w_0+\sqrt{2D+y_0^2}\right)-\p_xR'\left(\tilde{x}(y_0), y_0\right) \nonumber
    \\
    \left|\p_wg(y_0, w_0)\right|\geq 2\left|2D+y_0^2\right|^{\frac{1}{2}}-2|w_0|-\left|\p_xR'(\tilde{x}(y_0), y_0)\right|\geq 2\sqrt{2D}-2|w_0|-\left|\p_xR'(\tilde{x}(y_0),y_0)\right|. \nonumber
\end{align}
Using the fact that $D\sim\eta$, $|w|\leq C\eta^{1-2\epsilon}$, and $|\nabla R'|\leq C\eta^{1-\frac{4}{3}\epsilon}$ for $\epsilon<\frac{1}{4}$, we can bound $|\p_wg(y_0, w_0)|\geq c\eta^{\frac{1}{2}}$ for small enough $\eta$. This implies that (\ref{eq:coolname}) holds and by bounding
\begin{align}
    \p_yg(y_0,w_0)=\frac{y_0}{\sqrt{2D+y_0^2}}\left(2w_0-\p_xR'(\tilde{x}(y_0), y_0)\right)-\p_yR'(\tilde{x}(y_0), y_0) \nonumber
    \\
    \left|\p_y g(y_0, w_0)\right|\leq C\left(|w_0|+\left|\nabla R'(\tilde{x}(y_0), y_0)\right|\right)\leq C\eta^{1-2\epsilon}, \nonumber
\end{align}
we conclude that $|\tilde{x}'(y)|$ is bounded in $D_r(\textbf{z})$. This proves item $(i)$ of Theorem \ref{thm:VevenDer}.

%%%%%%%%%%%%%%%%%%%%%%%%%%%%%%%%%%%%%%%%%%%%%%%%%%%%%
\section{Description of the Nodal Set Away from the Center (k even)}\label{sec:awayEven}

\quad In this section, we study the nodal set of $v$ away from the center and away from the left and right boundaries, which will require a more delicate analysis in Section \ref{sec:bdry}. For our purposes, let $\widetilde{\Omega}=\widetilde{\Omega}_\phi(\eta, N)$ denote the region in $\Omega_\phi(\eta, N)$ outside a square with side length $r=2\eta^{\frac{1}{2}-\frac{2}{3}\epsilon}$ centered at $\textbf{z}=\big(\frac{N}{2},\overline{y}\big)$ and with $\eta^{\frac{1}{3}(1-\epsilon)}<x<N-\eta^{\frac{1}{3}(1-\epsilon)}$. By the previous section, the nodal set of $v$ exits the center square along four curves, one out of each side of this square. We aim to describe how these curves continue as they approach the boundaries. 
\begin{figure}[H]
    \centering
    \includegraphics[scale=0.5]{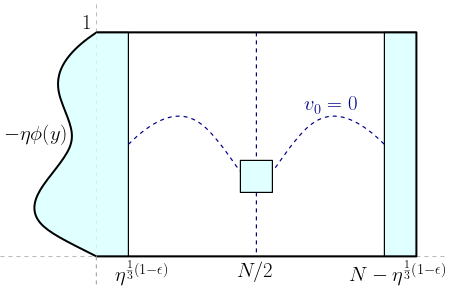}
    \caption{The Subdomain $\widetilde{\Omega}$.}
    \label{fig:tildeOmega}
\end{figure}

\begin{rem}
    The domain $\widetilde{\Omega}$ is designed both to complement the analysis in Section \ref{sec:centerEven} and to optimize the estimates in Theorem \ref{thm:VevenPt}. In particular, we expect our estimates to suffer a loss in $\eta$ in regions close to the left and right boundaries of $\Omega$. Taking $\eta^{\frac{1}{3}(1-\epsilon)}<x<N-\eta^{\frac{1}{3}(1-\epsilon)}$ in $\widetilde{\Omega}$ minimizes this loss. See Remark \ref{rem:lineSplit} for details.
\end{rem}

\quad As in (\ref{eq:onRect}), we let $v_0=\lim_{\eta\to 0}v$ so that by (\ref{eq:fv}), the level set of the eigenfunction on the unperturbed rectangle $R$ can be written $$v_0^{-1}(0)=\left\{ y=f_v(x)\right\} \cup \left\{x=\frac{N}{2}\right\},$$ as pictured in Figure \ref{fig:tildeOmega}. By Proposition \ref{prop:V},
\begin{equation}
    \left|v(x,y)-v_0(x,y)\right|\leq C\eta^{1-\epsilon}/ N^{\frac{3}{2}} \nonumber
\end{equation}
so if $(x_0, y_0)\in v^{-1}(0)\cap R$, then $\big| v_0(x_0, y_0)\big|\leq C\eta^{1-\epsilon}/N^{\frac{3}{2}}$. Since we have an explicit expression for $v_0$ from Section \ref{sec:Hadamard}, we can estimate the location of such points $(x_0, y_0)$. We decompose $\widetilde{\Omega}$ into different regions and perform our analysis locally. For even $k\geq 4$, the nodal set for $v$ looks very much like that of $v_0$ in $\widetilde{\Omega}$, as made precise in the following statements.

\begin{prop}\label{prop:horizontal}
Under the same conditions as in Theorem \ref{prop:n=0}, there exists a constant $C_0>0$, dependent only on $\Lambda_\phi$, such that for all $k\geq 4$ even, there exists $\eta_0(k,\epsilon)>0$ such that for $0<\eta\leq \eta_0(k,\epsilon)$, the nodal set of $v$ has the following property: If $(x_0, y_0)\in v^{-1}(0)$ is a point in $\widetilde{\Omega}_\phi(\eta, N)$ with $\left|y_0-\frac{1}{2}\right|<\frac{1}{2}\max\{\overline{y}, 1-\overline{y}\}$ and $\left|x_0-\frac{N}{2}\right|\geq 2\eta^{\frac{1}{2}-\frac{2}{3}\epsilon}$, then there exists an open neighborhood $U_1$ containing $x_0$ and a function $h_1(x)$ such that $v^{-1}(0)\cap U_1=\{(x, h_1(x))\}$. Further,
\begin{enumerate}
    \item  If $0.1<\left|x_0-\frac{N}{2}\right|<\frac{N}{2}-0.1$, then for all $x\in U_1$,
\begin{equation}
    \left|h_1(x)-f_v(x)\right|+\left|h_1'(x)-f_v'(x)\right|\leq \frac{C_0\eta^{1-\epsilon}}{N\left|\sin\left(\frac{2\pi}{N}x\right)\right|}. \nonumber
\end{equation}
\item If $2\eta^{\frac{1}{2}-\frac{2}{3}\epsilon}\leq \left|x_0-\frac{N}{2}\right|\leq 0.1$, then for all $x\in U_1$,
\begin{equation}
    \left|h_1(x)-f_v(x)\right|\leq \frac{C_0\eta^{1-\epsilon}}{\left|x-\frac{N}{2}\right|}, \quad \quad \left|h_1'(x)-f_v'(x)\right|\leq \frac{C_0\eta^{1-\epsilon}}{\left|x-\frac{N}{2}\right|^2}. \nonumber
\end{equation}
\item If $\eta^{\frac{1}{3}(1-\epsilon)}\leq \left|x_0-tN\right|\leq 0.1$ with $t\in\{0,1\}$, then for all $x\in U_1$,
\begin{equation}
    \left|h_1(x)-f_v(x)\right|\leq \frac{C_0\eta^{1-\epsilon}}{\left|x-tN\right|}, \quad \quad \left|h_1'(x)-f_v'(x)\right|\leq \frac{C_0\eta^{1-\epsilon}}{\left|x-tN\right|^2}. \nonumber
\end{equation}
\end{enumerate}
\end{prop}

\quad This result states that, away from the midline $x=\frac{N}{2}$, the nodal set of $v$ closely resembles the graph $\left(x,f_v(x)\right)$ for $f_v(x)$ as in (\ref{eq:fv}). The estimates worsen in regions closer to the line $x=\frac{N}{2}$ or the left or right boundary of $\Omega$; Remark \ref{rem:lineSplit} discusses our efforts to manage this loss.

\begin{proof}[Proof of Proposition \ref{prop:horizontal}.] Our first goal is to show that the nodal set for $v$ can be written locally as a graph of $x$ in $\widetilde{\Omega}$ with $\left|y-\frac{1}{2}\right|<\frac{1}{2}\max\{\overline{y}, 1-\overline{y}\}$ and $\left|x-\frac{N}{2}\right|\geq 2\eta^{\frac{1}{2}-\frac{2}{3}\epsilon}$. Throughout the proof, fix $t\in\left\{0,\frac{1}{2},1\right\}$. The analysis differs slightly between when $\left|x-tN\right|\leq 0.1$, and $0.1<\left|x-\frac{N}{2}\right|<\frac{N}{2}-0.1$; we will highlight where differences occur. By Proposition \ref{prop:V},
\begin{equation}\label{eq:key}
    \big|v(x,y)\big|\geq \big|v_0(x,y)\big|-C\eta^{1-\epsilon}N^{-\frac{3}{2}}\geq cN^{-\frac{1}{2}}\big|\sin\left(2\pi x/N\right)\big|\bigg|\frac{c_1}{2c_2}\frac{\sin(k\pi x/N)}{\sin(2\pi x/N)}+\cos(\pi y)\bigg|-C\eta^{1-\epsilon}N^{-\frac{3}{2}}. 
\end{equation}
If $(x_0,y_0)\in v^{-1}(0)$, (\ref{eq:fv}) implies
\begin{equation}\label{eq:cylinder}
    \big|y_0-f_v(x_0)\big|\leq C\eta^{1-\epsilon}N^{-1}\left|\sin(2\pi x_0/N)\right|^{-1}.
\end{equation}
With this bound in place, we next want to bound $\big|\p_y v(x_0,y_0)\big|$ below. To do this, we compare it to $\p_yv_0(x_0,f_v(x_0))$, which satisfies
\begin{align}\label{eq:dyg}
    \left|\p_yv_0(x, f_v(x))\right|\geq cN^{-\frac{1}{2}}\left|\sin\left(2\pi x/N\right)\right|.
\end{align}
Any improvement on this lower bound depends on how close $x$ is to $tN$. Observe that if $(x,y)\in\widetilde{\Omega}$ such that $\left|y-\frac{1}{2}\right|< \frac{1}{2}\max\{\overline{y}, 1-\overline{y}\}$, then there is a constant $c>0$ such that 
\begin{equation}
    \left|\sin\left(2\pi x/N\right)\right|\geq
    \begin{cases}
        cN^{-1}, & 0.1<\left|x-\frac{N}{2}\right|<\frac{N}{2}-0.1 \\
        cN^{-1}\left|x-tN\right|, & \left|x-tN\right|\leq 0.1.
    \end{cases} \label{eq:largenbhd}
\end{equation}
With these bounds in place, we turn to $\p_yv$. Again by Proposition \ref{prop:V},
\begin{equation}\label{eq:dyv}
    \left| \p_y v(x,y)\right|\geq \left|\p_y v_0(x,y)\right|-C\eta^{1-\epsilon}N^{-\frac{3}{2}}
\end{equation}
and because the derivative $\p_y^2v_0$ is bounded,
\begin{align}
    \left| \p_y v_0(x, y)-\p_yv_0(x, f_v(x))\right|\leq \left(CN^{-\frac{3}{2}}+CN^{-\frac{1}{2}}\left|\sin\big(2\pi x/N\big)\right|\right)\left|y-f_v(x)\right|.  \nonumber
\end{align}
For $(x_0,y_0)\in v^{-1}(0)$, this implies
\begin{align}
    \left| \p_y v_0(x_0, y_0)-\p_yv_0(x_0, f_v(x_0))\right|\leq C\eta^{1-\epsilon}N^{-\frac{5}{2}}\left|\sin\left(2\pi x_0/N\right)\right|^{-1}+C\eta^{1-\epsilon}N^{-\frac{3}{2}}. \nonumber
\end{align}
By (\ref{eq:dyg}) and (\ref{eq:largenbhd}), 
\begin{equation}\label{eq:smallnbhd}
    \left|\p_yv_0(x_0,y_0)\right|\geq 
    \begin{cases}
    cN^{-\frac{1}{2}}\left|\sin\left(2\pi x_0/N\right)\right|-C\eta^{1-\epsilon}N^{-\frac{3}{2}}, & 0.1<\left|x_0-\frac{N}{2}\right|<\frac{N}{2}-0.1 \\ 
    cN^{-\frac{3}{2}}\left|x_0-tN\right|-C\eta^{1-\epsilon}N^{-\frac{3}{2}}\left|x_0-tN\right|^{-1}-C\eta^{1-\epsilon}N^{-\frac{3}{2}}, &  \left|x_0-tN\right|\leq 0.1.
    \end{cases}
\end{equation}
Here the discussion in Remark \ref{rem:Radius} becomes partially relevant. Because $\eta^{1-\epsilon}\ll \eta^{1-\frac{4}{3}\epsilon}$ and $\eta^{1-\epsilon}\ll \eta^{\frac{2}{3}(1-\epsilon)}$, $\left|\p_y v_0(x_0,y_0)\right|$ is positive over the regions of each item in Proposition \ref{prop:horizontal}. Alongside (\ref{eq:dyv}), this means that $|\p_y v(x_0, y_0)|>0$ for small enough $\eta$. The implicit function theorem guarantees the existence of a neighborhood over which the level set can be written as a graph $(x, h_1(x))$. Combining (\ref{eq:cylinder}) with (\ref{eq:largenbhd}) gives the function bounds in Proposition \ref{prop:horizontal}. To obtain bounds on the derivatives, we write
\begin{equation}\label{eq:hImp}
    \left|h_1'(x)-f_v'(x)\right|=\left|\frac{\p_x v(x, h_1(x))}{\p_y v(x, h_1(x))}+f_v'(x)\right|=\left|\frac{\p_xv(x, h_1(x))+f_v'(x)\p_yv(x,h_1(x))}{\p_y v(x,h_1(x))}\right| 
\end{equation}
and use (\ref{eq:cylinder}) and (\ref{eq:smallnbhd}) estimate the denominator. To bound the numerator, we use the mean value theorem to compare the closely related quantity $\p_x \left(v_0(x, f_v(x))\right)$, which is identically zero over $[0,N]$. By Proposition \ref{prop:V}, (\ref{eq:cylinder}), and the bound on $|f_v'(x)|$, 
\begin{equation}
    \left|\p_xv(x_0, y_0)+f_v'(x_0)\p_yv(x_0, y_0)\right|\leq C\eta^{1-\epsilon}N^{-\frac{3}{2}}+C\eta^{1-\epsilon}N^{-\frac{5}{2}}\left|\sin(2\pi x_0/N)\right|^{-1}. \nonumber
\end{equation}
Combining these bounds for the numerator and denominator of (\ref{eq:hImp}) along with (\ref{eq:largenbhd}) and (\ref{eq:smallnbhd}), we have the desired estimates for the derivative differences.
\end{proof}

\quad Next we study the nodal set both near the top and bottom boundary components of $\Omega$ and near the midline $x=\frac{N}{2}$. The behavior near the boundary components is delicate; we present estimates on the structure of $v^{-1}(0)$ in Proposition \ref{prop:vertical} but defer the presentation of an improved regularity estimate until Proposition \ref{prop:bottomBdry}. In combination, Propositions \ref{prop:horizontal} and \ref{prop:vertical} completely describe the nodal set of $v$ in $\widetilde{\Omega}$. 

\begin{prop} \label{prop:vertical}
Under the same conditions as in Theorem \ref{prop:n=0}, there exists a constant $C_0>0$, dependent only on $\Lambda_\phi$, such that for all $k\geq 4$ even, there exists $\eta_0(k,\epsilon)>0$ such that for all $0<\eta\leq \eta_0(k,\epsilon)$, the nodal set of $v$ has the following property: If $(x_0, y_0)\in v^{-1}(0)$ is a point in $\widetilde{\Omega}_\phi(\eta, N)$ with $\left|x_0-\frac{N}{2}\right|\leq 2\eta^{\frac{1}{2}-\frac{2}{3}\epsilon}$ or $\left|y_0-\frac{1}{2}\right|\geq \frac{1}{2}\max\{\overline{y}, 1-\overline{y}\}$, then there exists an open neighborhood $U_2$ containing $y_0$ and a function $h_2(y)$ such that $v^{-1}(0)\cap U_2=\{(h_2(y),y)\}$. For all $y\in U_2$, 
\begin{equation}
    \left|h_2(y)-\frac{N}{2}\right|\leq \frac{C_0\eta^{1-\epsilon}}{|y-\overline{y}|}, \quad \quad \left|h_2'(y)\right|\leq\frac{C_0\eta^{1-\epsilon}}{|y-\overline{y}|^2}. \nonumber
\end{equation}
\end{prop}

\quad Proposition \ref{prop:vertical} states that away from $\left(\frac{N}{2},\overline{y}\right)$, the nodal set of $v$ in this region behaves like the vertical line $x=\frac{N}{2}$. However, the estimates worsen for points close to the disc $D_r(\textbf{z})$ studied in Section \ref{sec:centerEven}, where the nodal set resembles a hyperbola. 

\begin{proof}[Proof of Proposition \ref{prop:vertical}.] For this proof, we break $\widetilde{\Omega}$ into two subregions of interest:
\begin{enumerate}[label=(\alph*)]
    \item $\big|y-\frac{1}{2}\big|\leq \frac{1}{2}\max\{\overline{y}, 1-\overline{y}\}$ and $\big|x-\frac{N}{2}\big|\leq 2\eta^{\frac{1}{2}-\frac{2}{3}\epsilon}$
    \item $\big|y-\frac{1}{2}\big|> \frac{1}{2}\max\{\overline{y}, 1-\overline{y}\}$
\end{enumerate}
In region ($a$), we follow a similar argument as in the proof of the previous proposition. Equation (\ref{eq:key}) still holds, but this time we are close to $x=\frac{N}{2}$. Thus, for $(x_0,y_0)\in v^{-1}(0)$,
\begin{equation}\label{eq:QD1}
    \left|x_0-\frac{N}{2}\right|\leq C\eta^{1-\epsilon}\left|y_0-f(x_0)\right|^{-1} \leq C\eta^{1-\epsilon}\left|y_0-\overline{y}\right|^{-1}
\end{equation}
where the last inequality holds because $|y_0-f_v(x_0)|\geq c|y_0-\overline{y}|$ for small enough $\eta$. By Proposition \ref{prop:V} and the mean value theorem,
\begin{equation}
    \left|\p_xv(x_0, y_0)\right|\geq \left|\p_xv_0(x_0, y_0)\right|-C\eta^{1-\epsilon}N^{-\frac{3}{2}}\geq \left|\p_xv_0\left(\frac{N}{2}, y_0\right)\right|- CN^{-\frac{3}{2}}\left|x_0-\frac{N}{2}\right|-C\eta^{1-\epsilon}N^{-\frac{3}{2}}. \nonumber
\end{equation}
The expression for $v_0$ gives
\begin{equation}
    \left|\p_xv_0\left(\frac{N}{2}, y_0\right)\right|\geq cN^{-\frac{3}{2}}|y_0-\overline{y}| \nonumber
\end{equation}
and by (\ref{eq:QD1}), this lower bound translates to $\big|\p_xv(x_0,y_0)\big|$. The implicit function theorem allows us to write $v^{-1}(0)$ locally in this region as a graph $\big(h_2(y),y\big)$. A pointwise estimate follows from (\ref{eq:QD1}), but in order to estimate the derivative, we need an upper bound on $\p_yv\big(h_2(y),y\big)$. Because $\p_yv_0\left(\frac{N}{2}, y\right)=0$,
\begin{equation}
    \left|\p_yv(x_0, y_0)\right|\leq \left|\p_yv_0(x_0, y_0)\right|+C\eta^{1-\epsilon}N^{-\frac{3}{2}}\leq CN^{-\frac{3}{2}}\left|x_0-\frac{N}{2}\right|+C\eta^{1-\epsilon}N^{-\frac{3}{2}}\leq C\eta^{1-\epsilon}N^{-\frac{3}{2}}\left|y_0-\overline{y}\right|^{-1}. \nonumber
\end{equation}
By the implicit function theorem,
\begin{equation}
    \left|h_2'(y)\right|=\left|\frac{\p_yv(h_2(y), y)}{\p_x v(h_2(y),y)}\right|\leq C\eta^{1-\epsilon}\left|y-\overline{y}\right|^{-2} \nonumber
\end{equation}
establishing the estimate for this region. 

\quad In region ($b$), we use a line integral from the bottom or top boundaries to control the error. We focus on the component with $0< y<\frac{1}{2}\min\{\overline{y}, 1-\overline{y}\}$. Recall $E(x,0)=E(x,1)=0$ so by item $(iii)$ in Proposition \ref{prop:V},
\begin{equation}
    \left| E(x,y)-E(x,0)\right|\leq \int_0^y \left|\nabla E\right|dt\leq C\eta N^{-\frac{3}{2}}|y|. \nonumber
\end{equation}
We apply (\ref{eq:key}) as a lower bound for $|v|$, but this time $|\sin(\pi y)|$ is bounded below by $|y|$ rather than a constant
\begin{align}\label{eq:LowBdforV}
    \left|v(x,y)\right|\geq cN^{-\frac{1}{2}}|y|\left|\sin\left(2\pi x/N\right)\right|\left|\cos(\pi y)-\cos(\pi f(x))\right|-C\eta^{1-\epsilon}N^{-\frac{3}{2}}|y|.
\end{align}
Because $\left|f_v(x)-\frac{1}{2}\right|\leq \left|\overline{y}-\frac{1}{2}\right|$ and we are considering $y<\frac{1}{2}\min\{\overline{y},1-\overline{y}\}$, this implies that  $(x_0,y_0)\in v^{-1}(0)$ only if $\left|x_0-\frac{N}{2}\right|$ is bounded by a multiple of $\eta^{1-\epsilon}$. With this knowledge, we can bound the partial $\p_xv$ from below
\begin{align}
    \left|\p_xv(x_0,y_0)\right|\geq \left|\p_xv_0(x_0, y_0)\right|-C\eta^{1-\epsilon}N^{-\frac{3}{2}}|y_0|\geq \left|\p_xv_0\left(\frac{N}{2},y_0\right)\right|-CN^{-\frac{3}{2}}|y_0|\left|x_0-\frac{N}{2}\right|-C\eta^{1-\epsilon}N^{-\frac{3}{2}}|y_0| \nonumber
    \\
    \geq cN^{-\frac{3}{2}}|y_0|-C\eta^{1-\epsilon}N^{-\frac{3}{2}}|y_0| \nonumber
\end{align}
So $|\p_xv(x_0, y_0)|\geq cN^{-\frac{3}{2}}|y_0|$ for small enough $\eta$. The implicit function theorem then guarantees the existence of $h_2(y)$ as in Proposition \ref{prop:vertical}, but in order to bound the derivative $h_2'(y)$, we need to likewise bound $|\p_yv(x_0, y_0)|$ in terms of $|y_0|$. To do this, we write $\p_yv$ in terms of the error $E$. If $(x_0, y_0)\in v^{-1}(0)$, then 
\begin{equation}
    E(x_0, y_0)=-v_1(x_0)\sin(\pi y_0)-v_2(x_0)\sin(2\pi y_0) \nonumber
\end{equation}
and therefore
\begin{align}\label{eq:nextSection}
    \p_yv(x_0, y_0)=-\pi\frac{\cos(\pi y_0)}{\sin(\pi y_0)}E(x_0, y_0)+\p_yE(x_0, y_0)+\pi v_2(x_0)\bigg(2\cos(2\pi y_0)-\frac{\sin(2\pi y_0)}{\sin(\pi y_0)}\cos(\pi y_0)\bigg). 
\end{align}
Because $x_0$ is close to $\frac{N}{2}$, we can use item $(ii)$ of Proposition \ref{prop:V} to bound $|v_2(x_0)|\leq C\eta^{1-\epsilon}N^{-\frac{3}{2}}$. The accompanying term in parenthesis is bounded by a multiple of $|y_0|$, so we restrict our focus to the first two terms. Because $E(x_0,0)=0$, we can Taylor expand the error terms in the following manner.
\begin{align}
    E(x_0, y_0)=y_0\p_y E(x_0, 0)+\frac{1}{2}y_0^2\p_y^2E(x_0, m_1), \nonumber
    \\
    \p_yE(x_0, y_0)=\p_y E(x_0, 0)+y_0\p_y^2 E(x_0, m_2) \nonumber
\end{align}
for some $m_1, m_2\in(0, y_0)$. By Proposition \ref{prop:V} item $(iii)$, we can then write
\begin{equation}
    \left|\p_yv(x_0, y_0)-\left(1-\pi\frac{\cos(\pi y_0)}{\sin(\pi y_0)}y_0\right)\p_y E(x_0, 0)\right|\leq C\eta^{1-\epsilon}N^{-\frac{3}{2}}|y_0|. \nonumber
\end{equation}
Fortunately, $\left|1-\pi\frac{\cos(\pi y_0)}{\sin(\pi y_0)}y_0\right|\leq C|y_0|$, so the bound $|\p_yE|\leq C\eta N^{-\frac{3}{2}}$ allows us to write $|\p_yv(x_0, y_0)|\leq C\eta^{1-\epsilon}N^{-\frac{3}{2}}|y_0|$. Thus,
\begin{equation}
    \big|h_2'(y)\big|=\bigg|\frac{\p_yv(h_2(y),y)}{\p_xv(h_2(y),y)}\bigg|\leq C\eta^{1-\epsilon}\leq C\eta^{1-\epsilon}\left|y-\overline{y}\right|^{-2}. \nonumber
\end{equation}
The same argument holds for the case when $1-y\leq \frac{1}{2}\min\{\overline{y},1-\overline{y}\}$. 
\end{proof}

\quad Proposition \ref{prop:horizontal} immediately implies items $(ii)$ and $(iii)$ of Theorem \ref{thm:VevenPt} away from the left and right boundaries of $\Omega$. The regularity statements of Propositions \ref{prop:horizontal} and \ref{prop:vertical} establish item $(ii)$ of Theorem \ref{thm:VevenDer} for $\left|x-tN\right|\geq \eta^{\frac{1}{3}(1-\epsilon)}$ with $t\in\left\{0,1\right\}$.

%%%%%%%%%%%%%%%%%%%%%%%%%%%%%%%%%%%%%%%%%%%%%%%%%%%
\section{Description of the Nodal Set Near the Boundary (k even)}\label{sec:bdry}

\quad In this section we describe the nodal set of $v$ near the boundary $\p\Omega_\phi(\eta, N)$ when $k\geq 4$ is even. Because the domain deformation affects only the left boundary component, we focus primarily on the nodal structure in the region 
\begin{equation}\label{eq:newRegion}
    \Omega_\phi(\eta, N)\cap\left\{x<2\eta^{\frac{1}{3}(1-\epsilon)}\right\} \quad \textrm{with} \quad \epsilon\in\left(0,\frac{1}{4}\right)
\end{equation}
before proving derivative estimates on the nodal set at each of the four points where it intersects the boundary. According to \cite{Helffer2009}, each boundary intersection occurs orthogonally; we justify this quantitatively. The region (\ref{eq:newRegion}) is chosen to provide some overlap with $\widetilde{\Omega}_\phi(\eta, N)$ as defined in Section \ref{sec:awayEven}. To study the behavior near the left and right boundaries, we employ an isometry $F$ featured in \cite{BCM, BGM}. For $(x_0,y_0)\in v^{-1}(0)$ with $x_0<2\eta^{\frac{1}{3}(1-\epsilon)}$, let $(x_1, y_1)\in\p\Omega$ be such that
\begin{equation}
    \inf_{(x,y)\in\p\Omega}\left|\left|(x_0,y_0)-(x,y)\right|\right|_{\ell^2}=\left|\left|(x_0,y_0)-(x_1, y_1)\right|\right|_{\ell^2}. \nonumber
\end{equation}
Then $F^{-1}$ is a rotation about $(x_0,y_0)$ such that $ F^{-1}(x_1,y_1)=(\:\cdot\:, y_0)$. We write $(\tilde{x}, \tilde{y})=F^{-1}(x,y)$ as the rotated coordinates and define $\tilde{v}(\tilde{x}, \tilde{y})=v\circ F(\tilde{x}, \tilde{y})$. For $x_0>N-2\eta^{\frac{1}{3}(1-\epsilon)}$, the rotation is trivial; it is for this reason that we focus on the left boundary. Let $\tilde{x}=-\eta\tilde{\phi}(\tilde{y})$ describe the rotated boundary so that the isometry $F$ guarantees $\tilde{\phi}'(y_0)=0$. In these coordinates, we provide the following eigenfunction decomposition for $\tilde{v}$.

\begin{lemma}\label{lem:7} 
Under the same conditions as in Theorem \ref{thm:VevenPt}, there exists a constant $C_0>0$, dependent only on $\Lambda_\phi$, such that for each $k\geq 3$, there exists $\eta_0(k,\epsilon)>0$ such that for $0<\eta\leq \eta_0(k,\epsilon)$, we can write $\tilde{v}(\tilde{x},\tilde{y})$ as 
\begin{equation}
    \tilde{v}(\tilde{x}, \tilde{y})=\tilde{v}_1(\tilde{x})\sin(\pi\tilde{y})+\tilde{v}_2(\tilde{x})\sin(2\pi\tilde{y})+\tilde{E}(\tilde{x}, \tilde{y}) \nonumber
\end{equation}
over all $(\tilde{x}, \tilde{y})\in\Omega$ satisfying $\tilde{x}<2\eta^{\frac{1}{3}(1-\epsilon)}$ and $\left|\tilde{y}-\frac{1}{2}\right|< \frac{1}{2}\max\{\overline{y}, 1-\overline{y}\}$. In particular, $\tilde{v}_n(-\eta\tilde{\phi}(y_0))=0$ with
\begin{align}
    \left|\tilde{v}_1^\ell(\tilde{x})-\frac{2c_1}{\sqrt{N}}\frac{d^\ell}{d\tilde{x}^\ell}\left(\sin\left(\frac{k\pi}{N}\tilde{x}\right)\right)\right|\leq C_0\eta^{1-\epsilon}/N^{\frac{3}{2}}, \nonumber
    \\
    \left|\tilde{v}_2^\ell(\tilde{x})-\frac{2c_2}{\sqrt{N}}\frac{d^\ell}{d\tilde{x}^\ell}\left(\sin\left(\frac{2\pi}{N}\tilde{x}\right)\right)\right|\leq C_0\eta^{1-\epsilon}/N^{\frac{3}{2}}, \nonumber
\end{align}
and
\begin{equation}
    \left|\nabla^\ell\tilde{E}(\tilde{x},\tilde{y})\right|\leq C_0\eta/N^{\frac{3}{2}} \nonumber
\end{equation}
for $0\leq\ell\leq 3$.
\end{lemma}

\quad This result serves as a companion to Proposition \ref{prop:V} in the rotated coordinates $\left(\tilde{x},\tilde{y}\right)$, allowing us to compare the function $\tilde{v}$ to the expression for $v_0$ established in Section \ref{sec:Hadamard}. In particular, the estimates in Lemma \ref{lem:7} provide enough control to describe the nodal behavior of $v$ away from the top and bottom components of $\p\Omega$. At the end of this section, we prove that this is sufficient in describing the nodal set in (\ref{eq:newRegion}) as $v$ is nonzero near the corners of $\Omega$.

\begin{proof}[Proof of Lemma \ref{lem:7}.] Fix $(x_0,y_0)\in v^{-1}(0)$ with $x_0<2\eta^{\frac{1}{3}(1-\epsilon)}$ and let $\tilde{x}=-\eta\tilde{\phi}(\tilde{y})$ describe the rotated boundary. We first observe that by Definition \ref{def:phi}, the angle of rotation $\theta$ is bounded by a multiple of $\eta$. If we write
\begin{equation}
    \begin{pmatrix} \tilde{x}-x_0 \\ \tilde{y}-y_0 \end{pmatrix}=\begin{pmatrix} \cos\theta & -\sin\theta \\ \sin\theta & \cos\theta \end{pmatrix} \begin{pmatrix} x-x_0 \\ y-y_0\end{pmatrix} \nonumber
\end{equation}
then we can bound 
\begin{equation}\label{eq:7}
    |\tilde{x}-x|+|\tilde{y}-y|\leq C\eta.
\end{equation}
In the rotated coordinates $(\tilde{x}, \tilde{y})$, we define
\begin{equation}
    \tilde{v}_1(\tilde{x})= v_1(\tilde{x})-v_1(-\eta\tilde{\phi}(y_0)), \quad \quad \tilde{v}_2(\tilde{x})= v_2(\tilde{x})-v_2(-\eta\tilde{\phi}(y_0)) \nonumber
\end{equation}
so that $\tilde{v}_n(-\eta\tilde{\phi}(y_0))=0$ and $\tilde{v}_n^{(\ell)}(\tilde{x})=v_n^{(\ell)}(\tilde{x})$ for $n=1,2$. Proposition \ref{prop:V} provides the estimates in Lemma \ref{lem:7}. To determine the error bound, we observe that because $\tilde{v}=v\circ F$, 
\begin{align}
    \tilde{E}(\tilde{x}, \tilde{y})-E(x,y)=v_1(x)\sin(\pi y)+v_2(x)\sin(2\pi y)-\tilde{v}_1(\tilde{x})\sin(\pi\tilde{y})-\tilde{v}_2(\tilde{x})\sin(2\pi\tilde{y}) \nonumber
    \\
    =v_1(x)\big(\sin(\pi y)-\sin(\pi\tilde{y})\big)+v_2(x)\big(\sin(2\pi y)-\sin(2\pi\tilde{y})\big)+\big(v_1(x)-\tilde{v}_1(\tilde{x})\big)\sin(\pi\tilde{y})+\big(v_2(x)-\tilde{v}_2(\tilde{x})\big)\sin(2\pi\tilde{y}). \nonumber
\end{align}
Using (\ref{eq:7}) and the bound $|v_n(x)|\leq C/N^{3/2}$, the first two terms are bounded by $C\eta/N^{\frac{3}{2}}$. To bound the final two terms, 
\begin{equation}
    \big|v_n(x)-\tilde{v}_n(\tilde{x})\big|\leq \big|v_n(x)-v_n(\tilde{x})\big|+\big|v_n(-\eta\tilde{\phi}(y_0))\big|\leq C\eta/N^{\frac{3}{2}} \nonumber
\end{equation}
using the mean value theorem on the first piece and Proposition \ref{prop:V} on the second. This implies that
\begin{equation}
    \big|\tilde{E}(\tilde{x}, \tilde{y})-E(x,y)\big|\leq C\eta/N^{\frac{3}{2}} \nonumber
\end{equation}
and the bounds on $E(x,y)$ give us the final result. Differentiating these expressions and repeating the same argument provides the corresponding bounds on $\big|\nabla^\ell(\tilde{E}-E)\big|$. Because $\overline{y}$ is $\sim 1$ from the boundaries, we are away from the corners of $\Omega$ and can use item $(iii)$ of Proposition \ref{prop:V} to then get control over $|\nabla^\ell \tilde{E}|$. 
\end{proof}

\quad We expect the nodal set in the rotated coordinates to behave much like $\tilde{y}=\lim_{x\to 0^+}f_v(x)$ where $f_v(x)$ parametrizes the nodal set of $v_0$, featured in (\ref{eq:fv}). By (\ref{eq:ybar}), $\left|\overline{y}-\frac{1}{2}\right|=\left|f_v(0)-\frac{1}{2}\right|$ and therefore Lemma \ref{prop:coeff} implies that $f_v(0)=\lim_{x\to 0^+}f_v(x)$ is bounded away from the top and bottom boundaries. Under this condition, a precise description of the nodal set near the left boundary is provided through the following proposition.

\begin{prop}\label{prop:leftBdry}
    Fix $(x_0,y_0)\in v^{-1}(0)$ with $x_0<2\eta^{\frac{1}{3}(1-\epsilon)}$ and apply the transform $F^{-1}$ as described. Under the same conditions as in Theorem \ref{thm:VevenPt}, there exists a constant $C_0>0$, dependent only on $\Lambda_\phi$, such that for each $k\geq 4$ even, there exists $\eta_0(k,\epsilon)>0$ such that for all $0<\eta\leq\eta_0(k,\epsilon)$, the nodal set of $v$ can be written locally as a graph $\tilde{y}=h_3(\tilde{x})$ with
    \begin{equation}
        \left|h_3(\tilde{x})-f_v(0)\right|\leq C_0\eta^{\frac{2}{3}(1-\epsilon)}, \quad \quad \left|h_3'(\tilde{x})\right|\leq C_0\eta^{\frac{2}{3}(1-\epsilon)}\left|\tilde{x}+\eta\tilde{\phi}(y_0)\right| \nonumber
    \end{equation}
    where $f_v(0)=\frac{1}{\pi}\arccos\left(-\frac{kc_1}{4c_2}\right)$.
\end{prop} 

\quad Because $\tilde{\phi}'(y_0)=0$, the regularity estimate in Proposition \ref{prop:leftBdry} implies that the nodal set intersects the left boundary orthogonally. The right boundary can be handled in an identical argument, with the rotated coordinates $(\tilde{x},\tilde{y})=(x,y)$. Because $f_v'(0)=0$ per (\ref{eq:fv}), Proposition \ref{prop:leftBdry} also implies that 
\begin{equation}\label{eq:turntables}
    \left|h_3(\tilde{x})-f_v(\tilde{x})\right|\leq C\eta^{\frac{2}{3}(1-\epsilon)}+\left|f_v(\tilde{x})-f_v(0)\right|\leq C\eta^{\frac{2}{3}(1-\epsilon)}
\end{equation}
for $|\tilde{x}|\leq C\eta^{\frac{1}{3}(1-\epsilon)}$. By (\ref{eq:7}), this holds true in the non-rotated coordinates also and provides a description of the nodal curve approaching the left boundary. A similar analysis holds for bounding $\left|h_3'(x)-f_v'(x)\right|$ near the boundary.

\begin{rem}\label{rem:lineSplit}
    The results in Proposition \ref{prop:leftBdry} are dependent on the bound $x_0<2\eta^{\frac{1}{3}(1-\epsilon)}$, in the sense that increasing this bound worsens the estimates. The results in Section \ref{sec:awayEven} are likewise dependent on the construction of $\widetilde{\Omega}$. The choice to split our analysis across a line $x\sim \eta^{\frac{1}{3}(1-\epsilon)}$ minimizes the bound in item $(iii)$ of Theorem \ref{thm:VevenPt} by enforcing an equal maximal loss in $\eta$ over both sides of this line.
\end{rem}

\begin{proof}[Proof of Proposition \ref{prop:leftBdry}.] Fix $(x_0,y_0)$ in the nodal set of $v$ with $x_0<2\eta^{\frac{1}{3}(1-\epsilon)}$. Because $F$ is a rotation about this point and $\tilde{v}=v\circ F$, this point also belongs to the nodal set of $\tilde{v}$. For simplicity of notation, let $x_l=-\eta\tilde{\phi}(y_0)$ and $V(\tilde{x},\tilde{y})=\tilde{v}_1(\tilde{x})\sin(\pi\tilde{y})+\tilde{v}_2(\tilde{x})\sin(2\pi\tilde{y})$. By construction of the transformation, $\tilde{\phi}'(y_0)=0$ and $\tilde{v}(-\eta\tilde{\phi}(\tilde{y}),\tilde{y})=0$. If we evaluate this expression and its $\tilde{y}$-derivative at $y_0$, we have
\begin{equation}\label{eq:EEs}
    \tilde{E}(x_l,y_0)=\p_{\tilde{y}}\tilde{E}(x_l,y_0)=0.
\end{equation}
Lemma \ref{lem:7} then implies that $|\tilde{E}(x_0,y_0)|\leq C\eta N^{-\frac{3}{2}}|x_0-x_l|$. Further, because $V(x_l,\:\cdot\:)=0$, 
\begin{equation}
    \left|V(x_0,y_0)\right|=\left|\p_{\tilde{x}}V(m,y_0)\right|\left|x_0-x_l\right| \nonumber
\end{equation}
for some $|m|\leq C\eta^{\frac{1}{3}(1-\epsilon)}$. By Lemma \ref{lem:7},
\begin{equation}
    \left|\p_{\tilde{x}}V(m,y_0)\right|\geq cN^{-\frac{3}{2}}\left|\frac{kc_1}{4c_2}\frac{\cos(k\pi m/N)}{\cos(2\pi m/N)}+\cos(\pi y_0)\right|-C\eta^{1-\epsilon}N^{-\frac{3}{2}}\geq cN^{-\frac{3}{2}}\big|y_0-f_v(0)\big|-C\eta^{\frac{2}{3}(1-\epsilon)}N^{-\frac{3}{2}} \nonumber
\end{equation}
because $\left|\frac{\cos(k\pi m/N)}{\cos(2\pi m/N)}-1\right|\leq Cm^2$. This provides the lower bound
\begin{equation}
    0=\left|\tilde{v}(x_0,y_0)\right|\geq cN^{-\frac{3}{2}}\left|y_0-f_v(0)\right|\left|x_0-x_l\right|-C\eta^{\frac{2}{3}(1-\epsilon)}N^{-\frac{3}{2}}\left|x_0-x_l\right|, \nonumber
\end{equation}
which implies that $|y_0-f_v(0)|\leq C\eta^{\frac{2}{3}(1-\epsilon)}$. With control over $y_0$, we use the fact that $\p_{\tilde{y}}V(x_l,\:\cdot\:)=0$ to write
\begin{equation}
    \left|\p_{\tilde{y}}V(x_0,y_0)\right|=\left|\p_{\tilde{x}\tilde{y}}^2V(m,y_0)\right|\left|x_0-x_l\right| \nonumber
\end{equation}
for some $|m|\leq C\eta^{\frac{1}{3}(1-\epsilon)}$. Lemma \ref{lem:7} then implies
\begin{equation}\label{eq:LowerBd}
    \left|\p_{\tilde{y}}V(x_0,y_0)\right|\geq cN^{-\frac{3}{2}}\left|x_0-x_l\right|
\end{equation}
given $y_0$ close to $f_v(0)$ and $\eta$ small. By (\ref{eq:EEs}) and Lemma \ref{lem:7}, $\left|\p_{\tilde{y}}\tilde{E}(x_0,y_0)\right|\leq C\eta N^{-\frac{3}{2}}\left|x_0-x_l\right|$ and so the lower bound in (\ref{eq:LowerBd}) translates to $|\p_{\tilde{y}}\tilde{v}(x_0,y_0)|$. The implicit function theorem guarantees the existence of a neighborhood containing $x_0$ such that the nodal set of $\tilde{v}$ can be written as the graph $\tilde{y}=h_3(\tilde{x})$. To get the desired estimate on $h_3'(x_0)$, it suffices to prove
\begin{equation}\label{eq:Final}
    \left|\p_{\tilde{x}}\tilde{v}(x_0,y_0)\right|\leq C\eta^{\frac{2}{3}(1-\epsilon)}N^{-\frac{3}{2}}\left|x_0-x_l\right|^2.
\end{equation}
To accomplish this, we write
\begin{equation}\label{eq:3Terms}
    \p_{\tilde{x}}\tilde{v}(x_0,y_0)=\tilde{v}_1(x_0)\sin(\pi y_0)\left(\frac{\tilde{v}_1'(x_0)}{\tilde{v}_1(x_0)}-\frac{\tilde{v}_2'(x_0)}{\tilde{v}_2(x_0)}\right)+\p_{\tilde{x}}\tilde{E}(x_0,y_0)-\frac{\tilde{v}_2'(x_0)}{\tilde{v}_2(x_0)}\tilde{E}(x_0,y_0) 
\end{equation}
using the fact that
\begin{equation}
    \tilde{v}(x_0,y_0)=\tilde{v}_1(x_0)\sin(\pi y_0)+\tilde{v}_2(x_0)\sin(2\pi y_0)+\tilde{E}(x_0,y_0)=0. \nonumber
\end{equation}
For $n=1,2$, we can Taylor expand 
\begin{align}
    \tilde{v}_n(x_0)=\left(x_0-x_l\right)\tilde{v}_n'(x_l)+\frac{1}{2}\left(x_0-x_l\right)^2\tilde{v}_n''(x_l)+\frac{1}{6}\left(x_0-x_l\right)^3\tilde{v}_n'''(m_1^n) \nonumber
    \\
    \tilde{v}_n'(x_0)=\tilde{v}_n'(x_l)+\left(x_0-x_l\right)\tilde{v}_n''(x_l)+\frac{1}{2}\left(x_0-x_l\right)^2\tilde{v}_n'''(m_2^n) \nonumber
\end{align}
for some $|m_1^n|,|m_2^n|\leq C\eta^{\frac{1}{3}(1-\epsilon)}$. Combining these expressions,
\begin{equation}\label{eq:Wterms}
    \frac{\tilde{v}_n'(x_0)}{\tilde{v}_n(x_0)}=\frac{1}{x_0-x_l}+\frac{1}{2}\left(x_0-x_l\right)\frac{\tilde{v}_n''(x_l)}{\tilde{v}_n(x_0)}+\frac{1}{6}\left(x_0-x_l\right)^2\frac{\left(3\tilde{v}_n'''(m_2^n)-\tilde{v}_n'''(m_1^n)\right)}{\tilde{v}_n(x_0)}.
\end{equation}
By Lemma \ref{lem:7}, $|\tilde{v}_n(x_0)|\sim N^{-\frac{3}{2}}|x_0-x_l|$ and the last term above is bounded by a multiple of $\eta^{\frac{2}{3}(1-\epsilon)}|x_0-x_l|$. Thus, the first term in (\ref{eq:3Terms}) can be written
\begin{equation}\label{eq:Final3}
    \tilde{v}_1(x_0)\sin(\pi y_0)\left(\frac{\tilde{v}_1'(x_0)}{\tilde{v}_1(x_0)}-\frac{\tilde{v}_2'(x_0)}{\tilde{v}_2(x_0)}\right)=\frac{1}{2}\tilde{v}_1(x_0)\sin(\pi y_0)\left(x_0-x_l\right)\left(\frac{\tilde{v}_1''(x_l)}{\tilde{v}_1(x_0)}-\frac{\tilde{v}_2''(x_l)}{\tilde{v}_2(x_0)}\right)
\end{equation}
up to acceptable error per (\ref{eq:Final}). To bound the remaining terms, we expand the error
\begin{align}
    \tilde{E}(x_0,y_0)=\left(x_0-x_l\right)\p_{\tilde{x}}\tilde{E}(x_l,y_0)+\frac{1}{2}\left(x_0-x_l\right)^2\p_{\tilde{x}}^2\tilde{E}(x_l,y_0) +\frac{1}{6}\left(x_0-x_l\right)^3\p_{\tilde{x}}^3\tilde{E}(m_3,y_0) \nonumber
    \\
    \p_{\tilde{x}}\tilde{E}(x_0,y_0)=\p_{\tilde{x}}\tilde{E}(x_l,y_0)+\left(x_0-x_l\right)\p_{\tilde{x}}\tilde{E}(x_l,y_0)+\frac{1}{2}\left(x_0-x_l\right)^2\p_{\tilde{x}}^2\tilde{E}(m_4,y_0) \nonumber
\end{align}
for some $|m_3|,|m_4|\leq C\eta^{\frac{1}{3}(1-\epsilon)}$. By the error bounds in Lemma \ref{lem:7} and the expression for $n=2$ in (\ref{eq:Wterms}), the second and third terms in (\ref{eq:3Terms}) can be written as
\begin{equation}\label{eq:Final2}
    \p_{\tilde{x}}\tilde{E}(x_0,y_0)-\frac{\tilde{v}_2'(x_0)}{\tilde{v}_2(x_0)}\tilde{E}(x_0,y_0)=\frac{1}{2}(x_0-x_l)\p_{\tilde{x}}^2\tilde{E}(x_l,y_0)-\frac{1}{2}(x_0-x_l)^2\frac{\tilde{v}_2''(x_l)}{\tilde{v}_2(x_0)}\p_{\tilde{x}}\tilde{E}(x_l,y_0)
\end{equation}
up to acceptable error per (\ref{eq:Final}). Thus, if we set aside the well-behaved terms, (\ref{eq:3Terms}) reduces to the sum of (\ref{eq:Final3}) and (\ref{eq:Final2}). To tackle this sum, recall that $\tilde{v}(-\eta\tilde{\phi}(\tilde{y}),\tilde{y})=0$. If we differentiate twice in $\tilde{y}$ and evaluate at $y_0$, we get
\begin{equation}
    \p_{\tilde{y}}^2\tilde{v}(x_l,y_0)=\eta\tilde{\phi}''(y_0)\p_{\tilde{x}}\tilde{v}(x_l,y_0). \nonumber
\end{equation}
Because $\left(\Delta_{(\tilde{x},\tilde{y})}+\mu\right)\tilde{v}=0$ and $\tilde{v}(x_l,y_0)=0$, this means that
\begin{equation}
    \p_{\tilde{x}}^2\tilde{v}(x_l,y_0)=\tilde{v}_1''(x_l)\sin(\pi y_0)+\tilde{v}_2''(x_l)\sin(2\pi y_0)+\p_{\tilde{x}}^2\tilde{E}(x_l,y_0)=-\eta\tilde{\phi}''(y_0)\p_{\tilde{x}}\tilde{v}(x_l,y_0). \nonumber
\end{equation}
By rearranging this expression for $\p_{\tilde{x}}^2\tilde{E}$ and using the fact that $\tilde{v}(x_0,y_0)=0$,
\begin{equation}
    \p_{\tilde{x}}^2\tilde{E}(x_l,y_0)=-\eta\tilde{\phi}''(y_0)\p_{\tilde{x}}\tilde{v}(x_l,y_0)-\tilde{v}_1(x_0)\sin(\pi y_0)\left(\frac{\tilde{v}_1''(x_l)}{\tilde{v}_1(x_0)}-\frac{\tilde{v}_2''(x_l)}{\tilde{v}_2(x_0)}\right)+\frac{\tilde{v}_2''(x_l)}{\tilde{v}_2(x_0)}\tilde{E}(x_0,y_0). \nonumber
\end{equation}
If we substitute this expression into (\ref{eq:Final2}) and sum with (\ref{eq:Final3}), then we are left with
\begin{equation}
    \frac{1}{2}(x_0-x_l)\left(-\eta\tilde{\phi}''(y_0)\p_{\tilde{x}}\tilde{v}(x_0,y_0)\right)+\frac{1}{2}(x_0-x_l)\frac{\tilde{v}_2''(x_l)}{\tilde{v}_2(x_0)}\left(\tilde{E}(x_0,y_0)-(x_0-x_l)\p_{\tilde{x}}\tilde{E}(x_l,y_0)\right). \nonumber
\end{equation}
The first term can be bootstrapped in the inequality, and the last term is bounded by $C\eta^{\frac{2}{3}(1-\epsilon)}N^{-\frac{3}{2}}|x_0-x_l|^2$ by the Taylor expansion for $\tilde{E}$ and Lemma \ref{lem:7}. Thus, (\ref{eq:Final}) holds and $|h_3'(x_0)|\leq C\eta^{\frac{2}{3}(1-\epsilon)}|x_0-x_l|$. By taking the neighborhood of $x_0$ to be sufficiently small, we extend the bounds to nearby points $(\tilde{x},\tilde{y})\in \tilde{v}^{-1}(0)$.
\end{proof}

\quad Finally, we improve the regularity estimate in Proposition \ref{prop:vertical} to show that the nodal set of $v$ intersects the top and bottom boundaries orthogonally. For the sake of brevity, consider the following result near the lower boundary component.

\begin{prop}\label{prop:bottomBdry}
     Under the same conditions as in Theorem \ref{thm:VevenPt}, there exists a constant $C_0>0$, dependent only on $\Lambda_\phi$, such that for each $k\geq 4$ even, there exists $\eta_0(k,\epsilon)>0$ such that for all $0<\eta\leq\eta_0(k,\epsilon)$, the nodal set of $v$ with $y<\frac{1}{2}\min\{\overline{y},1-\overline{y}\}$ can be written as a graph $x=h_2(y)$ with
    \begin{equation}
        \left|h_2'(y)\right|\leq C_0\eta^{1-\epsilon}|y|. \nonumber
    \end{equation}
\end{prop}

\quad This establishes the orthogonality at the bottom boundary, and an analogous statement holds for the top boundary, allowing us to conclude item $(iii)$ of Theorem \ref{thm:VevenDer}.

\begin{proof}[Proof of Proposition \ref{prop:bottomBdry}.] From Proposition \ref{prop:vertical}, we have already established that the nodal set near the bottom boundary can be described by some graph $x=h_2(y)$ satisfying
\begin{equation}
    \left|h_2'(y_0)\right|\leq CN^{\frac{3}{2}}\left|\p_yv(x_0,y_0)\right||y_0|^{-1} \nonumber
\end{equation}
for $(x_0,y_0)\in v^{-1}(0)$. To improve this estimate, we return to the expression for $\p_yv(x_0,y_0)$ in (\ref{eq:nextSection}). The last term in that expression can in fact be bounded above by a multiple of $\eta^{1-\epsilon}N^{-\frac{3}{2}}|y_0|^2$, meaning that it suffices to show that
\begin{equation}\label{eq:Ess}
    \left|\pi\frac{\cos(\pi y_0)}{\sin(\pi y_0)}E(x_0, y_0)-\p_yE(x_0, y_0)\right|\leq C\eta^{1-\epsilon}N^{-\frac{3}{2}}|y_0|^2. 
\end{equation}
To accomplish this, we expand the error by one additional term. Because $E(x_0,0)=\p_y^2E(x_0,0)=0$,
\begin{align}
    E(x_0, y_0)=y_0\p_y E(x_0, 0)+\frac{1}{6}y_0^3\p_y^3E(x_0,m_1), \nonumber
    \\
    \p_yE(x_0, y_0)=\p_y E(x_0, 0)+\frac{1}{2}y_0^2\p_y^3 E(x_0, m_2) \nonumber
\end{align}
for some $m_1,m_2\in(0,y_0)$. For $y_0$ small, $\left|1-\pi\frac{\cos(\pi y_0)}{\sin(\pi y_0)}y_0\right|\leq C|y_0|^2$. Alongside the error estimates in Proposition \ref{prop:V}, this provides the bound in (\ref{eq:Ess}). By repeating this argument for $y_0$ close to $1$, we can show that near the top boundary, the nodal set is described by a graph with a derivative that vanishes as $y\to1$. 
\end{proof}

\quad There still remain the four corners of $\Omega$. By Propositions \ref{prop:vertical} and \ref{prop:leftBdry}, we know that no nodal lines cross the boundary of these regions. Thus, if the nodal set intersects a corner, it must entirely contain a nodal domain. Over such a domain, the ground-state eigenfunction would have eigenvalue $\mu\leq \lambda_{2,2}<8\pi^2$. Consider for example the lower-left corner with $x<2\eta^{\frac{1}{3}(1-\epsilon)}$ and $0\leq y\leq \frac{1}{2}\min\{\overline{y}, 1-\overline{y}\}$. This is a proper subset of $\left[-2\eta^{\frac{1}{3}(1-\epsilon)}, 2\eta^{\frac{1}{3}(1-\epsilon)}\right]\times\left[0, \frac{1}{2}\min\{\overline{y}, 1-\overline{y}\}\right]$ which has a ground-state eigenvalue $$\lambda_1\eqdef\pi^2\left(\frac{1}{16\eta^{\frac{2}{3}(1-\epsilon)}}+\frac{4}{\min\{\overline{y}, 1-\overline{y}\}^2}\right).$$ Domain monotonicity of Dirichlet eigenvalues tells us that $\lambda_1<\mu$. However, for small enough $\eta$, we find that $\mu\leq \lambda_1$, meaning no nodal domain is contained in this region. The same argument holds for the other components, and therefore $|v|>0$ near the four corners of $\Omega$. For this reason, our analysis of the nodal set $v^{-1}(0)$ with even $k$ is complete. To end this section, we briefly explain how these results translate to the estimates in Theorems \ref{thm:VevenPt} and \ref{thm:VevenDer}.

\quad By (\ref{eq:7}) and (\ref{eq:turntables}), Proposition \ref{prop:leftBdry} completes the proof of item $(iii)$ and establishes item $(iv)$ of Theorem \ref{thm:VevenPt}. The regularity statement in Proposition \ref{prop:leftBdry} completes the proof of item $(ii)$ in Theorem \ref{thm:VevenDer}, and orthogonality follows from Propositions \ref{prop:leftBdry} and \ref{prop:bottomBdry}. All that remains is to prove item $(v)$ of Theorem \ref{thm:VevenPt}. By Lemma \ref{prop:coeff} and (\ref{eq:fv}), there exists a small constant $c$ such that $$1-\left|f_v'(x)\right|\geq c>0$$ for all $x\in[0,N]$. By Theorem \ref{thm:VevenDer}, this implies that for small enough $\eta$, $\left|h'(x)\right|< 1$ whenever $v^{-1}(0)$ can be parametrized $y=h(x)$. Similarly, $\left|g'(y)\right|<1$ whenever $v^{-1}(0)$ can be parametrized $x=g(y)$. In combination with item $(i)$ of Theorem \ref{thm:VevenPt}, these derivative bounds prove item $(iv)$ of Theorem \ref{thm:VevenPt}. This concludes our comprehensive description of $v^{-1}(0)$ when $k$ is even and $\eta>0$.

%%%%%%%%%%%%%%%%%%%%%%%%%%%%%%%%%%%%%%%%%%%%%%%%%%%
\section{The Eigenfunction along the Upper Branch (k odd)}\label{sec:UpperOdd}

\quad In this section, we study the nodal set of the eigenfunction $v$ when $k\geq 5$ is an odd integer. When $\eta=0$, Theorem \ref{prop:n=0} tells us that the nodal set features two curves that lie outside a neighborhood of $x=\frac{N}{2}$ and separate the rectangle $R(N)$ into three nodal domains. We maintain (\ref{eq:fv}) as a description of the nodal set; recall that the quantity $f_v(0)=\lim_{x\to 0^+}f_v(x)$ is bounded away from the top and bottom boundaries by Lemma \ref{prop:coeff}. Proposition \ref{prop:V} still holds, meaning
\begin{equation}\label{eq:Voddeq}
    \left|v(x,y)-v_0(x,y)\right|\leq C\eta^{1-\epsilon}/N^{\frac{3}{2}}
\end{equation}
for $0<\eta\leq\eta_0(k,\epsilon)$. As in Section \ref{sec:awayEven}, our analysis is performed locally. We first establish a description of the nodal set away from the boundary $\p\Omega_\phi(\eta, N)$.

\begin{prop}\label{prop:Vodd1}
    Under the same conditions as in Theorem \ref{thm:VoddPt}, there exists a constant $C_0>0$, dependent only on $\Lambda_\phi$, such that for each $k\geq 5$ odd, there exists $\eta_0(k,\epsilon)>0$ such that for all $0<\eta\leq\eta_0(k,\epsilon)$, the nodal set of $v$ has the following property: If $(x_0, y_0)\in v^{-1}(0)$ is a point in $\Omega_\phi(\eta, N)$ with $\left|y_0-\frac{1}{2}\right|<\frac{1}{2}\max\{f(0), 1-f(0)\}$ and $\eta^{\frac{1}{3}(1-\epsilon)}<x_0<N-\eta^{\frac{1}{3}(1-\epsilon)}$, then there exists an open neighborhood $U_1$ containing $x_0$ and a function $h_1(x)$ such that $v^{-1}(0)\cap U_1=\{(x, h_1(x))\}$.
\begin{enumerate}
    \item If $0.1<x_0<N-0.1$, then for all $x\in U_1$,
\begin{equation}
    \left|h_1(x)-f_v(x)\right|+\left|h_1'(x)-f_v'(x)\right|\leq \frac{C_0\eta^{1-\epsilon}}{N\left|\sin\left(\frac{2\pi}{N}x\right)\right|}. \nonumber
\end{equation}
\item If instead $\eta^{\frac{1}{3}(1-\epsilon)}<\left|x_0-tN\right|\leq 0.1$, then for all $x\in U_1$,
\begin{equation}
    \left|h_1(x)-f_v(x)\right|\leq \frac{C_0\eta^{1-\epsilon}}{\left|x-tN\right|}, \quad \quad \left|h_1'(x)-f_v'(x)\right|\leq \frac{C_0\eta^{1-\epsilon}}{\left|x-tN\right|^2}. \nonumber
\end{equation}
Here $t\in\{0, 1\}$. 
\end{enumerate}
\end{prop}

\quad This result states that, over most of $\Omega$, the nodal set of $v$ behaves much like that of $v_0$ as constructed in Section \ref{sec:Hadamard}. Because $v_0^{-1}(0)$ features no internal crossing, the proof of Proposition \ref{prop:Vodd1} follows from the analysis of Section \ref{sec:awayEven}.

\begin{proof}[Proof of Proposition \ref{prop:Vodd1}.] By Theorem \ref{prop:n=0}, there is a neighborhood of size $\sim 1$ about the midline $x=\frac{N}{2}$ such that $v_0(x,y)$ is nonzero. This allows us to forgo repeating the analysis from Section \ref{sec:centerEven} for this case. By taking $\eta$ small enough, (\ref{eq:Voddeq}) implies that $v(x,y)$ is nonzero in a rectangle of side length $\sim 1$ centered at $\left(\frac{N}{2},\frac{1}{2}\right)$. Outside of this region, the proof of this proposition is identical to that of Proposition \ref{prop:horizontal}. 
\end{proof}

\quad We then focus on the nodal set near the horizontal boundary components. The following proposition applies to points near the bottom boundary, and an analogous statement holds for points near the top boundary. 

\begin{prop}\label{prop:Vodd2}
    Under the same conditions as in Theorem \ref{thm:VoddPt}, there exists a constant $C_0>0$, dependent only on $\Lambda_\phi$, such that for each $k\geq 5$ odd, there exists $\eta_0(k,\epsilon)>0$ such that for all $0<\eta\leq\eta_0(k,\epsilon)$, the nodal set of $v$ has the following property: If $(x_0, y_0)\in v^{-1}(0)$ is a point in $\Omega_\phi(\eta,N)$ with $\big|y_0\big|\leq\frac{1}{2}\min\{f(0), 1-f(0)\}$ and $\eta^{\frac{1}{3}(1-\epsilon)}<x_0<N-\eta^{\frac{1}{3}(1-\epsilon)}$, then there exists an open neighborhood $U_2$ containing $y_0$ and a function $h_2(y)$ such that $v^{-1}(0)\cap U_2=\{(h_2(y),y)\}$ and
\begin{equation}
    \left|h_2(y)-f_v^{-1}(y)\right|\leq C_0\eta^{1-\epsilon}, \quad \quad \left|h_2'(y)-\left(f_v^{-1}\right)'(y)\right|\leq C_0\eta^{1-\epsilon}|y|. \nonumber
\end{equation}
for all $y\in U_2$.
\end{prop}

\quad Because $\left(f_v^{-1}\right)'(0)=\left(f_v^{-1}\right)'(1)=0$, this result implies that the nodal set intersects the top and bottom boundaries orthogonally. 

\begin{proof}[Proof of Proposition \ref{prop:Vodd2}.] Initially, let $0.1<x<N-0.1$. Because we restrict $|y|<\frac{1}{2}\min\{f_v(0),1-f_v(0)\}$, the function $f_v$ is invertible and $f_v^{-1}(y)$ satisfies
\begin{equation}\label{eq:MVTforInvF}
    \left|\left(f_v^{-1}\right)'(y)\right|=\left|\frac{\p_yv_0(f_v^{-1}(y),y)}{\p_xv_0(f_v^{-1}(y),y)}\right|=\left|\frac{\p_y v_0(x,f_v(x))}{\p_x v_0(x,f_v(x))}\right|\leq C\left|\sin(\pi y)\right|.
\end{equation}
Largely following the proof of Proposition \ref{prop:vertical}, we begin by reestablishing (\ref{eq:LowBdforV}) as a lower bound for the eigenfunction in this region,
\begin{align}
    \left|v(x,y)\right|\geq cN^{-\frac{1}{2}}|y|\left|\sin\left(2\pi x/N\right)\right|\left|\cos(\pi y)-\cos(\pi f_v(x))\right|-C\eta^{1-\epsilon}N^{-\frac{3}{2}}|y|. \nonumber
\end{align}
We are bounded away from $x=\frac{N}{2}$, meaning that if $(x_0,y_0)\in v^{-1}(0)$, then $\left|\cos(\pi y_0)-\cos(\pi f_v(x_0))\right|\leq C\eta^{1-\epsilon}$. By (\ref{eq:MVTforInvF}), this implies that $\left|f_v^{-1}(y_0)-x_0\right|\leq C\eta^{1-\epsilon}$. Using this, we can show that $\left|\p_xv(x_0,y_0)\right|\geq cN^{-\frac{3}{2}}|y_0|$ in an identical fashion to Proposition \ref{prop:vertical}. There then exists a neighborhood of $y_0$ such that the level set $v^{-1}(0)$ can be written as the graph $\big(h_2(y),y\big)$ with
\begin{equation}
    \left|h_2'(y)-\left(f_v^{-1}\right)'(y)\right|=\left|\frac{\p_yv(h_2(y),y)}{\p_xv(h_2(y),y)}+\left(f_v^{-1}\right)'(y)\right|. \nonumber
\end{equation}
Because $\left|\p_xv(x_0,y_0)\right|\geq cN^{-\frac{3}{2}}|y_0|$, it suffices to prove
\begin{equation}\label{eq:Final4}
    \left|\p_yv(x_0,y_0)+\left(f_v^{-1}\right)'(y_0) \p_xv(x_0,y_0)\right|\leq C\eta^{1-\epsilon}N^{-\frac{3}{2}}|y_0|^2
\end{equation}
to establish the regularity estimate. Fortunately, because $\p_xE(x,0)=\p_xE(x,1)=0$ and $\left|\p_{xy}^2E\right|\leq C\eta N^{-\frac{3}{2}}$, Proposition \ref{prop:V} implies
\begin{equation}
    \left|\p_xv(x_0,y_0)-\p_xv_0(x_0,y_0)\right|\leq C\eta^{1-\epsilon}N^{-\frac{3}{2}}|y_0| \nonumber
\end{equation}
and because $x_0$ is close to $f_v^{-1}(y_0)$,
\begin{equation}
    \left|\p_xv_0(x_0,y_0)-\p_xv_0(f_v^{-1}(y_0),y_0)\right|\leq C\eta^{1-\epsilon}N^{-\frac{3}{2}}|y_0|. \nonumber
\end{equation}
In combination with the fact that $\left|\left(f_v^{-1}\right)'(y_0)\right|\leq C|y_0|$, this reduces (\ref{eq:Final4}) to proving the following:
\begin{equation}\label{eq:Final5}
    \left|\p_yv(x_0,y_0)-\p_y v_0\left(f_v^{-1}(y_0),y_0\right)\right|\leq C\eta^{1-\epsilon}N^{-\frac{3}{2}}|y_0|^2. 
\end{equation}
To establish this estimate, we use the expression in (\ref{eq:nextSection}) to describe $\p_yv(x_0,y_0)$. The sum of the first two terms can be bounded above appropriately as in the proof of Proposition \ref{prop:bottomBdry}, so we focus on the remaining difference
\begin{equation}
    \pi v_2(x_0)\left(2\cos(2\pi y_0)-\frac{\sin(2\pi y_0)}{\sin(\pi y_0)}\cos(\pi y_0)\right)-\p_yv_0\left(f_v^{-1}(y_0),y_0\right). \nonumber
\end{equation}
Because $\left|f_v^{-1}(y_0)-x_0\right|\leq C\eta^{1-\epsilon}$, Proposition \ref{prop:V} implies $\left|v_2(x_0)-\frac{2c_2}{\sqrt{N}}\sin\left(\frac{2\pi }{N}f_v^{-1}(0)\right)\right|\leq C\eta^{1-\epsilon}N^{-\frac{3}{2}}$. The function of $y_0$ in parenthesis can be bounded 
\begin{equation}
    \left|2\cos(2\pi y_0)-\frac{\sin(2\pi y_0)}{\sin(\pi y_0)}\cos(\pi y_0)\right|\leq C|y_0|^2 \nonumber
\end{equation}
meaning that the difference $\p_yv(x_0,y_0)-\p_yv_0\left(f_v^{-1}(y_0),y_0\right)$ is equal to 
\begin{equation}\label{eq:FinishIt}
    \frac{2\pi}{\sqrt{N}}c_2\sin\left(\frac{2\pi}{N}f_v^{-1}(y_0)\right)\left(2\cos(2\pi y_0)-\frac{\sin(2\pi y_0)}{\sin(\pi y_0)}\cos(\pi y_0)\right)-\p_yv_0(f_v^{-1}(y_0),y_0).
\end{equation}
up to acceptable error per (\ref{eq:Final5}). By (\ref{eq:fv}), the expression in (\ref{eq:FinishIt}) is identically zero, meaning that (\ref{eq:Final4}) holds for $0.1<x<N-0.1$. To extend the proposition to include $\eta^{\frac{1}{3}(1-\epsilon)}<\left|x-tN\right|\leq 0.1$ for $t\in\{0,1\}$, we reference the corner argument at the end of Section \ref{sec:bdry}.
\end{proof}

\quad Finally, we describe the nodal set near the left and right boundaries, making use of the same isometry $F^{-1}$ featured in Section \ref{sec:bdry}. In the rotated coordinates $(\tilde{x},\tilde{y})=F^{-1}(x,y)$ and $\tilde{x}=-\eta\tilde{\phi}(\tilde{y})$ describes the left boundary. 

\begin{prop}\label{prop:leftBdryOdd}
     Fix $(x_0,y_0)\in v^{-1}(0)$ such that $x_0<2\eta^{\frac{1}{3}(1-\epsilon)}$ and apply the transform $F^{-1}$ as described. Under the same conditions as in Theorem \ref{thm:VoddPt}, there exists a constant $C_0>0$, dependent only on $\Lambda_\phi$, such that for each $k\geq 5$ odd, there exists $\eta_0(k,\epsilon)>0$ such that for all $0<\eta\leq\eta_0(k,\epsilon)$, the nodal set of $v$ can be written locally as a graph $\tilde{y}=h_3(\tilde{x})$ with
    \begin{equation}
        \left|h_3(\tilde{x})-f_v(0)\right|\leq C_0\eta^{\frac{2}{3}(1-\epsilon)}, \quad \quad \left|h_3'(\tilde{x})\right|\leq C_0\eta^{\frac{2}{3}(1-\epsilon)}\left|\tilde{x}+\eta\tilde{\phi}(y_0)\right| \nonumber
    \end{equation}
    where $f_v(0)=\frac{1}{\pi}\arccos\big(-\frac{kc_1}{4c_2}\big)$.
\end{prop}

\quad Because $\tilde{\phi}'(y_0)=0$, Propositions \ref{prop:Vodd2} and \ref{prop:leftBdryOdd} establish orthogonality of the nodal set at the left boundary. A similar result holds at the right boundary. The proof of Proposition \ref{prop:leftBdryOdd} is identical to that of Proposition \ref{prop:leftBdry}. Just as in the case when $k$ is even, Propositions \ref{prop:Vodd1}, \ref{prop:Vodd2}, and \ref{prop:leftBdryOdd} are enough to prove the statements in Theorems \ref{thm:VoddPt} and \ref{them:VoddDer}.

%%%%%%%%%%%%%%%%%%%%%%%%%%%%%%%%%%%%%%%%%%%%%%%%%%%
\section{The Eigenfunction along the Lower Branch}\label{sec:Lower}

\quad In this section, we study the nodal set of the eigenfunction $w$ corresponding to the lower branch eigenvalue $\gamma$ for $\eta$ small and $k\geq 8$. When $\eta=0$, Theorem \ref{prop:n=0} tells us that the nodal set of $w_0$ features $(k-1)$ curves that lie outside a neighborhood of each $x_j=\frac{jN}{2k}$ for odd $j=\{1,\dots, 2k-1\}$ and separate the rectangle $R(N)$ into exactly $k$ nodal domains. To study the nodal set for $\eta>0$, we make use of the same partial Fourier series presented in Section \ref{sec:Prop1.1}. Over $R\subset \Omega$, let
\begin{equation}\label{eq:Wexp}
    w(x,y)=w_1(x)\sin(\pi y)+w_2(x)\sin(2\pi y)+\sum_{j\geq 3}w_j(x)\sin(j\pi y), \quad \quad w_j(x)=2\int_0^1w(x,y)\sin(j\pi y)dy.
\end{equation}
We again combine the higher modes into an error term
\begin{equation}
    E_w(x,y)\eqdef\sum_{j\geq 3}w_j(x)\sin(j\pi y) \nonumber
\end{equation}
and provide a detailed estimate of the first two modes. We find that the representation of $w$ in (\ref{eq:Wexp}) closely resembles the expression for $w_0$ as constructed in Section \ref{sec:Hadamard}, as made evident in the following statement.

\begin{prop}\label{prop:W}
Let the conditions in Theorem \ref{thm:WPt} hold, and let $\textbf{c}=(c_1,c_2)$ be as constructed in Proposition (\ref{prop:coeff}). There exists a constant $C_0>0$, dependent only on $\Lambda_\phi$, such that for all $k\geq 8$, there exists $\eta_0(k,\epsilon)>0$ such that for all $0<\eta\leq \eta_0(k,\epsilon)$, the following holds:
\begin{enumerate}
    \item The first Fourier mode $w_1(x)$ satisfies the estimates
    \begin{equation}
        \left|w_1^{(\ell)}(x)+\frac{2c_2}{\sqrt{N}}\frac{d^\ell}{dx^\ell}\left(\sin\left(\frac{k\pi}{N}x\right)\right)\right|\leq C_0\eta^{1-\epsilon}/N^{\frac{3}{2}} \quad \textrm{for} \quad 0\leq \ell\leq 3. \nonumber
    \end{equation}
    \item The second Fourier mode $w_2(x)$ satisfies the estimates
    \begin{equation}
        \left|w_2^{(\ell)}(x)-\frac{2c_1}{\sqrt{N}}\frac{d^\ell}{dx^\ell}\left(\sin\left(\frac{2\pi}{N}x\right)\right)\right|\leq C_0\eta^{1-\epsilon}/N^{\frac{3}{2}} \quad \textrm{for} \quad 0\leq \ell\leq 3. \nonumber
    \end{equation}
    \item The error $E_w(x,y)$ satisfies the estimates
    \begin{equation}
        \left|E_w(x,y)\right|+\left|\nabla E_w(x,y)\right|\leq C_0\eta^{1-\epsilon}/N^{\frac{3}{2}} \quad \textrm{for} \quad (x,y)\in\Omega. \nonumber
    \end{equation}
\end{enumerate}
\end{prop}

\quad This result is comparable to Proposition \ref{prop:V}, describing how the Fourier modes in (\ref{eq:Wexp}) behave. As in Sections \ref{sec:centerEven} - \ref{sec:UpperOdd}, we can extend the estimates in Proposition \ref{prop:W} to estimates on the nodal set of $w$ in $\Omega$. A result analogous to Proposition \ref{prop:VJ} for $w_j(0)$ also holds but is not necessary for our proofs because the nodal set of $w_0$ does not feature a crossing. 

\begin{proof}[Proof of Proposition \ref{prop:W}.] This proof follows the analysis of Section \ref{sec:Prop1.1}, but we highlight one observation. By domain monotonicity for Dirichlet eigenvalues, $\gamma\leq \lambda_{k,1}$ because $[0,N]\times[0,1]\subset\Omega$. Similarly, $\gamma$ is bounded below by the eigenvalue of $[-\eta,N]\times[0,1]$ at the corresponding level in the spectrum. However, there is a crossing that occurs between the branches stemming from $\lambda_{k,1}$ and $\lambda_{1,2}$ (\ref{eq:eigenpairs}) as $\eta$ increases.
\begin{figure}[H]
    \centering
    \includegraphics[scale=0.32]{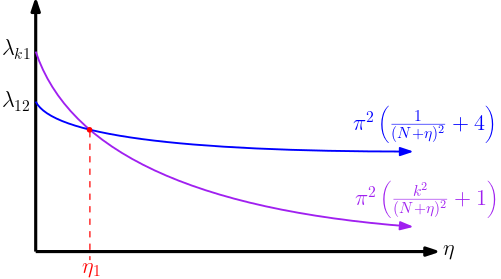} \hspace{1.5cm}
    \includegraphics[scale=0.32]{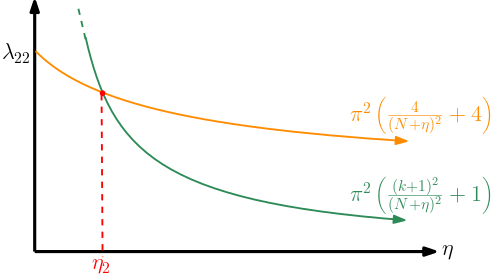}
    \caption{Eigenvalues of the Rectangle $[-\eta, N]\times[0,1]$.}
\end{figure}
This crossing occurs at $\eta_1=\sqrt{N^2+1}-N$, so we take $\eta_0<\eta_1$ to avoid this crossing. The branch stemming from $\lambda_{2,2}$ also features a crossing from a higher eigenvalue, but it does so for a value $\eta_2$ that can be bounded below by a constant. In our analysis along the $\gamma$ branch, we thus have another justification as to why $\eta_0(k,\epsilon)$ must be limited by the value $k$. Items $(i)$ and $(ii)$ of Proposition \ref{prop:W} result from repeating the arguments in Section \ref{sec:Prop1.1}, although we are left with an error estimate $\big|E_w\big|\leq C\eta/N^{\frac{1}{2}}$. To get a matching denominator, we sacrifice slightly in $\eta$. 
\end{proof}

\quad Note that when $k$ is even, $w_0\left(\frac{N}{2},y\right)=0$, so the function $f_w(x)$ featured in (\ref{eq:fw}) only describes the nodal set away from the midline. Meanwhile if $k$ is odd, Theorem \ref{prop:n=0} implies that the nodal set of $w_0$ does not intersect a neighborhood of $x=\frac{N}{2}$. First we study the nodal set of $w$ for all $k\geq 8$ outside of this neighborhood. Then we focus on the case when $k$ is even and $x$ is close to $\frac{N}{2}$.

\begin{prop}\label{prop:verticalW}
    Under the same conditions as in Theorem \ref{thm:WPt}, there exists a constant $C_0>0$, dependent only on $\Lambda_\phi$, such that for each $k\geq 8$, there exists $\eta_0(k,\epsilon)>0$ such that for all $0<\eta\leq\eta_0(k,\epsilon)$, the nodal set of $w$ has the following properties: 
    \begin{enumerate}
        \item If $(x_0,y_0)\in w^{-1}(0)$ is a point in $\Omega_\phi(\eta, N)$ with $0.1\leq y_0\leq0.9$ and $x_0\in\left(\frac{jN}{2k},\frac{N}{k}+\frac{jN}{2k}\right)$ for any odd $j\in\{1,\dots, 2k-3\}$ with $\frac{N}{2}\notin\left(\frac{jN}{2k},\frac{N}{k}+\frac{jN}{2k}\right)$, then there exists an open neighborhood $U_1$ containing $x_0$ and a function $h_1(x)$ such that $w^{-1}(0)\cap U_1=\{(x, h_1(x))\}$ and
    \begin{equation}
        \left|h_1(x)-f_w(x)\right|+\left|h_1'(x)-f_w'(x)\right|\leq C_0\eta^{1-\epsilon} \nonumber
    \end{equation}
    for all $x\in U_1$.
    \item Let $t\in\{0,1\}$. If instead $(x_0,y_0)\in w^{-1}(0)$ is a point in $\Omega_\phi(\eta, N)$ with $|y_0-t|<0.1$ and $x_0\in\left(\frac{jN}{2k},\frac{N}{k}+\frac{jN}{2k}\right)$ for any odd $j\in\{1,\dots, 2k-3\}$ with $\frac{N}{2}\notin\left(\frac{jN}{2k},\frac{N}{k}+\frac{jN}{2k}\right)$, then there exists an open neighborhood $U_2$ containing $y_0$ and a function $h_2(y)$ such that $w^{-1}(0)\cap U_2=\{(h_2(y),y)\}$ and
    \begin{equation}
        \left|h_2(y)-f_w^{-1}(y)\right|\leq C_0\eta^{1-\epsilon}, \quad \left|h_2'(y)-\left(f_w^{-1}\right)'(y)\right|\leq C_0\eta^{1-\epsilon}|y-t|. \nonumber
    \end{equation}
    for all $y\in U_2$. 
    \end{enumerate}
    The eigenfunction does not vanish for $x\leq\frac{N}{2k}$ or $x\geq N-\frac{N}{2k}$.
\end{prop}

\quad By Theorem \ref{prop:n=0}, Proposition \ref{prop:verticalW} completely describes the nodal set of $w$ for odd $k$. However, if $k$ is even, then additional analysis is needed for the region $\left(\frac{N}{2}-\frac{N}{2k}, \frac{N}{2}+\frac{N}{2k}\right)$ containing $\frac{N}{2}$. 

\begin{prop}\label{prop:midW}
    Under the same conditions as in Theorem \ref{thm:WPt}, there exists a constant $C_0>0$, dependent only on $\Lambda_\phi$, such that for each $k\geq 8$ even, there exists $\eta_0(k,\epsilon)>0$ such that for all $0<\eta\leq\eta_0(k,\epsilon)$, the nodal set of $w$ has the following property:
    \begin{enumerate}
        \item If $(x_0,y_0)\in w^{-1}(0)$ is a point in $\Omega_\phi(\eta, N)$ with $x_0\in\left(\frac{N}{2}-\frac{N}{2k}, \frac{N}{2}+\frac{N}{2k}\right)$, then there exists an open neighborhood $U_3$ containing $y_0$ and a function $h_3(y)$ such that $w^{-1}(0)\cap U_3=\{(h_3(y),y)\}$ and
    \begin{equation}
        \left|h_3(y)-\frac{N}{2}\right|\leq C_0\eta^{1-\epsilon}, \quad \quad \left|h_3'(y)\right|\leq C_0\eta^{1-\epsilon} \nonumber
    \end{equation}
    for all $y\in U_3$.
    \item If $|y_0-t|<0.1$, then we can improve the derivative estimate to 
    \begin{equation}
        \left|h_3'(y)\right|\leq C_0\eta^{1-\epsilon}\left|y-t\right| \nonumber
    \end{equation}
    for $t\in\{0,1\}$.
    \end{enumerate}
\end{prop}

\quad Proposition \ref{prop:verticalW} follows from the proofs of Propositions \ref{prop:horizontal} and \ref{prop:Vodd2}. In the regions close to the left and right boundaries, we reference the corner argument from Section \ref{sec:bdry}. Meanwhile, Lemma \ref{prop:coeff} implies that
\begin{equation}
    \bigg|\frac{c_2}{2c_1}\frac{\sin(k\pi x/N)}{\sin(2\pi x/N)}\bigg|\geq 2 \nonumber
\end{equation}
for all $k\geq 8$ even and $x\in\left(\frac{N}{2}-\frac{N}{2k}, \frac{N}{2}+\frac{N}{2k}\right)$. In combination with (\ref{eq:key}), this allows us to establish Proposition \ref{prop:midW} by following the proofs of Propositions \ref{prop:vertical} and \ref{prop:bottomBdry}. Theorems \ref{thm:WPt} and \ref{thm:WDer} follow from Propositions \ref{prop:verticalW} and \ref{prop:midW}.

%%%%%%%%%%%%%%%%%%%%%%%%%%%%%%%%%%%%%%%%%%%%%%%%%%%

\bibliographystyle{plain}
\bibliography{NodalBib}
 
\end{document}